\documentclass[12pt]{amsart}
\usepackage{hyperref}
\hypersetup{
    colorlinks=true,
    linkcolor=blue,
    citecolor=blue,
    urlcolor=blue,
    pdfauthor={Diego Corro and Jesús Nuñez Zimbrón and Jaime Santos-Rodríguez},
    pdftitle={Codimension 1 RCD-spaces}
}

\usepackage{verbatim}
\usepackage[numbers]{natbib}
\usepackage{url}
\usepackage{graphicx}
\usepackage{amsfonts, amssymb, amsmath, amsthm}
\usepackage{mathtools}
\usepackage{multicol}
\usepackage{tikz}
\usepackage{tikz-cd}
\usepackage{tikz-3dplot}
\usepackage{cancel}
\usepackage{changepage}
\usepackage{csquotes}
\usepackage{adjustbox}
\usetikzlibrary{matrix}
\usepackage{enumerate}
\usepackage[normalem]{ulem}
\usepackage{comment}

\allowdisplaybreaks

\newtheorem{mtheorem}{Theorem}

\newtheorem{theorem}{Theorem}[section]
\newtheorem{lemma}[theorem]{Lemma}
\newtheorem{prop}[theorem]{Proposition}
\newtheorem{cor}[theorem]{Corollary}

\newtheorem{question}[theorem]{Question}

\theoremstyle{definition}
\newtheorem{definition}[theorem]{Definition}

\newtheorem{rmk}[theorem]{Remark}

 \newtheoremstyle{TheoremNum}
        {\topsep}{\topsep}              
        {\itshape}                      
        {}                              
        {\bfseries}                     
        {.}                             
        { }                             
        {\thmname{#1}\thmnote{ \bfseries #3}}
    \theoremstyle{TheoremNum}
    \newtheorem{duplicate}{Theorem}
    
    \newtheorem{duplicatePROP}{Proposition}

\newcommand{\Ch}{\mathrm{Ch}}
\newcommand{\mycomment}[1]{%
}
\newcommand{\dis}{d}
\newcommand{\dime}{\mathrm{dim}}
\newcommand{\dimess}{\mathrm{dim}_{\mathrm{ess}}}

\newcommand{\E}{\mathbb{E}}

\newcommand{\Hauss}{\mathcal{H}}
\newcommand{\Iso}{\mathrm{Isom}}
\newcommand{\kernel}{Ker}
\newcommand{\Lip}{\mathrm{Lip}}

\newcommand{\BE}{\mathsf{BE}}
\newcommand{\CD}{\mathsf{CD}}
\newcommand{\RCD}{\mathsf{RCD}}

\newcommand{\Sob}[2][1]{W^{#1,#2}}
\newcommand{\Tan}{\mathrm{Tan}}

\DeclareMathOperator{\supp}{supp}
\DeclareMathOperator{\curv}{curv}

\DeclareMathOperator{\diam}{diam}

\DeclareMathOperator{\Tub}{Tub}

\DeclareMathOperator{\Ric}{Ric}

\DeclareMathOperator{\cl}{Cl}

\DeclareMathOperator{\Geo}{Geo}
\DeclareMathOperator{\Opt}{Opt}
\newcommand{\m}{\mathfrak{m}}
\newcommand{\N}{\mathbb{N}} 
\newcommand{\R}{\mathbb{R}} 
\newcommand{\Z}{\mathbb{Z}} 
\newcommand{\C}{\mathbb{C}} 

\newcommand{\Sp}{\mathbf{S}} 
\newcommand{\SO}{\mathrm{SO}}

\usepackage{lineno}

\usepackage[reftex]{theoremref}

\usepackage[colorinlistoftodos, textwidth=3cm]{todonotes}
\newcounter{jncomment}

\newcounter{dccomment}

\newcounter{jscomment}

\begin{document}
\title{Cohomogeneity one $\RCD$-spaces}

\author[D. CORRO]{Diego Corro$^{*}$}
\address[D. CORRO]{Fakultät für Mathematik\\ 
Karlsruher Institut für Technologie\\
Englerstr. 2\\
76131 Karlsruhe,
Deutschland.
}
\curraddr[D. CORRO]{School of Mathematics, Cardiff University, Cardiff, UK.}
\email{diego.corro.math@gmail.com}
\thanks{$^{*}$Supported by a UKRI Future Leaders Fellowship [grant number MR/W01176X/1; PI: J Harvey], by DFG (281869850, RTG 2229 ``Asymptotic Invariants and Limits of Groups and Spaces'' and by DFG-Eigenstelle Fellowship CO 2359/1-1}

\author[J. NUÑEZ-ZIMBRON]{Jesús Núñez-Zimbrón$^{\ddagger}$}
\address[J. NUÑEZ-ZIMBRON]{Facultad de Ciencias\\
Universidad Nacional Autónoma de México (UNAM)\\
Circuito de la Investigación Científica, C.U.\\ 
Coyoacán, 04510 Ciudad de Mé\-xico\\
México.}
\email{nunez-zimbron@ciencias.unam.mx}
\thanks{$^{\ddagger}$Supported by PAPIIT-UNAM project IN101322 and CONAHCyT-SNII 323477}

\author[J. SANTOS-RODRIGUEZ]{Jaime Santos-Rodríguez$^\dagger$}
\address[J. SANTOS-RODRIGUEZ]{Departamento de Matem\'atica Aplicada, Universidad Polit\'e\-cnica de Madrid, Spain.
}
\email{jaime.santos@upm.es}
\thanks{$^{\dagger}$ Supported in part by a Margarita Salas Fellowship of the Universidad Aut\'onoma de Madrid CA1/RSUE/2021--00625, and by research grants  MTM2017-‐85934-‐C3-‐2-‐P, PID2021--124195NB--C32 from the Ministerio de Econom\'ia y Competitividad de Espa\~{na} (MINECO)}


\setlength{\overfullrule}{5pt}

\begin{abstract}
We study $\RCD$-spaces $(X,\dis,\m)$ with group actions by isometries preserving the reference measure $\m$ and whose orbit space has dimension one, i.e. cohomogeneity one actions. To this end we prove a Slice Theorem asserting that when $X$ is non-collapsed the slices are homeomorphic to metric cones over homogeneous spaces with $\Ric \geq 0$. As a consequence we obtain complete topological structural results (also in the collapsed case) and a regular  orbit representation theorem. Conversely, we show how to construct new $\RCD$-spaces from a cohomogeneity one group diagram, giving  a complete description of $\RCD$-spaces of cohomogeneity one. As an application of these results we obtain the classification of cohomogeneity one, non-collapsed $\RCD$-spaces of essential dimension at most $4$.  
\end{abstract}

\maketitle

\tableofcontents

  \makeatletter
  \providecommand\@dotsep{5}
  \makeatother

\section{Introduction}

The theory of $\RCD$-spaces  has emerged from, and is still  largely motivated, by the extensive and deep work of Cheeger and Colding on Gromov-Hausdorff limit spaces of sequences consisting of Riemannian manifolds with Ricci curvature bounded from below \cite{CheegerColding1997, CheegerColding2000a, CheegerColding2000b}. Cheeger and Colding themselves asked whether a synthetic treatment of Ricci curvature bounds was possible. The most natural answer is the now well-developed theory of spaces with the curvature-dimension condition, initiated by Lott, Sturm, and Villani \cite{LottVillani2009, SturmI2006, SturmII2006} and refined to exclude Finslerian spaces by Ambrosio, Gigli, and Savaré \cite{AmbrosioGigliSavare2014, Gigli2015}. Thus, the study of $\RCD$-spaces carries intrinsic theoretical value; however, it has also produced several non-trivial applications to the study of Riemannian manifolds which are beyond the scope of the classical standpoint, i.e. only using smooth tools (see \cite[Section 7]{Gigli2023} for some examples of these results).  

Therefore it is compelling to understand the geometries and topologies of non-smooth spaces with curvature bounded below. As these spaces admit topological and metric singularities, which may have a complicated arrangement, it is necessary to look for systematic ways to study them. In particular, the topological structure of $\RCD$-spaces is complicated and, as of now, there are only a few results describing it (see for example \cite{MondinoWei2019, SantosZamora2023, Wang2024} and references therein). 

An approach to produce such a systematic method for the study of $\RCD$-spaces is the so called \textit{extended symmetry program}. Grove originally proposed  in \cite{Grove2002} that, in the search of new examples or obstructions to the existence of Riemannian metrics with positive sectional curvature, we should first study the manifolds with a high degree of symmetry. This approach has been extended to the study of Riemannian manifolds with  other lower curvature bounds, and has also been successfully extrapolated to the study of Alexandrov spaces, i.e. length spaces with a synthetic lower sectional curvature bound (see \cite{Galaz-GarciaSearle2011, HarveySearle2017} and references therein). 

Since the group of measure preserving isometries of an $\RCD$-space is a Lie group \cite{GuijarroSantos2019, Sosa2018}, we can try to follow the symmetry program philosophy for $\RCD$-spaces. In this article we do this by considering \textit{cohomogeneity one $\RCD$-spaces}, that is $\RCD$-spaces admitting an isometric and measure-preserving action of a compact Lie group in such a way that the orbit space is of Hausdorff dimension one. Our main results are the following.

\begin{mtheorem}\th\label{MC: Homeomorphism rigidity}
Let $(X,\dis,\m)$ be an $\RCD(K,N)$-space. Let $G$ be a compact Lie group acting on $X$ by measure preserving isometries and cohomogeneity one. Then the following hold:
\begin{enumerate}[(a)]
\item\label{MC: Homeomorphism rigidity close interval} When $X^\ast$ is homeomorphic to $[0,1]$, then $X$ is equivariantly homeomorphic to the union of two cone bundles over the singular orbits, glued along their boundary.

\item\label{MC: homeomorphism rigidity ray} When $X^\ast$ is homeomorphic to $[0,\infty)$, then $X$ is homeomorphic to a cone bundle.

\item\label{MC: homeomorphism rigidity Circle} When $X^\ast$ is homeomorphic to $\Sp^1$, then $X$ is the total space of a fiber bundle with homogeneous fiber $G/H$ and structure group $N_G(H)/H$. 

\item\label{MC: homeomorphism rigidity line} When $X^\ast$ is homeomorphic to $\R$, then $X$ is homeomorphic to $\R\times G/H$.
\end{enumerate}

\noindent Moreover the cone fibers in items $(a)$ and $(b)$ are cones over homogeneous spaces. 

In the case when $X$ is non-collapsed then  the cone fibers admit  metric cone structures over $\RCD(N-k_{\pm}-2,N-k_\pm-1)$-spaces, where $k_\pm = \dim(G(x_\pm))$.
\end{mtheorem}

This result gives the full description of $\RCD$-spaces of cohomogeneity $1$, completing the rigidity results of  Galaz-García, Kell, Mondino and Sosa in \cite[Corollary 1.4]{GalazGarciaKellMondinoSosa2018}. Also, it will be clear from the proof that if $X_1$ and $X_2$ have the same orbit space and orbit types then they are equivariantly homeomorphic. We also show that, to obtain the conclusions about the homeomorphism type in \th\ref{MC: Homeomorphism rigidity}, we only need an essentially non-branching metric-measure structure (see Section~\ref{S: Slice THeorem RCD} and \th\ref{T: Topological classification geodesic coho 1 spaces.}). 

The topological description depends on having a geodesic structure, and the geometry of the quotient space, as stated in \th\ref{T: Topological classification geodesic coho 1 spaces.}.

Moreover, in the case of \th\ref{MC: Homeomorphism rigidity}~\eqref{MC: Homeomorphism rigidity close interval} we get three subgroups $H$, $K_-$, and $K_+$, that, up to conjugation, are all the possible isotropy groups of the action. By \th\ref{T: Principal isotropy is subgroup of singular isotropy}, it follows that $H<K_\pm$. Thus, we have a tuple $(G,H,K_-,K_+)$ associated to the action.   We show that the converse is true: each \textit{cohomogeneity one group diagram} $(G,H,K_-,K_+)$, that is a collection of subgroups where $H$ is a Lie subgroup of $K_{\pm}$, and in turn $K_{\pm}$ are Lie subgroups of $G$ such that $K_{\pm}/H$ are homogeneous spaces with positive Ricci curvature if they have dimension $\geq 2$, and in the case of dimension $1$ they have bounded diameter, determines a cohomogeneity one $\RCD$-space. 

\begin{mtheorem}
\th\label{thm.rcd space from group diagram}
Let $G$ be a compact Lie group and $K_+$, $K_-$, $H$ Lie subgroups of $G$ such that $(G,H,K_+,K_-)$ is a cohomogeneity one group diagram. Then, there exists an $\RCD$-space $(X,\dis,\m)$ admitting a cohomogeneity one action of $G$ by measure preserving isometries such that the associated group diagram is $(G,H,K_+,K_-)$.
\end{mtheorem}

These two results give a complete description of $\RCD$-spaces of cohomogeneity $1$, which is analogous to the descriptions of topological and Riemannian manifolds of cohomogeneity $1$, and Alexandrov spaces of cohomogeneity $1$ (see \cite{AlekseevskyAlekseevsky1993}, \cite{Mostert1957a}, \cite{Mostert1957b}, \cite{Neumann1967}, \cite{GalazGarciaZarei2018}, and \cite{Galaz-GarciaSearle2011}).

We also remark that the condition on $G_\pm/H$ having a metric of positive Ricci curvature is equivalent to $G_\pm/H$ being compact and with finite fundamental group \cite[Theorem 1]{Berestovskii1995}, in the case of dimensions at least $2$.

A crucial ingredient in the proof of the classification \th\ref{MC: Homeomorphism rigidity} is the Slice Theorem. This is used in general in the Grove symmetry program as a recognition tool for the local topology around giving an explicit topological model of a tubular neighborhood of a fixed orbit. 
 
All known geometric proofs of the Slice Theorem, be it in the smooth or Alexandrov setting, require a notion of ``space of directions'' and exponential map (see \cite{HarveySearle2017}). In the $\RCD$-setting, while a notion of space of directions can be defined a.e. (see \cite{HanMondino2017}), the infinitesimal-to-local relation (in the spirit of Perelman's conical neighborhood theorem) is a priori not known to hold. That is, it is not yet fully understood how the mGH-tangent space and $L^2$-tangent module (which are equivalent by work of Gigli and Pasqualetto \cite{GigliPasqualetto2022}) determine the space locally  (see however the recent work of Honda and Peng \cite{HondaPeng2023} in this direction). Instead, we work in the setting of optimal transport theory and exploiting the space's low-dimensionality to circumvent this obstacle.

\begin{mtheorem}\th\label{MT: Slice Theorem Cohomogenity one}
Let $(X,\dis,\m)$ be an essentially non-branching space (see \th\ref{D: essentiallynonbranch} ) and $G$ a compact Lie group acting on $X$ by  measure-preserving isometries such that $X^{\ast}$ is isometric to $[-1,1]$, $[0,\infty)$, $\R$, or $\Sp^1$. Then for any $x_0\in X$, there exists $\delta_0>0$ such that, for any $0<\delta<\delta_0$ we have that
\[
  S_{x_0}:= \lbrace y \in X \,|\, \dis(x_0,y)=\dis(G(x_0), y)  \rbrace\cap B_{\delta}(x_0)
\]
is a slice through $x_0$ (see \th\ref{D: Slice}).
\end{mtheorem}

It is worth pointing out that we do not assume the $\RCD$-condition (or any other bounded curvature condition) in this result. Thus, it does not depend on any specific curvature bound. It depends only on the dimension of the orbit space and properties of the metric measure structure. 

Once this description of the slices at each point has been established we move on to show that, in general, the slices admit a non-negatively curved $\RCD$-structure (that is, each slice is homeomorphic to an $\RCD(0,M)$-space, for the appropriate dimension bound $M$ as stated below). When considering the case when $X$  is non-collapsed, we further show that slices are homeomorphic to metric cones over homogeneous spaces with adequate $\RCD$-structures.

\begin{mtheorem}\th\label{MT: Geometry of the slice}
Let $(X,\dis,\Hauss^N)$ be a non-collapsed $\RCD(K,N)$-space with $N>2$. Assume that $G$  is a compact Lie group acting effectively by measure preserving isometries on $X$, such that $(X^\ast,\dis^\ast,(\Hauss^N)^\ast)$ has essential dimension equal to $1$. Then the following hold:
\begin{enumerate}[(i)]
    \item For any $x_0\in X$, the slice $S_{x_0}$ admits an $\RCD(0,N-k)$ structure, where $k= \dim(G(x_0))$.
    \item  Moreover, $S_{x_0}$ is homeomorphic to a metric cone over a homogeneous smooth Riemannian manifold with Ricci curvature greater or equal to $N-k-2$. 
\end{enumerate}
\end{mtheorem}

We point out that by the work of Palais \cite{Palais1960} the existence of slices is well understood for compact group actions by isometries. Nonetheless, in the Riemannian and Alexandrov setting, the metric structure determines a canonical choice of a slice, which inherits geometric properties of the global space. In this way, for the local topological recognition, we may apply a reduction on the dimension being considered while still considering spaces in the appropriate setting. This key point is beyond the conclusion of Palais' Slice Theorem \cite{Palais1960}, and it is the point being proven in \th\ref{MT: Slice Theorem Cohomogenity one} and \th\ref{MT: Geometry of the slice}. Thus, by \th\ref{MT: Geometry of the slice}, \th\ref{MT: Slice Theorem Cohomogenity one} gives a positive answer to \cite[Problem 1]{BIRSreport2022} for non-collapsed $\RCD$-spaces of cohomogeneity one. 

The slice in \th\ref{MT: Geometry of the slice} is the fiber of the cone bundles in \th\ref{MC: Homeomorphism rigidity}. In the non-collapsed case, the homogeneous spaces are determined by the slice.

Finally, we use our previous results to obtain a topological classification of non-collapsed $\RCD$-spaces of essential dimensions at most $4$, in the spirit of the corresponding classifications in the smooth and Alexandrov cases. In fact, we obtain that all the $\RCD$-spaces considered with these restrictions are homeomorphic to Alexandrov spaces. This solves a particular case (that of spaces admitting cohomogeneity one actions) of a conjecture by Mondino asserting that $\RCD$-spaces with essential dimension $3$ are homeomorphic to orbifolds (see Question 9 in Section 2 of \cite{BIRSreport2022}). 

\begin{mtheorem}
\th\label{MT: Low-dim-classification}
Let $G$ be a compact Lie group acting almost effectively by measure preserving isometries and cohomogeneity one on a closed non-collapsed $\RCD(K,N)$-space $(X,\dis,\Hauss^N)$ with $N\leq4$. Then $X$ is homeomorphic to an Alexandrov space. 
\end{mtheorem}

Combining  the characterization  in \cite{Berestovskii1995} of homogeneous spaces admitting positive Ricci curvature with the cone construction in \cite{Ketterer2013}, it is easy to construct $\RCD(K,N)$-spaces of cohomogeneity one for some $K\geq 0$. This construction allows us to show that the dimension bound in \th\ref{MT: Low-dim-classification} is optimal (see \th\ref{T: Suspensions are RCD spaces of coho 1}). Namely, the following example shows that in dimension $5$ there are non-collapsed $\RCD$-spaces of cohomogeneity one which are not Alexandrov spaces of cohomogeneity one.

\begin{mtheorem}\th\label{MT: Grassmannians as examples}
The suspension $\mathrm{Susp}(\Sp^2\times\Sp^2)$ of $\Sp^2\times\Sp^2$ admits a non-collapsed $\RCD(K,5)$-structure for any $K\geq 0$ such that the suspension of the $\SO(3)\times\SO(3)$-action on $\Sp^2\times\Sp^2$ is by measure preserving isometries and by cohomogeneity one, with group diagram  $(\SO(3)\times\SO(3),\SO(2)\times\SO(2),\SO(3)\times\SO(3),\SO(3)\times\SO(3))$. Moreover, $\mathrm{Susp}(\Sp^2\times\Sp^2)$ with this action of $\SO(3)\times\SO(3)$ is not an Alexandrov space of cohomogeneity one, i.e. it does not admit a metric making it an Alexandrov space and such that the action of $\SO(3)\times\SO(3)$ is by isometries.
\end{mtheorem}

We observe that the spaces in  \th\ref{MC: Homeomorphism rigidity}~\eqref{MC: homeomorphism rigidity Circle} and \eqref{MC: homeomorphism rigidity line}  admit an $\RCD(0,N)$-structure of cohomogeneity one given by  a smooth Riemannian metric on the smooth representatives of the homeomorphism types. Moreover, due to Gigli's Splitting Theorem \cite{Gigli2014}, the spaces in \th\ref{MC: Homeomorphism rigidity}~\eqref{MC: homeomorphism rigidity line} do not admit an $\RCD(K,N)$-structure with $K>0$. For spaces in \th\ref{MC: Homeomorphism rigidity}~\eqref{MC: homeomorphism rigidity Circle}, since they have infinite fundamental group, by Myers Theorem (see \cite[Theorem 3.5]{MondinoWei2019} together with \cite[Main Theorem]{Wang2024}) it follows that they also do not admit an $\RCD(K,N)$-structure with $K>0$. For spaces in \th\ref{MC: Homeomorphism rigidity}~\eqref{MC: homeomorphism rigidity ray}, we give in Theorem~\ref{MT: coho one ray with conditions is RCD 0,N} sufficient conditions to guarantee that they  also admit an $\RCD(0,N)$-structure of cohomogeneity one. Thus it natural to ask if this also holds for the remaining cases.

\begin{question}
Let $X$ be an $\RCD$-space as in \th\ref{MC: Homeomorphism rigidity}~\eqref{MC: Homeomorphism rigidity close interval} or \eqref{MC: homeomorphism rigidity ray}. Does it admit a (possibly different) $\RCD(0,N)$-structure of cohomogeneity one?
\end{question}

This question was answered positively by Grove and Ziller in \cite[Theorem B]{GroveZiller2002} when $X$ is a smooth Riemannian manifold. Moreover, in \cite[Theorem A]{GroveZiller2002}, they showed that a smooth compact Riemannian manifold with an action of cohomogeneity one admits a Riemannian metric of positive Ricci curvature if and only if it has finite fundamental group. Due to the fact that Myers Theorem holds for $\RCD$-spaces, we ask if the cardinality of the fundamental group is the unique obstruction to ``synthetic positive Ricci curvature'' in the case of compact $\RCD$-spaces of cohomogeneity one.

\begin{question}
Let $X$ be an $\RCD$-space as in \th\ref{MC: Homeomorphism rigidity}~\eqref{MC: Homeomorphism rigidity close interval}. If $X$ has finite fundamental group, does it admit an $\RCD(K,N)$-structure with $K>0$?
\end{question}

We present a simple presentation of the fundamental group of an $\RCD$-space satisfying \th\ref{MC: Homeomorphism rigidity}~\eqref{MC: Homeomorphism rigidity close interval} in \th\ref{T: fundamental group}.\\

The organization of the article is the following. In Section~\ref{SEC:PRELIMINARIES} we collect the definitions and results from the theory of non-smooth differential geometry and transformation groups that we will need in the sequel. In Section~\ref{S: Slice THeorem RCD} we prove \th\ref{MT: Slice Theorem Cohomogenity one}. We are also able to prove that the set $S_x$ is a slice without the cohomogeneity one restriction, albeit in the case that the orbit space happens to be isometric to an Alexandrov space, which is of independent interest. In Section~\ref{S: GEOMETRY-OF-THE-SLICE} we prove \th\ref{MT: Geometry of the slice} obtaining along the way a version of the orbit representation theorem. We also establish the topological rigidity result contained in \th\ref{MC: Homeomorphism rigidity}. We continue in Section~\ref{S: Gluing of RCD-spaces} proving that one can construct $\RCD$-spaces from any given cohomogeneity one group diagram, as expressed in \th\ref{thm.rcd space from group diagram}. Finally, we address the low-dimensional classification of \th\ref{MT: Low-dim-classification} in Section~\ref{S:Low-DIM-Classification}.\\

\subsection*{Acknowledgments} 
The authors wish to thank Fernando Galaz-García, Nicola Gigli, John Harvey, Alexander Lytchak, Marco Radeschi, and Masoumeh Zarei for very useful discussions and communications and for commenting on a first draft of this manuscript. We thank Elia Bruè for useful discussions about the infinitesimal action which improved Section~\ref{S: GEOMETRY-OF-THE-SLICE}, and for pointing us to the work of Pan and Wei \cite{PanWei2022}. We are also grateful to Jiayin Pan and Dimitri Navarro for useful discussions on results in Section~\ref{S: GEOMETRY-OF-THE-SLICE}. We thank Christoph B\"{o}hm for questions about Section~\ref{S: Gluing of RCD-spaces}.

\section{Preliminaries}
\label{SEC:PRELIMINARIES}

\subsection{Calculus on metric measure spaces}

Here we introduce the differential structure that we need in order to define the appropriate curvature notions that  we use later. For a more detailed exposition, the reader may consult the book by Gigli and Pasqualetto \cite{GigliPasqualetto2020}. 

By a metric measure space we  mean a triple $(X,\dis,\m)$, where $(X,\dis)$ is a complete and separable metric space, and the reference measure $\m$ is a non-negative Borel measure on $X$ which is finite on balls.

Let $f \colon X \rightarrow \R$ be a Lipschitz function. We define its \emph{local Lipschitz constant} $\Lip f \colon X \rightarrow \R$ as:
\[
\Lip f (x) := \limsup_{y \rightarrow x}\frac{|f(x)-f(y)|}{\dis(x,y)}. 
\]
We denote the space of Lipschitz functions by $\mathrm{LIP}(X)$.

\begin{definition}[Cheeger energy]
Let $(X,\dis,\m)$ be a metric measure space. Given a  function $f \in L^2(\m)$, we define its \emph{Cheeger energy} as: 
\[  
\mathrm{Ch}(f) :=\inf \left\lbrace \liminf_{n \rightarrow \infty} \frac{1}{2}\int |\Lip f_n|^2 d \m \,|\, f_n \in \mathrm{LIP}(X),  f_n \rightarrow f  \text{ in } L^2 (\m)   \right\rbrace. 
\]
\end{definition}

\begin{definition}[Sobolev space]
Let $(X,\dis,\m)$ be a metric measure space. We define the \emph{Sobolev space} $\Sob{2}(X):= \lbrace f \in L^2(\m)\,|\, \Ch(f) < \infty \rbrace$, equipped with the norm  $\|f\|^2_{\Sob{2}} := \|f \|^2_2+ 2\Ch(f)$. In the case that this norm comes from an inner product, we say that $(X,\dis,\m)$ is \emph{infinitesimally Hilbertian}.
\end{definition}

Given a function $f \in\Sob{2}(X)$, there exists a distinguished function $|\nabla f| \in L^2(\m)$ called the \emph{minimal weak upper gradient} of $f$, which satisfies $\Ch(f) = \frac{1}{2}\int |\nabla f|^2 d\m$ (see Proposition $2.2.8$ in \cite{GigliPasqualetto2020}).

\begin{rmk}
It should be noted that we can define Sobolev spaces,  equivalently, via test plans as in Chapter $2$ of \cite{GigliPasqualetto2020}. However, for our purpose we find more convenient to make use of  approximations by Lipschitz functions.  
\end{rmk}

\begin{definition}[Pointwise inner product]
Given $f,g \in \Sob{2}(X)$ we define 
\[
 \langle\nabla f,\nabla g \rangle := \frac{1}{4}\left(|\nabla (f+ g)|^2-|\nabla (f- g)|^2 \right) \in L^1(\m).
\]
\end{definition}

Let us recall the following characterization of infinitesimal Hilbertianity (see Theorem 4.3.3. in \cite{GigliPasqualetto2020}).

\begin{prop}\th\label{prop.paralellogramlaw}
The space $\Sob{2}(X)$ with the pointwise inner product is a Hilbert space if and only if the parallelogram rule 
\begin{linenomath}
 \begin{equation}\label{EQ: Paralellogram Law}
 2\left(|\nabla f|^2+|\nabla g|^2 \right) = |\nabla (f+ g)|^2+|\nabla (f- g)|^2
\end{equation}
\end{linenomath}
holds $\m$-a.e. for all $f,g\in \Sob{2}(X)$.
\end{prop}

\begin{definition}[Laplacian]
Let $(X,\dis,\m)$ be an infinitesimally Hilbertian space. A function $f\in\Sob{2}(X) $ is defined to be \emph{in the domain of the Laplacian}, $D(\Delta)$, if there exists some $g \in L^2(\m)$ such that:
\[
 \int hg\, dm = -\int \langle\nabla f, \nabla h\rangle d\m, \quad \forall h\in \Sob{2}(X).
\]
We denote the  function $g$ (which is unique if exists)  by $\Delta f$ and refer to it as the \emph{Laplacian of $f$}. 
\end{definition}

In the case where the underlying space is infinitesimally Hilbertian we have at our disposal the following good properties of the Laplacian and of the pointwise inner product (see Theorem 4.3.3 part $v)$,  Remark 5.2.2, and Proposition 5.2.3 in \cite{GigliPasqualetto2020}).

\begin{prop}
Let $(X,\dis,\m)$ be an infinitesimally Hilbertian space. Then we have the following properties:
\begin{itemize}
    \item $D(\Delta)$ is a vector space.
    \item  The Laplacian $\Delta$ is linear.
    \item  The operator $\Delta (\cdot)$ is closed.
    \item  The pointwise inner product  $\langle \nabla \cdot, \nabla\cdot \rangle $ is continuous.
\end{itemize}
\end{prop}

\subsection{Bakry-\'Emery condition}

Instead of giving the usual definition of an $\RCD-$space, we  present the  Bakry-\'Emery condition, originally introduced in \cite{AmbrosioGigliSavare2015} by Ambrosio, Gigli and Savaré; the reason being that this formulation enables us to prove that the gluing of cohomogeneity one $\RCD-$spaces is again an $\RCD-$space in Section~\ref{S: Gluing of RCD-spaces}. 

\begin{definition}[Bakry-\'Emery condition]\th\label{def.BE}
 Let $(X,\dis,\m)$ be a metric measure space. We  say that it satisfies the \emph{$\BE(K,N)$-condition} for some $K\in \R$, $N\geq 1$, if for all $f \in D(\Delta)$ with $\Delta f \in \Sob{2}(X)$ and for all $g\in D(\Delta)$ non-negative, bounded with $\Delta g \in L^{\infty}(\m)$ we have
\begin{equation}
\label{eq.BakryEmery}
  \frac{1}{2}\int \Delta g |\nabla f|^2d\m-\int g \langle \nabla (\Delta f),\nabla f\rangle d\m\geq K\int g|\nabla f|^2 d\m+\frac{1}{N}\int g(\Delta f)^2d\m.
\end{equation}
\end{definition}

 Satisfying the $\BE(K,N)$ condition is equivalent to the space satisfying the $\CD^\ast(K,N)$ for infinitesimally Hilbertian metric measure spaces $(X,\dis,\m)$, under the under the following assumptions:

\begin{enumerate}
 \item \label{item.Sobolev-to-Lipschitz} The \emph{Sobolev-to-Lipschitz property} is satisfied, that is,  $\forall f\in \Sob{2}(X)$ with $|\nabla f|\leq 1$ there exists a Lipschitz representative of $f$.
 \item \label{item.volumegrowth} For any $x_0 \in X$ and $c>0$ we have 
 \[
 \int \exp(-c\dis^2(x_0,x))\dis\m(x)< \infty.
 \]
\end{enumerate}
This fact was shown by Ambrosio, Gigli and Savaré for $N=\infty$ (under equivalent technical assumptions) in \cite[Theorem 1.1]{AmbrosioGigliSavare2015} and for $N<\infty$ by Erbar, Kuwada and Sturm in \cite[Theorem 7]{ErbarKuwadaSturm2015}. Note also the equivalence between the $\CD^\ast(K,N)$- and $\CD(K,N)$- conditions due to Cavalletti and Milman, \cite[Corollary 13.7]{CavallettiMilman2021} and Li \cite{Li2024}. 

Since we only consider compact spaces with a Radon measure, it is clear that condition (2) above is automatically fulfilled.

From these considerations, we point out that we do not need to make use of the optimal transport formulation of the curvature-dimension condition, for example as stated in \cite{GuijarroSantos2019}, so instead we settle on the following definition.

\begin{definition}[$\RCD$-spaces]\th\label{D: RCD-condition}
Let $K \in \R, N \geq 1$, we  say that a metric measure space $(X,\dis,\m)$ is an \emph{$\RCD (K,N)$-space} if it satisfies the  $\BE(K,N)$ condition, is infinitesimally Hilbertian, satisfies Sobolev-to-Lipschitz (\ref{item.Sobolev-to-Lipschitz}), and the volume growth condition (\ref{item.volumegrowth}).
\end{definition}

One important thing to mention is that in general the Bakry-\'Emery condition is strictly weaker than the $\RCD$-condition. In \cite{Honda2018} Honda gives examples of spaces that satisfy a $\BE(K,N)$ condition but that are not $\RCD$-spaces. The reason being that they do not satisfy the Sobolev-to-Lipschitz property. Moreover, the Sobolev-to-Lipschitz property is essential as the $\BE(K,N)$-condition alone cannot imply anything about the geometry of the space as was pointed out to us by Gigli with the following examples. Any metric space equipped with an atomic measure satisfies $\BE(K,N)$ for any $K$ and $N$ as the Cheeger energy is identically $0$. A more geometric example is the following: For a given compact smooth embedded submanifold $M$ in $\mathbb{R}^n$ one can consider the induced Riemannian metric $g$ from $\mathbb{R}^n$ (and in turn the Riemannian distance $\dis_g$) and Riemannian volume $d\mathrm{vol}_g$. Then by compactness, the metric measure space $(M, \dis_g, d\mathrm{vol}_g)$ is an $\RCD(K,N)$-space for adequate $K$ and $N$. Now, $M$ can be endowed with the restriction of the Euclidean distance as well, i.e., $\dis(x,y)= \left\| x-y\right\|$ (and the same measure $d\mathrm{vol}_g$). It is clear that this distance in general does not make $M$ a geodesic space and thus, this metric measure structure cannot satisfy the $\RCD$-condition in general. However, the Cheeger energies associated to both distances agree as can be seen directly from the definition above since the Lipschitz constants of functions almost agree for both distances since they almost coincide at small scales.   

Often, if there is no room for confusion, we simply refer to  metric measure spaces satisfying \th\ref{D: RCD-condition}  as $\RCD$-spaces without making any explicit mention of the parameters $K$ and $N$.

Denote by $\mathbf{X}$ the class of metric measure spaces $(X,\dis,\m)$ satisfying:

\begin{enumerate}[(i)]
\item $(X,\dis)$ is complete and separable.
\item $\m$ is a Borel measure (i.e. defined on the Borel $\sigma$-algebra of open sets of $X$)  with $\mathrm{supp}(\m) = X$, satisfying for all $r>0$ and $x\in X$:
\[
\m(B_r(x))\leq ce^{Ar^2},
\]
for  appropriate constants $A,c\geq 0$.
\item $(X,\dis,\m)$ is infinitesimally Hilbertian.
\end{enumerate}

Then, we have the following global-to-local and local-to-global results.

\begin{theorem}[Global-to-Local for $\RCD(K, N )$, Proposition 7.7 in \cite{AmbrosioMondinoSavare2016}]\th\label{T: Global-to-local}
Let $(X, \dis, \m) \in  \mathbf{X}$ be  an $\RCD(K, N )$-space, and let $U\subset X$ be open. If $\m(\partial U ) = 0$ and $(\bar{U} , \dis)$ is geodesic, then $(\bar{U} , \dis, \m\llcorner \bar{U} )$ is an $\RCD(K, N )$-space.
\end{theorem}

\begin{theorem}[Local-to-Global for $\RCD(K, N )$, Theorem 7.8 in \cite{AmbrosioMondinoSavare2016}]\th\label{T:local-to-global}
Let $(X, \dis, \m) \in \mathbf{X}$ be a length space and
assume that there exists a covering $\{U_i\}_{i\in I}$ of $X$ by non-empty open subsets such that $\m(U_i) <
\infty$ if $K < 0$, and $(\overline{U}_i, \dis, \m \llcorner
\overline{U}_i)\in X$ satisfy $\RCD(K, N )$.
Then $(X, \dis, \m)$ is a $\RCD(K, N )$ space.
\end{theorem}

Let us describe some features of $\RCD$-spaces that will be useful to us. Their proofs can be found for example in \cite{ErbarKuwadaSturm2015}. We remark that the tensorization property in item (2) below was originally shown to hold for the $\BE$-condition in \cite{AmbrosioGigliSavare2015}.

\begin{prop}\th\label{prop.productsrescallings}
Let $(X,\dis_X, \mathfrak{m}_X),(Y,\dis_Y, \mathfrak{m}_Y)$ be $\RCD$-spaces that satisfy  the curvature-dimension bounds $(K_X,N_X), (K_Y,N_Y)$ respectively then:
\begin{enumerate}
    \item $(X,\dis_X, \mathfrak{m}_X)$ is also an $\RCD(L,M)$ for all $L\leq K_X$, $M\geq N_X$.
    \item The product $(X\times Y, \dis_{X\times Y}, \mathfrak{m}_X\otimes \mathfrak{m}_Y)$ is an $\RCD(\min\{K_X,K_Y\},N_X+N_Y)-$space.
    \item If $\alpha, \beta >0$ then the metric measure space $(X,\alpha \dis_X, \beta \mathfrak{m}_X)$ is an $\RCD(\alpha^{-2}K_X,N_X)$.
\end{enumerate}
\end{prop}

One important feature of the class of $\RCD(K,N)$-spaces is that it is compact with respect to measured Gromov-Hausdorff convergence (see \cite{AmbrosioGigliSavare2014}, \cite{Gigli2010}, \cite{Gigli2015}, \cite{GigliMondinoSavare2015}, \cite{LottVillani2009}, \cite{SturmI2006}, \cite{SturmII2006}).

\begin{definition}\th\label{def.epsilon-aprox}
Let $(X,\dis_X,x), (Y,\dis_Y,y)$ be complete and separable pointed metric spa\-ces. Given $\epsilon >0$, we  say that a map $f_\epsilon \colon B_{1/\epsilon}(x)\rightarrow B_{1/\epsilon}(y)$ is an \emph{$\epsilon$-Gromov-Hausdorff approximation} if:
\begin{enumerate}
    \item $f_\epsilon (x)= y$.
    \item For all  $u,v \in B_{1/\epsilon}(x)$, $|\dis_X(u,v)-\dis_Y(f_\epsilon (u),f_\epsilon (v))|<\epsilon$.
    \item For all $y \in B_{1/\epsilon}(y)$, there exists some $x\in B_{1/\epsilon}(x)$ such that $\dis_Y(f_\epsilon (x),y)<\epsilon$.
\end{enumerate}
\end{definition}

\begin{definition}
Let $\lbrace (X_n,\dis_n,x_n) \rbrace_{n \in \N}$ be a sequence of complete separable pointed metric spaces. We say that \emph{the sequence converges in the pointed Gromov-Hausdorff sense} to a complete separable pointed metric space $(Y,\dis_Y,y)$ if there exists a sequence $\{\epsilon_n\}_{n\in \N}\subset \R$ such that $\epsilon_n \rightarrow 0$ as $n \rightarrow \infty$, and $\epsilon_n$-Gromov-Hausdorff  approximations from $B_{1/\epsilon_n}(x)\subset (X_n,\dis_n,x_n)$ to $B_{1/\varepsilon_n}(y)\subset (Y,\dis_Y,y)$.
\end{definition}

If in addition our spaces are equipped with  non-negative Borel reference measures which are finite on balls we define:

\begin{definition}
A sequence $\lbrace (X_n,\dis_n,\m_n,x_n) \rbrace_{n \in \N}$ of complete pointed metric measure spaces \emph{converges in the pointed measured Gromov-Hausdorff} sense to a complete pointed metric measure space  $(Y,\dis_Y, \m_Y,y)$ if it converges in the pointed Gromov-Hausdorff sense  with the extra assumption that the $\epsilon_n$-Gromov-Hausdorff approximations are measurable and satisfy 
\[
f_{\epsilon\#}\m_{n} \rightharpoonup \m_Y
\]
where the topology is the $\text{weak}^{\star}$-topology with respect to continuous functions with bounded support on $Y$.

We abbreviate this by:
\[
(X_n,\dis_n,\m_n,x_n)\xrightarrow{pmGH} (Y,\dis_Y,\m_Y,y).
\]
\end{definition}

Let $(X,\dis,\m)$ be an $\RCD$-space, $x \in X$, and some $r \in (0,1)$. Consider now the pointed, rescaled, and normalized metric measure space  $(X,r^{-1}\dis,\m^x_r, x)$, where
\[
\m^x_r := \left(\int_{B_r(x)}1-\frac{\dis(x,y)}{r}d\m(y) \right)^{-1}\m.
\]
Using this notation we define:

\begin{definition}[Tangent space]
Let $(X,\dis,\m)$ be an $\RCD(K,N)$-space for $K\in \R$ and $N\in [1,\infty)$. Given a point $x_0\in X$, we say that a pointed metric measure space $(Y,\dis_Y,\m_Y,y)$ is \emph{a tangent space of $X$ at $x_0$} if there exists a sequence $\lbrace r_i \rbrace_{i\in \N}\subset (0,1)$  such that $r_i \rightarrow 0$, and the sequence 
\[
(X,r^{-1}_i \dis,\m^{x_0}_{r_i},x_0)\xrightarrow{pmGH} (Y,\dis_Y,\m_Y,y). 
\]
We denote the set of all tangents of $X$ at $x_0$ by $\Tan(X,\dis, \m,x_0)$.
\end{definition}

\begin{rmk}\th\label{R: tangent of RCD-space is RCD(0,N)-space}
We recall that if a space $(X,\dis,\m)$ is an $\RCD(K,N)$-space, then for $r,\lambda>0$ the space $(X,\dis/r,\lambda \m)$ is an $\RCD(r^{2}K,N)$-space. Since the $\RCD$-condition is stable under Gromov-Hausdorff limits, we have that tangents of $\RCD(K,N)$-spaces are $\RCD(0,N)$-spaces.
\end{rmk}

\begin{rmk}
Given an $\RCD$-space $(X,\dis,\m)$ and $x_0\in X$, the set $\Tan(X,\dis,\m,x_0)$ is non-empty (see Proposition $2.2$ in \cite{MondinoNaber2019}), but it may contain several spaces   which may not even be metric cones (see Examples $8.41,8.80,8.95$ in \cite{CheegerColding1997}). 
\end{rmk}

Given $m \in \mathbb{N}\cap [1,N]$, let $\mathcal{R}_m$ denote the set of points $x \in X$ such that $\Tan(X,\dis, \m,x) = \lbrace (\R^m,\dis_\E,\mathcal{L}^m, 0) \rbrace. $
Using Regular Lagrangian flows, Bru\`e and Semola in \cite{BrueSemola2020} proved that there is precisely a unique $m$ such that the set $\mathcal{R}_m$  has full measure.

\begin{theorem}[Essential dimension]
Let $(X,\dis,\m)$ be an $\RCD(K,N)-$space with $K \in \R$ and $N \in [1,\infty)$. Then there exists a unique $m \in \mathbb{N}\cap [1,N]$ such that $\m(X-\mathcal{R}_m)=0$.  We define such $m$ to be the \emph{essential dimension} of $(X,\dis,\m)$, and we denote it by $\dimess (X)$. The set $\mathcal{R}_m$ is the \emph{set of regular points of $X$}.
\end{theorem}

\subsection{Non-collapsed spaces}

Here we distinguish a particular class of $\RCD$-spaces: the so called non-collapsed spaces. These spaces were studied by De Philippis and Gigli in \cite{dePhilippisGigli2018}, and were further characterized in terms of properties of the $N$-dimensional Bishop-Gromov density by Brena, Gigli, Honda, and Zhu  in \cite{BrenaGigliHondaZhu2023}. 

\begin{definition}[Non-collapsed $\RCD$-space]
Let $(X,\dis,\m)$ be an $\RCD(K,N)$-space. We say that it is \emph{non-collapsed} if $\m= \Hauss^N$.    
\end{definition}

Thanks to the structure of $\RCD$-spaces, we can check that for non-collapsed spaces we  have  necessarily that $N \in \N$ (see Theorem $1.12$ in  \cite{dePhilippisGigli2018}). Non-collapsed $\RCD$-spaces also enjoy good structural properties which are stronger than those of more general $\RCD$-spaces. For example, we have the following result about the tangent spaces (see Proposition $2.8$ in \cite{dePhilippisGigli2018}).

\begin{prop}
Let $(X,\dis,\Hauss^N)$ be a non-collapsed $\RCD(K,N)$-space. Then the tangent spaces at every point are metric cones.     
\end{prop}

Regarding the essential dimension, we have the following result by Brena, Gigli, Honda, and Zhu \cite[Theorem 1.7]{BrenaGigliHondaZhu2023}
which lets us identify non-collapsed spaces. 

\begin{theorem}\th\label{T: Characterization non-collapsed}
Let $(X,\dis,\m)$ be an $\RCD(K,N)$-space. Suppose that the essential dimension of some tangent space $(Y,\dis_Y,\m_Y,y)$ at some $x\in X$ is equal to $N$, then $\m= c\Hauss^N$ for some $c>0$.    
\end{theorem}

Finally, we recall that the essential dimension of a non-collapsed $\RCD(K,N)$-space is $N$ (see the last paragraph in Page 3 of \cite{BrenaGigliHondaZhu2023}). 

\subsection{Warped products}\label{SS: Warped Products}

Recall that given $(X,\dis_X)$ a complete and separable metric space, a curve $\gamma\colon [0,1]\to X$ is \emph{absolutely continuous (AC)}, if there exists a map $f\in L^1[0,1]$ such that for every $0\leq s<t\leq 1$ we have
\begin{linenomath}
\begin{align}
  \dis_X(\gamma(t),\gamma(s))\leq \int_s^t f(r)\, dr. \label{EQ: AC}   
\end{align}
\end{linenomath}
Given an AC curve $\gamma\colon [0,1]\to X$, the following limit exists for a.e. $t_0\in [0,1]$:
\[
|\dot{\gamma}(t_0)|:= \lim_{h\to 0} \frac{\dis_X(\gamma(t_0+h),\gamma(t_0))}{|h|}.
\]
Moreover, the function $t\mapsto |\dot{\gamma}(t)|$ (called the \emph{metric speed of $\gamma$}) is in $L^1(0,1)$ and is the minimal function satisfying \eqref{EQ: AC}.

Let $(B,\dis_B)$ and $(F,\dis_F,\m_F)$ be complete, locally compact geodesic metric spaces. Consider $f\colon B\to \R_{\geq 0}$ a locally Lipschitz function. A curve $\gamma=(\alpha,\beta)\colon [0,1]\to B\times F$ is \emph{admissible} if $\alpha$ and $\beta$ are absolutely continuous in $B$ and $F$ respectively. For an admissible curve $\gamma\colon [0,1]\to B\times F$ we define the \emph{length with respect to $f$} as
\[
L_f(\gamma) := \int_0^1 \sqrt{|\dot{\alpha}(t)|^2+(f\circ \alpha)^2(t)|\dot{\beta}(t)|^2}\, dt.
\]
For convenience, we define the length of a non-admissible curve as $+\infty$. $L_f$ is a length-structure on the class of admissible curves. We define a pseudo-metric $\tilde{d}_f$ on $B\times F$, between $(p,x)$ and $(q,y)$ by setting
\[
\tilde{d}_f((p,x),(q,y)):= \inf\{L_f(\gamma)\mid \gamma \mbox{ admisible and connects }(p,x) \mbox{ to }(q,y)\}.
\]
Then we have an induced equivalence relation $\sim$, where $(p,x)\sim (q,y)$ if and only if \linebreak$\tilde{d}_f((p,x),(q,y)) =0$. We define the \emph{warped product of $B$ and $F$ via $f$} as $B\times_f F := B\times F/ \sim$ with the distance $\dis_f([p,x],[q,y]):=\tilde{d}_f((p,x),(q,y))$. By definition $B\times_f F$ is an intrinsic metric space, and since $B$ and $F$ are complete and locally compact, then the warped product is complete and locally compact. As pointed out in \cite[Remark 2.4]{Ketterer2013}, the warped product $B\times_f F$ is geodesic. 

We consider $(X,\dis_X,\m_X)$ and $(Y,\dis_Y,\m_Y)$ two complete separable length spaces, and $\m_X$ and $\m_Y$ Radon measures. We also assume that $\m_X(X)<\infty$. We consider continuous functions $\omega_d,\omega_{\m}\colon Y\to [0,\infty)$. For $\omega=(\omega_d,\omega_{\m})$ define the \emph{$\omega$-warped product}  $Y\times_{\omega} X$ as the metric space $Y\times_{\omega_d} X$, and set the measure 
\[
\m_{\omega}:= \pi_\ast((\omega_{\m}\m_Y)\otimes \m_X),
\]
where $\pi\colon Y\times X\to (Y\times_{\omega_d} X)$ is the quotient map.

When $\omega_d = f$, and $\omega_{\m}=f^N$ for $f\colon Y\to [0,\infty)$ a continuous function, we set $Y\times_{f}^N X:=Y\times_{(f,f^N)} X$. The space $Y\times_f^N X$ is referred to as the \emph{$N$-warped product of $Y$, $X$ and $f$} (see \cite{Ketterer2013}). 

We also define here the $(K,N)$-cones for $K\in \R$, $N\in (0,\infty)$: Given a metric measure space $(F,\dis_F,\m_F)$, the \emph{$(K,N)$-cone of $F$}, denoted by $\mathrm{Con}^N_K(F)$ is the metric measure space $(\mathrm{Con}_K(F),\dis_{\mathrm{Con}_K},\m^N_{\mathrm{Con}_K})$ defined by
 \[
    \mathrm{Con}_K(F) = \begin{cases}
        F\times [0,\pi/\sqrt{K}]/(F\times \{0,\pi/\sqrt{K}\}) & \mbox{if }K>0\\
        F\times [0,\infty)/(F\times \{0\}) & \mbox{if } K \leq 0,
        \end{cases}
    \]
\begin{linenomath}
\begin{align*}
&\dis_{\mathrm{Con}_K}([p,t],[q,s])\\
&:= \begin{cases}
    \cos^{-1}_K\Big(\cos_K(s)\cos_K(t)+K\sin_K(s)\sin_K(t)\cos_K\big(\dis_F(p,q)\wedge \pi\big)\Big) & \mbox{if }K\neq 0\\
    \sqrt{s^2+t^2-2st\cos\big(\dis_F(p,q)\wedge\pi\big)} & \mbox{if }K=0,
    \end{cases}
\end{align*}
\end{linenomath}
and
\[
\m_{\mathrm{Con}_K}^N :=\sin_K^N(t)dt\otimes \m_F.
\]
Here, for $I_K =[0,\pi/\sqrt{K}]$ for $K>0$ and $[0,\infty)$ for $K\leq 0$, we have $\cos_K\colon I_K\to [0,\infty)$ given by
\[
\cos_K(t) := \begin{cases}
    \cos\left(\sqrt{K}t\right)& \mbox{if } K>0\\
    \cosh\left(\sqrt{-K}t\right)&\mbox{if } K<0,
\end{cases}
\]
and $\sin_K\colon I_K\to [0,\infty)$ given by
\[
\sin_K(t):= \begin{cases}
    \frac{1}{\sqrt{K}}\sin\left(\sqrt{K}t\right) & \mbox{if }K>0\\
    t & K=0\\
    \frac{1}{\sqrt{|K|}}\sinh\left(\sqrt{|K|}t\right) & \mbox{if }K<0.\\
\end{cases}
\]
\begin{rmk}
If $\mathrm{diam}(F)\leq \pi$ and $F$ is a length space, then the $(K,N)$-cone $\mathrm{Con}^N_K(F)$ coincides with the $N$-warped product $I_K\times_{\sin_K}^N F$ (e.g. \cite[Theorem 3.6.17]{BuragoBuragoIvanov}).
\end{rmk}

\subsection{Transformation groups}

In order to set the notation, let us recall the basic facts of the theory of transformation groups that we use  throughout the manuscript.

Let $(X,\dis,\m)$ be a metric measure space, and $G$ a  topological group. A \emph{continuous (left) action of $G$ on $X$} is a continuous map $\alpha\colon G\times X\to X$ such that for any $g,h\in G$ and $x\in X$ we have $\alpha(g,\alpha(h,x)) = \alpha(gh,x)$, and for $e\in G$ the identity element and any $x\in X$ we have $\alpha(e,x) = x$. From now on we denote $gx:=\alpha(g,x)$.

The \emph{isotropy group at $x\in X$} is the closed subgroup $G_x=\left\{ g\in G \mid gx=x\right\}$. We say the action $\alpha$ is \emph{effective} if the only element of $G$ that fixes all the points in $X$ is the identity element, i.e. $\cap_{x\in X} G_x = \{e\}$. The action is \emph{proper} if the map $A\colon G\times X\to X \times X$ given by $A(g,x) = (gx,x)$ is a proper map (that is, the preimage $A^{-1}(K)\subset G\times X$ of a compact subset $K\subset X\times X$ is compact. 

The \emph{orbit of a point $x\in X$} is the subset of $X$ given by $G(x)=\left\{ gx\in X \mid g\in G \right\}$. The \emph{orbit type of an orbit $G(x)$} is the conjugation class $(G_x):= \{gG_xg^{-1}\mid g\in G \}$ of the isotropy.

The \emph{orbit space} is denoted by $X/G$ and $\pi\colon X\to X/G$ is the natural projection. We equip $X/G$ with the quotient topology. We also use the notation $ U^\ast:=\pi(U)$, for any $U\subset X$. In particular, $X^\ast=X/G$ and $\pi(x)=x^\ast$ for each $x\in X$.

When we have that $\dis(\alpha(g,x),\alpha(g,y)) = \dis(x,y)$ for any $g\in G$ and any $x,y\in X$ we say that \emph{the action $\alpha$ is by isometries}. Given $g\in G$ fixed we have a homeomorphism $\alpha_g\colon X\to X$ given as $\alpha_g(x) = \alpha(g,x)$. In the case when the action is by isometries and $(\alpha_g)_\# \m = \m$ for any $g\in G$, we say that \emph{the action is by measure-preserving isometries}, and that the metric $\dis$ is \emph{$G$-invariant}. For simplicity, unless otherwise stated, we denote the action as $gx := \alpha(g,x)$.  In the case where the orbits are compact, the orbit space is equipped with the \emph{orbital metric} 
\[
\dis^\ast(x^\ast, y^\ast) := \dis(G(x),G(y))
\]
and the \emph{push-forward measure} $\m^\ast = \pi_{\#}\m$. 

When the action is proper, the isotropy groups are compact. Moreover, properness gives us the description of each orbit as a homogeneous space.

\begin{theorem}[Theorem 2.3 in \cite{CorroKordass2021}]\th\label{T: equivariant homeomorphism from the orbit}
Let $G$ be a Hausdorff topological group acting properly on
a metric space $(X, \dis)$, where $\dis$ is $G$-invariant. Fix $x\in X$, and consider $\alpha_{x}\colon  G \to X$ given by $\alpha_x(g) = gx$. Let $\rho\colon  G \to G/G_x$ be the quotient map. Then there exists a $G$-equivariant homeomorphism  $\tilde{\alpha}_x \colon G/G_x \to G(x)$ onto the orbit through $x$ such that $\tilde{\alpha}_x \circ \rho = \alpha_x$. Furthermore, the orbit $G(x)$ is a closed subspace of $X$.
\end{theorem}

\begin{rmk}
In the case when $G$ is a compact Lie group, an action of $G$ on a metric space $X$ is proper (see for example  \cite[Theorem 2.2]{CorroKordass2021}). From now on, we only consider compact Lie groups.
\end{rmk}

We recall from \cite[Theorem 1.3]{GalazGarciaKellMondinoSosa2018} that given an $\RCD$-space $(X,\dis,\m)$ with an action by measure preserving isometries by a compact Lie group $G$, there exist (up to conjugation) a unique subgroup $G_{\min}<G$, such that the orbit $G(y)$ of $\m$-a.e $y\in M$ is homeomorphic to the quotient $G/G_{\min}$. We call points $y\in M$ such that $G_y$ is conjugate to $G_{\min}$ in $G$ \emph{points with principal orbit type}. The orbit of a point of principal orbit type is referred to as a \emph{principal orbit}.

The following notion of regularity for the action was introduced in \cite{GalazGarciaKellMondinoSosa2018} and will also be useful to us. 

We will use the following definition of slice taken  from Bredon \cite{Bredon1972}
\begin{definition}\th\label{D: Slice}
Let $x\in X$ a set $x\in S$ is called a slice at $x$ if it satisfies the following:
\begin{enumerate}[(i)]
    \item\label{Slice-1} $S$ is closed in $G(S)$
    \item\label{Slice-2} $G(S)$ is an open neighbourhood of $G(x)$ 
    \item\label{Slice-3} $G_x (S)=S$
    \item\label{Slice-4} $(gS)\cap S \neq \emptyset $ implies that $g\in G_x$.
\end{enumerate}
\end{definition}

Sometimes we will use the following equivalent characterization of a slice.

\begin{theorem}[Chapter $2$, Section $4$ in \cite{Bredon1972}]\th\label{T: equivalence slice}
For an action $G$ on $X$, given a point $x\in X$ a slice through $x$ can be equivalently defined as a subset $S_x\subset X$ that satisfies:
\begin{itemize}
    \item $gS_x\subset S_x$ for all $g\in G_x$
    \item $gS_x\cap S_x \neq \emptyset$ implies that $g\in G_x$
    \item Setting $H= G_x$, there exists an open neighborhood $U \subset G/H $ around $eH$, and
    a cross-section $\chi \colon U \rightarrow G $ such that the function $F \colon U \times S_x \rightarrow X$ given by 
    \[
    F(gH,s) = \chi (gH)s
    \]
    is a homeomorphism onto its image.
\end{itemize}
\end{theorem}

We now state the main recognition tool in the theory of transformation groups, the so called \emph{Slice Theorem} (see \cite[Definition 4.1, Theorem 4.2]{Bredon1972}). 

\begin{theorem}[Slice Theorem]
\th\label{T:Slice-Theorem}
Let a compact Lie group $G$ act by isometries on a metric space $(X,\dis)$. For all $x\in X$ there is some $r_0>0$, such that  for all $0<r\leq r_0$ there exists a slice $S_x = S_x(r)$  and there is an equivariant homeomorphism 
\[
G\times_{G_x} S_x \to B_{r}(G(x)).
\]
\end{theorem}

\begin{rmk}
The dependence of $S_{x_0}$ on the radius $r$ is given by \th\ref{MT: Slice Theorem Cohomogenity one}.
\end{rmk}

Every point in a completely regular space has a slice (see \cite{Palais1960}). This fact is sometimes referred to as the Slice Theorem. It follows then that every point on an $\mathsf{RCD}$-space admits a slice. However, given an Alexandrov space $X$ and $G$ a Lie group acting by isometries, for each point $x\in X$ we can find a slice containing $x$ which is a metric cone over a positively curved Alexandrov space (namely the space of normal directions at the point in question; see the Slice Theorem for Alexandrov spaces due to Harvey and Searle, \cite[Theorem B]{HarveySearle2017}). This implies that for any point in an Alexandrov space equipped with a compact Lie group action by isometries, there is a slice through the given point which in itself admits an Alexandrov structure. 

Moreover, for  the case when the space is a Riemannian manifold with a compact Lie group action by isometries, the same holds: for any point there exists a slice containing the point, such that it admits a Riemannian space structure.

Thus in general, when considering a geometric space and a group action that preserves the geometric structure, we are interested in finding in a well-defined way a slice which belongs to the same class of geometric objects as the global space. 

\subsection{Equivariant Gromov-Hausdorff convergence}

We now define the notion of equivariant Gromov-Hausdorff convergence (see \cite{FukayaYamaguchi1992}). Given some  $r >0$ and 
$G$ a group acting by isometries on a pointed metric space $(X,\dis, x)$ we define: 
\[
G(r) := \lbrace g \in G\,|\, gx\in B_r (x) \rbrace.
\]

\begin{definition}
Let $(X,\dis_X,x,G)$, $(Y,\dis_Y,y,H)$ be two pointed metric spaces with $G$, $H$ groups acting isometrically on $X$ and $Y$ respectively. An \emph{equivariant $\epsilon$-Gromov-Hausdorff approximation} is a triple $(f_\epsilon, \varphi,\psi)$ which consists of maps
\[
f_\epsilon \colon B_{1/\epsilon}(x)\rightarrow B_{1/\epsilon}(y),\quad \varphi \colon G(1/\epsilon)\rightarrow H(1/\epsilon),\quad \psi \colon H(1/\epsilon)\rightarrow G(1/\epsilon)
\]
such that:
\begin{enumerate}
    \item $f_\epsilon$ is an $\epsilon$-Gromov-Hausdorff approximation.
    \item If $g\in G(1/\epsilon)$ and both $p, gp \in B_{1/\epsilon}(x)$, then 
    \[
    \dis_Y(f_\epsilon(gp), \varphi(g)f_\epsilon(p))<\epsilon.
    \]
    \item If $h\in H(1/\epsilon)$ and both $p, \psi(h)p \in B_{1/\epsilon}(x)$, then 
    \[
    \dis_Y(f_\epsilon(\psi(h)p),hf_\epsilon(p))<\epsilon.
    \] 
\end{enumerate}
\end{definition}

\begin{definition}
A sequence $\lbrace (X_n,\dis_n,x_n,G_n) \rbrace_{n \in \N}$ is be said to \emph{converge in the equivariant pointed Gromov-Hausdorff sense} to a pointed metric space  $(Y,\dis_Y,y,H)$ if there exist equivariant $\epsilon_n$-Gromov-Hausdorff approximations such that  $\epsilon_n \rightarrow 0$ as $n \rightarrow \infty$.   
\end{definition}

With this definition we present the following result that is used later on.

\begin{theorem}[Proposition 3.6 in \cite{FukayaYamaguchi1992}]\th\label{T: p-G-H implies eq-p-G-H}
Assume that for a sequence $\{X_i, \dis_{X_i} , H_i, x_i)\}_{i\in \N}$ of pointed metric spaces with $H_i\subset \Iso(X_i,\dis_i)$ a closed subgroup with respect to the compact-open topology acting effectively on $X_i$, we we have that $(Y, \dis_Y , y)$ is the pointed-Gromov-Hausdorff limit of $\{(X_i, \dis_{X_i}, x_i)\}_{i\in \N}$. Then there exist $H' \subset \Iso(Y, \dis_Y )$ a closed subgroup with respect to the compact-open topology, acting effectively on $Y$, and a subsequence $\{(X_{i_k} , \dis_{X_{i_k}} , H_{i_k} , x_{i_k} )\}_{k\in \N}$
converging to $(Y, \dis_Y , H', y)$ in the pointed-equivariant Gromov-Hausdorff sense.
\end{theorem}

\section{Slice Theorem for cohomogeneity one geodesic spaces}\label{S: Slice THeorem RCD}

In this section we do not need any curvature bounds on $X$. However, we  need  to assume some properties for some geodesics of $X$. In order to properly describe these properties we need to introduce some concepts of optimal transport which we recall below. The interested reader can find the proofs of the claims for example in \cite{AmbrosioGigli2013}.

A metric space $(X,\dis)$ is called a geodesic space if any two points $x_0,x_1 \in X$ can be joined by a curve 
$\gamma \colon [0,1]\rightarrow X$ with $\gamma_0 = x_0, \gamma_1=x_1$ and such that:
\[
\dis(\gamma_t,\gamma_s) =|t-s|\dis(\gamma_0,\gamma_1), \text{ for all } s,t\in [0,1].
\]
The space of all geodesics in $X$ is  denoted by $\Geo(X)$ and equipped  with the $\sup$ norm.

Let $\mathbb{P}(X)$ denote the space of Borel probability measures on $X$, and for $p\in [1,\infty)$ by $\mathbb{P}_p(X)$ those that have finite $p$-moments, that is
\[
\mathbb{P}_p(X) := \Big\lbrace \mu \in \mathbb{P}(X)\,|\, \int \dis^p(x,x_0)d\mu(x)< \infty \text{ for some } x_0 \in X \Big\rbrace .
\]
We then endow $\mathbb{P}_p(X)$ with the $L^p$-Wasserstein distance: For all $\mu_0,\mu_1\in \mathbb{P}_p(X)$
\begin{equation}
    \mathbb{W}_p^p(\mu_0,\mu_1) = \inf_{\rho} \int \dis^p(x,y) d\rho (x,y),
\end{equation}
where the infimum is taken over all probability measures in $\rho \in\mathbb{P}(X\times X)$ such that their marginals are $\mu_0$ and $\mu_1$, i.e. for the projections $X\times X\ni (x,y)\mapsto \mathrm{proj}_1(x,y)=x$, $X\times X\ni (x,y)\mapsto \mathrm{proj}_2(x,y)=y$ we have $(\mathrm{proj}_i)_\#\rho=\mu_{i-1}$ for $i\in \{1,2\}$. It can be proved that one can always find measures that achieve such infimum, they are called optimal transports and we denote the set of them by $\Opt(\mu_0,\mu_1)$. 

The $L^2$-Wasserstein space will inherit some geometric properties from its base. For example, if $(X,\dis)$ is a geodesic space, then $(\mathbb{P}_2(X),\mathbb{W}_2)$ is also a geodesic space. Furthermore, the next result tell us that the optimal transport between two measures occurs by moving  along geodesics of $X$.
Recall that for $t\in [0,1]$, the evaluation map $e_t \colon \Geo(X)\rightarrow X$ is defined as $e_t(\gamma)= \gamma_t$.

\begin{theorem}{(Theorem 3.10 in \cite{AmbrosioGigli2013})}\th\label{T: Dynamical optimal transport}
Let $(X,\dis)$ be a Polish geodesic space. Then the $L^2$-Wasserstein space  $(\mathbb{P}_2(X),\mathbb{W}_2)$ is also a geodesic space. Furthermore, the following are equivalent:
\begin{enumerate}
    \item A curve $t\mapsto \eta_t$ is a geodesic.
    \item There exists a measure $\eta\in \mathbb{P}(\Geo(X))$  called a dynamical optimal transport, such that $(e_0,e_1)_{\#}\eta \in \Opt(\eta_0,\eta_1)$ and 
    \[
    \eta_t = e_{t\# }\eta.
    \]
\end{enumerate}
The set of measures described in $(2)$ is denoted by $\Opt\Geo (\mu_0,\mu_1)$.
\end{theorem}

We recall the definition of an essentially non-branching space which was first introduced by Rajala and Sturm in \cite{RajalaSturm2014}. Although in said paper the authors were working with spaces satisfying a lower Ricci curvature bound, we stress that the definition does not rely on it. 

\begin{definition}\th\label{D: essentiallynonbranch}
A metric measure space $(X,\dis,\m)$ is essentially non-branching if for every $\mu_0,\mu_1\in \mathbb{P}_2(X)$  such that $\mu_0,\mu_1 \ll \m$, we have  that every $\rho\in \Opt\Geo(\mu_0,\mu_1)$
is concentrated on a set of non-branching geodesics.
\end{definition}

\begin{rmk}
In \cite{Ohta2014} Ohta constructed a couple of explicit examples of spaces that are essentially non-branching in the sense of \th\ref{D: essentiallynonbranch} and that do exhibit some branching geodesics. 
\end{rmk}

We now recall the main theorem of this section:
\vspace*{8pt}
\begin{duplicate}[\ref{MT: Slice Theorem Cohomogenity one}]
Let $(X,\dis, \mathfrak{m})$ be an essentially non-branching space  and $G$ a compact Lie group acting on $X$ by isometries such that $X^{\ast}$ is isometric to $[-1,1]$, $[0,\infty)$, $\R$, or $\Sp^1$. Then for any $x_0\in X$, there exists $\delta_0>0$ such that, for any $0<\delta<\delta_0$ we have that
\[
  S_{x_0}:= \lbrace y \in X \,|\, \dis(x_0,y)=\dis(G(x_0), y)  \rbrace\cap B_{\delta}(x_0)
\]
is a slice through $x_0$.
\end{duplicate}

\begin{rmk}
Here we need to take $\delta_0$ small enough so that $\pi(B_{\delta_0}(x_0))$ is homeomorphic to one of $(x_0^\ast-\delta_0,x_0^\ast+\delta_0)$ or $[x_0^\ast,x_0^\ast+\delta)$ or $(x_0-\delta_0,x_0]$, for $\pi\colon X\to X/G$.
\end{rmk}

Before going through the proof of this Theorem we will need to introduce a particular class of geodesics which let us define  the set that turns to be the slice. 

\begin{definition}
Let $\pi \colon X \rightarrow X^\ast$ be the quotient map. We will say that a geodesic $\gamma \in \Geo (X)$ is horizontal if $\pi\circ \gamma \in \Geo(X^\ast)$.
\end{definition}

\begin{lemma}\th\label{lemma.projection.horizontalgeo}
Let $\gamma \in \Geo(X)$ be such that $\dis(\gamma_0,\gamma_1)= \dis^\ast (\gamma_0^\ast,\gamma_1^\ast)$. Then $\pi\circ \gamma \in \Geo (X^\ast)$.    
\end{lemma}

\begin{proof}
Let $\gamma \in \Geo(X)$ be such that $\dis(\gamma_0,\gamma_1)= \dis^\ast (\gamma_0^\ast,\gamma_1^\ast)$, and denote by $\gamma^\ast = \pi \circ \gamma$. Take $s,t\in [0,1]$, and observe that we can assume without loss of generality that $s<t$. Now there exists some $g\in G$ such that $\dis(\gamma_s,g\gamma_t)= \dis^\ast(\gamma_s^\ast,\gamma_t^\ast)$. It is clear that $\dis(g\gamma_t, g\gamma_1)= \dis(\gamma_t,\gamma_1)$ so then
\begin{align*}
 \dis^\ast(\gamma_0^\ast,\gamma_1^\ast) &\leq\dis(\gamma_0,\gamma_1)\\
 &\leq \dis(\gamma_0,\gamma_s)+\dis(\gamma_s,g\gamma_t)+\dis(g\gamma_t,g\gamma_1)\\
 &\leq \dis(\gamma_0,\gamma_s)+\dis(\gamma_s,\gamma_t)+\dis(\gamma_t,\gamma_1)\\
 &= \dis(\gamma_0,\gamma_1)= \dis^\ast(\gamma_0^\ast,\gamma_1^\ast).
\end{align*}
This implies that $\dis(\gamma_s,\gamma_t)= \dis(\gamma_s,g\gamma_t)= \dis^\ast(\gamma_s^\ast,\gamma_t^\ast)$.
\end{proof}

\begin{lemma}\th\label{lemma.horizontaltransport}
Let $\mu_G, \nu_G \in \mathbb{P}_2(X)$ be two $G-$invariant measures (i.e. $g_{\#}\mu_G=\mu_G,g_{\#}\nu_G=\nu_G, \forall g \in G$) then any dynamical optimal transport between them must be supported on horizontal geodesics of $X$.   
\end{lemma}

\begin{proof}
Let $\mu_G, \nu_G \in \mathbb{P}_2(X)$ be $G$-invariant  measures  and take $\rho \in \Opt (\mu_G,\nu_G)$. Let $\pi\colon X\rightarrow X^\ast$ denote the quotient map. Then by \cite[Theorem 3.2]{GalazGarciaKellMondinoSosa2018} $(\pi,\pi)_{\#}\rho \in \Opt(\mu_G^\ast,\nu_G^\ast)$, where $\mu_G^\ast, \nu_G^\ast$ are now viewed as probability measures in $X^\ast$. From this we have that
\begin{align*}
    \int \dis^2(x,y) d\rho(x,y) &= \int \dis^{\ast 2} (x^\ast, y^\ast) d(\pi,\pi)_{\#}\rho (x^\ast,y^\ast)\\
    &= \int \dis^{\ast 2}(\pi(x),\pi(y))d\rho(x,y).
\end{align*}
This implies 
\[
\int \dis^2(x,y)-\dis^{\ast 2}(\pi(x),\pi(y)) d\rho (x,y)= 0.
\]
As the integrand is non-negative we have that it must be equal to zero $\rho$-a.e. This yields that  all points $(x,y)\in \supp \rho$ are joined by horizontal geodesics. Hence, when we lift it to a measure in $\Opt \Geo (\mu_G,\nu_G)$  via a geodesic selection (see for example Theorem 3.10 and Lemma 3.11 in \cite{AmbrosioGigli2013}) the resulting measure must be supported on horizontal geodesics.
\end{proof}

In \cite{GalazGarciaKellMondinoSosa2018} Galaz-Garc\'ia--Kell--Mondino--Sosa studied the relationship between measures supported on the quotient space $X^\ast$ and convenient lifts into $\mathbb{P}(X)$ that could be used to deduce curvature bounds on the quotient. Let us recall parts of this analysis (the interested reader can find further details in subsection $2.3$ and Section $3$ in \cite{GalazGarciaKellMondinoSosa2018} ).

Let $\mathbb{H}$ be the normalized Haar measure of the compact Lie group $G$. Then for every $x \in X$ the measure 
\begin{equation}\label{G-invariant-lift-Dirac-delta}
    \nu_x = \int_G \delta_{gx}d\mathbb{H}(g)
\end{equation}
is the unique $G-$invariant measure such that $\pi_{\#}\nu_x = \delta_{x^\ast}$. Furthermore, the assignment $x^\ast \mapsto \nu_x$ is well defined and measurable.
This can be used to define a lift map of measures on $X^\ast$
\begin{align*}
    \Lambda \colon  \mathbb{P}(&X^\ast) \rightarrow \mathbb{P}(X),\\
    &\mu \mapsto \Lambda(\mu) := \int_{X^\ast}\nu_xd\mu(x^\ast).
\end{align*}
Let $\mathbb{P}^G(X)$ denote the set of $G-$invariant measures on $X$. It is easy to check that $\Lambda(\mathbb{P}(X^\ast))\subset \mathbb{P}^G(X). $

\begin{theorem}{(Theorem $3.2$ in \cite{GalazGarciaKellMondinoSosa2018})}\th\label{theorem-lift-to-G-invariant-measures}
Let $\Lambda \colon \mathbb{P}(X^\ast)\rightarrow \mathbb{P}^G(X)$ be the lift defined above. Assume that $\mu_0,\mu_1\in \mathbb{P}(X^\ast)$ and let $\Lambda(\mu_0),\Lambda(\mu_1) \in \mathbb{P}^G(X)$ be their respective lifts. Denote by $\mathbb{P}^{ac}(X)$ the set of all measures absolutely continuous with respect to $\mathfrak{m}$. Then the following holds:
\begin{enumerate}
    \item\label{theorem-lift-to-G-invariant-measures (1)} $\Lambda({\mu}_0)$ is the unique $G-$invariant probability measure that satisfies $\pi_{\#}\Lambda({\mu}_0)=\mu_0$,
    \item\label{theorem-lift-to-G-invariant-measures (2)} $\Lambda(\mathbb{P}^{ac}(X^\ast))= \mathbb{P}^{ac}(X)\cap \mathbb{P}^G(X)$,
    \item $\Lambda(\mathbb{P}_p(X^\ast))= \mathbb{P}_p(X)\cap\mathbb{P}^G(X), p\geq 1$,
    \item $\mathbb{W}_p(\Lambda({\mu}_0),\Lambda({\mu}_1))= \mathbb{W}_p(\mu_0,\mu_1)$ whenever $\mu_0,\mu_1\in \mathbb{P}_p(X^\ast)$
\end{enumerate}
Therefore, $\Lambda$ is an isometric embedding that preserves absolutely continuous measures. And in particular, the lifts of Wasserstein geodesics in $\mathbb{P}_p(X^\ast)$ are $G-$invariant Wasserstein geodesics in $\mathbb{P}_p(X)$.
\end{theorem}

In order to prove the main result of this section we need to be more careful on how we lift measures.
Observe that given a geodesic $\gamma \in \Geo(X)$, we have that $(\gamma [0,1],\dis)$ is a   convex subset of $(X,\dis)$.

\begin{prop}\th\label{prop.induced-isometry}
Fix $\gamma^\ast \in \Geo(X^\ast)$ and let $\gamma\in \Geo(X)$ be a horizontal geodesic such that $\pi(\gamma_t)=\gamma^\ast_t$ for all $t \in [0,1]$. Then there is an isometry $F \colon (\gamma^\ast [0,1],\dis^\ast) \rightarrow (\gamma[0,1], \dis)$ that induces an isometry of their respective Wasserstein spaces. Furthermore, $F$ also induces a correspondence between (dynamical) optimal transports in $(\gamma^\ast [0,1],\dis^\ast)$ and (dynamical) optimal transports in $ (\gamma[0,1], \dis)$.
\end{prop}
\begin{proof}
Let us define the map $F \colon (\gamma^\ast [0,1],\dis^\ast)\rightarrow (\gamma [0,1],\dis)$ by $F(\gamma^\ast_t)= \gamma_t$. This is clearly an isometry since $\gamma $ is a horizontal geodesic such that $\pi (\gamma_t)= \gamma^\ast_t$, for any  $t\in [0,1]$. By taking pushforwards, we have a map 
\[
F_{\#}\colon \mathbb{P}_2(\gamma^\ast[0,1])\rightarrow\mathbb{P}_2(\gamma[0,1]).
\]
Take two measures $\mu_0^\ast,\mu_1^\ast\in \mathbb{P}_2(\gamma^\ast[0,1])$ and consider $\pi^\ast \in \Opt (\mu_0^\ast,\mu_1^\ast)$. The measure $(F,F)_{\#}\pi^\ast$ is an admissible plan between the measures $F_{\#}\mu_0^\ast$ and $F_{\#}\mu_1^\ast$. Now,
\begin{linenomath}
\begin{align*}
\mathbb{W}_2^2(F_{\#}\mu_0^\ast,F_{\#}\mu_1^\ast)&\leq \int_{\gamma[0,1]} \dis^{2}(x,y)d(F,F)_{\#}\pi^\ast (x,y)\\ &= \int_{\gamma^\ast[0,1]} \dis^{\ast 2}(x^\ast,y^\ast)d\pi^\ast (x^\ast,y^\ast) = \mathbb{W}^2_2(\mu_0^\ast, \mu_1^\ast).
\end{align*}
\end{linenomath}
When we consider the maps $F^{-1}$ and $F_{\#}^{-1}$ we obtain the reversed inequality, and thus conclude that $\mathbb{W}_2^2(\mu_0^\ast,\mu_1^\ast)= \mathbb{W}_2^2(F_{\#}\mu_0^\ast,F_{\#}\mu_1^\ast)$. Hence $F_{\#}$ is an isometry.
Observe that we have also shown that $(F,F)_{\#}$ induces a correspondence between optimal transports.
Thus, we only need to consider dynamical optimal transports. But first let us notice that if $\sigma \in \Geo(\gamma^\ast[0,1])$, then the curve $F(\sigma)$ defined by $F(\sigma)_t = F(\sigma_t)$ for all $t \in [0,1]$  belongs to $\Geo(\gamma [0,1])$.
Taking pushforwards we have 
\[
F_{\#}\colon \mathbb{P}(\Geo(\gamma^\ast[0,1]))\rightarrow \mathbb{P}(\Geo(\gamma [0,1])).
\] 
Proceeding just as before it is easy to confirm that if $\eta \in \Opt\Geo(\mu_0^\ast,\mu_1^\ast)$ for some measures 
$\mu_0^\ast, \mu_1^\ast \in \mathbb{P}_2(\gamma^\ast[0,1])$, then $F_{\ast}\eta \in \Opt\Geo(F_{\#}\mu_0^\ast,F_{\#}\mu_1^\ast)$.
\end{proof}

In the next result we prove that for essentially non-branching spaces with a cohomogeneity one group action we can further affirm that horizontal geodesics cannot branch. If that were to happen, we  can then find measures $\mu_0,\mu_1\in \mathbb{P}(X)$ such that 
the transport is done by moving along branching horizontal geodesics. Now, by taking $G-$invariant measures $\int_G g_{\#}\mu_id\mathbb{H}(g)$
we can obtain measures that are absolutely continuous with respect to $\mathfrak{m}$ and such that the transport is again realized by moving along horizontal branching geodesics. This gives us a contradiction.

\begin{theorem}\th\label{proposition.nonbranchinghorizontalgeodesic}
Let $(X,\dis,\mathfrak{m})$ be an essentially nonbranching space with $G$ compact Lie group acting by measure preserving isometries, assume that $(X^\ast,\dis^\ast, \mathfrak{m}^\ast)$ is isometric to either $[-1,1]$, $[0,\infty)$, $\mathbb{R}$, or $\Sp^1$.   Take $x_0\in X$  a point such that $x_0^\ast \in int X^\ast$. Then $x_0$ is in the interior of a horizontal geodesic. Furthermore, this geodesic cannot branch.     
\end{theorem}

\begin{proof}
Let $x_0\in X$ such that $x_0^\ast \in \mathrm{int} (X^\ast)$. Thanks to our assumptions on the quotient $X^\ast$ we can take $y_0^\ast, y_1^\ast \in X^\ast$ such that $x_0^\ast$ is the unique midpoint between them. This gives us in particular that there is a unique geodesic $\gamma^\ast_t$ from $y^\ast_0$ to $y^\ast_1$ and the curve $\eta^\ast_t:=\delta_{\gamma^\ast_t}$ is a constant speed geodesic of $(\mathbb{P}_2(X^\ast),\mathbb{W}_2)$ joining $\delta_{y_0^\ast}$ and $\delta_{y_1^{\ast}}$ (see \cite[p. 37]{AmbrosioGigli2013}).

Recall that we have a lift $\Lambda\colon \mathbb{P}_2(X)\to \mathbb{P}_2(X)$ defined in \eqref{G-invariant-lift-Dirac-delta} and consider the $G$-in\-va\-riant measures $\nu_{y_0},\nu_{y_1},\nu_{x_0}\in \mathbb{P}^G(X)$. Now consider the curve $\eta_t:=\Lambda(\eta_t^\ast)$. Since we have that for any $s,t\in [0,1]$ 
\[
\mathbb{W}_2(\eta_s,\eta_t)=\mathbb{W}_2(\eta^\ast_s,\eta^\ast_t)=|s-t|\dis^\ast(y^\ast_0,y^\ast_1),
\]
we conclude that $\eta_t$ is a geodesic in $(\mathbb{P}_2(X),\mathbb{W}_2)$. Then by \th\ref{T: Dynamical optimal transport} there exists a dynamical optimal transport $\eta \in \Opt\Geo (\nu_{y_0},\nu_{y_1})$, and by \th\ref{lemma.horizontaltransport} it has to be supported on horizontal geodesics. 

We point out that by construction we have that $(\pi\circ e_{\frac{1}{2}})_{\#}\eta = \delta_{x^\ast_0}$.
Hence, we have a horizontal geodesic $\gamma \in \supp \eta$ such that $\gamma_{\frac{1}{2}}\in \pi^{-1}(x^\ast_0)$. Observe that there exists $g \in G$ such that $g\gamma_{\frac{1}{2}}=x_0$, so  $g\gamma$ is the sought horizontal geodesic.

Now, to prove the second part of the proposition we will proceed by contradiction. Suppose that we have two horizontal geodesics $\gamma^1,\gamma^2 \in \Geo(X)$ and some $t_0\in (0,1) $ such that 
\begin{align*}
    \gamma^1_t=\gamma^2_t &\text{ for all } t \in [0,t_0],\\
    \gamma^1_s \neq \gamma^2_s &\text{ for some } s \in (t_0,1].\\
\end{align*}
WLOG we can assume that $\gamma^{1\ast}_t= \pi(\gamma^1)_t = \pi(\gamma^2)_t= \gamma^{2\ast}_t$ for all $t \in [0,1]$ and that there is a unique geodesic $\gamma^\ast$ between $\gamma^{1\ast}_0$ and $\gamma^{1\ast}_1$,  and $\gamma^\ast_t=\gamma^{1\ast}_t=\gamma^{2\ast}_t$ for all $t\in [0,1]$. 

We can easily find isometries 
\begin{align*}
    F^1\colon& \gamma^{\ast}[0,1]\rightarrow \gamma^1[0,1]\\ 
    F^2\colon& \gamma^{\ast}[0,1]\rightarrow \gamma^2[0,1]\\ 
\end{align*}
that by \th\ref{prop.induced-isometry} induce isometries of their respective Wasserstein spaces and correspondences between (dynamical) optimal transports. Observe that $(F^{i})^{-1}=\pi|_{\gamma^i[0,1]}$. Furthermore, the maps $F^i \colon \gamma^{\ast}[0,1]\rightarrow \gamma^i[0,1]$ can also be used to induce maps between the corresponding spaces of geodesics. That is, we have for $i= 1,2$ maps 
$F^i\colon \Geo(\gamma^\ast [0,1])\rightarrow \Geo (\gamma^i[0,1])$ given by $(F^i(\sigma^\ast))_t:= F^i(\sigma^\ast_t)$ for all $t \in [0,1]$ and $\sigma^\ast \in \Geo(\gamma^\ast[0,1])$. Put in other words, we have that $e_t(F^i(\sigma^\ast))= (F^i\circ e_t)(\sigma^\ast)$ for all $t \in [0,1]$ and $\sigma^\ast \in \Geo(\gamma^\ast[0,1])$.

We set 
\[
\mu_0 = \frac{\mathfrak{m}^\ast \llcorner\gamma^{1\ast}[0, t_0/2]}{\mathfrak{m}^\ast(\gamma^{1\ast}[0,t_0/2])}, \quad  \mu_1= \frac{\mathfrak{m}^\ast\llcorner\gamma^{1\ast} [t_0,1]}{\mathfrak{m}^\ast(\gamma^{1\ast}[t_0,1])}\in \mathbb{P}_2(X^\ast).
\]
We observe for $i=1,2$ and $j=0,1$ we have  $F^i_{\#}\mu_j\in \mathbb{P}(X)$, defined simply for a Borel subset $A\subset X$ by $F^i_{\#}\mu_j(A):=F^i_{\#}\mu_j(A\cap \gamma^i[0,1])$. From the definition of $F^i_{\#}\mu_j$, the fact that $\gamma^i[0,1]\subset X$ is a totally convex subset of $X$ since it is a minimizing geodesic, and $F^i$ is an isometry, one can check that $F^i_{\#}\mu_j\in \mathbb{P}_2(X)$. We also point out that 
\[
F^1_{\#}\mu_0 = F^2_{\#}\mu_0 \text{ since } \gamma^1_t = \gamma^2_t \text{ for all } t \in [0,t_0].
\]

Let $\eta^\ast\in \Opt\Geo (\mu_0,\mu_1)$, and consider a geodesic $\sigma^\ast\in \supp \eta^\ast$. By our hypothesis on $X^\ast$ notice that its endpoints belong to $\gamma^\ast [0,1]$. Thus $\sigma^\ast[0,1]\subset\gamma^\ast[0,1]$. Hence 
$\eta^\ast \in \mathbb{P}_2(\Geo(\gamma^\ast [0,1]))$.
Let $\eta^i = F^i_{\#}\eta^\ast\in \mathbb{P}(\mathrm{Geo}(\gamma^i[0,1]))$ for $i=1,2$. Since $\gamma^i[0,1]$ is totally convex in $X$, then we have $\mathrm{Geo}(\gamma^i[0,1])\subset \mathrm{Geo}(X)$. Thus we can define $F^i_{\#}\eta^\ast$ on $\mathrm{Geo}(X)$ by $F^i_{\#}\eta^\ast(\Gamma):=F^i_{\#}\eta^\ast(\Gamma\cap \mathrm{Geo}(\gamma^i[0,1]))$ for  $\Gamma\subset \mathrm{Geo}(X)$ a Borel subset. 

We claim that the measure $\frac{1}{2}\eta^1+\frac{1}{2}\eta^2$ is a dynamical optimal transport for the measures $F^1_{\#}\mu_0$ and $\frac{1}{2}F^1_{\#}\mu_1+\frac{1}{2}F^2_{\#}\mu_1$. Recalling that $e_t\circ F^i = F^i\circ e_t$ for all $t \in [0,1]$, we have that 
\begin{align*}
  e_{t\#}\left(\frac{1}{2}\eta^1+\frac{1}{2}\eta^2 \right)  &= \frac{1}{2}e_{t\#}\eta^1+\frac{1}{2}e_{t\#}\eta^2 \\
  &= \frac{1}{2}(e_t\circ F^1)_\#\eta^\ast+\frac{1}{2}(e_t\circ F^2)_\#\eta^\ast\\
  &= \frac{1}{2}(F^1\circ e_t)_\#\eta^\ast+\frac{1}{2}(F^2\circ e_t)_\#\eta^\ast\\
  &= \frac{1}{2}F^1_{\#}(e_{t\#}\eta^\ast)+\frac{1}{2}F^2_{\#}(e_{t\#}\eta^\ast).
\end{align*}
From this we conclude that $t \mapsto e_{t\#}\left(\frac{1}{2}\eta^1+\frac{1}{2}\eta^2 \right)$ is a curve joining $F^1_{\#}\mu_0$ with $\frac{1}{2}F^1_{\#}\mu_1+\frac{1}{2}F^2_{\#}\mu_1$.
Similarly we can check that $(e_0,e_1)_{\#}(\frac{1}{2}\eta^1+\frac{1}{2}\eta^2)\in \mathbb{P}(X \times X)$ has marginals $F^1_{\#}\mu_0$ and $\frac{1}{2}F^1_{\#}\mu_1+\frac{1}{2}F^2_{\#}\mu_1$.

Let us observe a couple of things:
\begin{itemize}
    \item $\pi_{\#}F^i_{\#}\eta^\ast=\eta^\ast$ so $\pi_{\#}(\frac{1}{2}\eta^1+\frac{1}{2}\eta^2)=\eta^\ast$,
    \item $\mathbb{W}_2^2(\mu_0,\mu_1)=\mathbb{W}_2^2(\mu_0,\pi_{\#}(\frac{1}{2}F^1_{\#}\mu_1+\frac{1}{2}F^2_{\#}\mu_1))\leq \mathbb{W}^2_2(F^1_{\#}\mu_0,\frac{1}{2}F^1_{\#}\mu_1+\frac{1}{2}F^2_{\#}\mu_1)$.
\end{itemize}
So when we compute the cost, we obtain
\begin{align*}
\int_{X\times X} \dis^2(x,y)d(e_0,e_1)_{\#}\left(\frac{1}{2}\eta^1+\frac{1}{2}\eta^2\right)(x,y) =& \int_{\mathrm{Geo}(X)} \dis^2(\sigma_0,\sigma_1)d\left(\frac{1}{2}F^1_{\#}\eta^\ast+\frac{1}{2}F^2_{\#}\eta^\ast\right)(\sigma)\\
=& \frac{1}{2}\int_{\mathrm{Geo}(X)} \dis^2(\sigma_0,\sigma_1)dF^{1}_{\#}\eta^\ast(\sigma)\\
&+\frac{1}{2}\int_{\mathrm{Geo}(X)} \dis^2(\sigma_0),\sigma_1))dF^2_{\#}\eta^\ast(\sigma)\\
=& \frac{1}{2}\int_{\mathrm{Geo}(\gamma^1[0,1])} \dis^2(\sigma_0,\sigma_1)dF^{1}_{\#}\eta^\ast(\sigma)\\
&+\frac{1}{2}\int_{\mathrm{Geo}(\gamma^2[0,1])} \dis^2(\tilde{\sigma}_0),\tilde{\sigma}_1))dF^2_{\#}\eta^\ast(\tilde{\sigma})\\
=& \frac{1}{2}\int_{\mathrm{Geo}(\gamma[0,1])} \dis^2(F^1(\sigma^\ast_0),F^1(\sigma^\ast_1))d\eta^\ast(\sigma^\ast)\\
&+\frac{1}{2}\int_{\mathrm{Geo}(\gamma[0,1])} \dis^2(F^2(\tilde{\sigma}^\ast_0),F^2(\tilde{\sigma}^\ast_1))d\eta^\ast(\tilde{\sigma}^\ast)\\
=& \frac{1}{2}\int_{\mathrm{Geo}(\gamma[0,1])} \dis^{\ast2}(\sigma^\ast_0,\sigma^\ast_1)d\eta^\ast(\sigma^\ast)\\
&+\frac{1}{2}\int_{\mathrm{Geo}(\gamma[0,1])} \dis^{\ast 2}(\tilde{\sigma}^\ast_0,\tilde{\sigma}^\ast_1)d\eta^\ast(\tilde{\sigma}^\ast)\\
=&\int_{\mathrm{Geo}(\gamma[0,1])} \dis^{\ast 2}(\sigma^\ast_0,\sigma^\ast_1)d\eta^\ast(\sigma^\ast)\\
=&\int_{\mathrm{Geo}(X^\ast} \dis^{\ast 2}(\sigma^\ast_0,\sigma^\ast_1)d\eta^\ast(\sigma^\ast)\\
=& \mathbb{W}^2_2(\mu_0,\mu_1).
\end{align*}
Therefore we have that $\mathbb{W}_2^2(\mu_0,\pi_{\#}(\frac{1}{2}F^1_{\#}\mu_1+\frac{1}{2}F^2_{\#}\mu_1))= \mathbb{W}^2_2(F^1_{\#}\mu_0,\frac{1}{2}F^1_{\#}\mu_1+\frac{1}{2}F^2_{\#}\mu_1)$ and that $\frac{1}{2}\eta^1+\frac{1}{2}\eta^2$ is optimal.

Let $\Gamma^\ast \subset \Geo(\gamma^\ast[0,1])$ be such that $\eta^\ast(\Gamma^\ast)=1$. Then $ \Gamma^1 = F^1\Gamma^\ast\subset \Geo(\gamma^1[0,1]), \Gamma^2 = F^2\Gamma^\ast \subset \Geo(\gamma^2[0,1])$ satisfy $\eta^i(\Gamma^i)=1$ for $i=1,2$, and so we have that $(\frac{1}{2}\eta^1+\frac{1}{2}\eta^2)(\Gamma^1\cup\Gamma^2)=1$.  

We show next that the measure $\frac{1}{2}\eta^1+\frac{1}{2}\eta^2$ gives positive mass to a set of branching geodesics. Notice that the set $\lbrace t \in [0,1]\,|\, \gamma^1_t \neq \gamma^2_t \rbrace\subset (t_0,1]$ is open: consider $t_0<s_0\leq 1$ such that $\gamma^1_{s_0}\neq \gamma^2_{s_0}$, and assume that for any $\varepsilon>0$ such that $t_0<s_0-\varepsilon<s_0+\varepsilon\leq 1$ there exists $t_\varepsilon\in (s_0-\varepsilon,s_0+\varepsilon)$ with $\gamma^1_{t_\varepsilon}=\gamma^2_{t_\varepsilon}$. Then by taking $\varepsilon=1/n$ we obtain a sequence $\{t_n\}_{n\in \N}\subset(t_0,1]$ converging to $s_0$, such that $\dis(\gamma^1_{t_n},\gamma^2_{t_n})=0$. But by continuity of $\gamma^1$, $\gamma^2$ and $\dis$  this implies that $\dis(\gamma^1_{s_0},\gamma^2_{s_0})=0$ which is a contradiction. Thus, we can find $s_0 \in (t_0,1]$ and $\epsilon > 0$ such that $(s_0-\epsilon,s_0+\epsilon) \subset\lbrace t \in [0,1]\,|\, \gamma^1_t \neq \gamma^2_t  \rbrace. $ As the set $\gamma^\ast(s_0-\epsilon,s_0+\epsilon)$ is open we have that $\mu_1(\gamma^\ast(s_0-\epsilon,s_0+\epsilon))>0$. 
Then by defining $\Gamma^\ast_{s_0,\epsilon} = e^{-1}_1(\gamma^\ast(s_0-\epsilon,s_0+\epsilon))\cap\Gamma^\ast, $ we have
\[
\eta^\ast (\Gamma^\ast_{s_0,\epsilon})= \eta^\ast(e^{-1}_1(\gamma^\ast(s_0-\epsilon,s_0+\epsilon))\cap\Gamma^\ast)= \mu_1(\gamma^\ast(s_0-\epsilon,s_0+\epsilon))>0.
\]
So we just notice that $F^1\Gamma^\ast_{s_0,\epsilon}\subset \Geo(\gamma^1[0,1]), F^2\Gamma^\ast_{s_0,\epsilon}\subset\Geo(\gamma^2[0,1])$ which give us that 
\[
\eta^1(F^1\Gamma^\ast_{s_0,\epsilon})= F^1_{\#}\eta^\ast(F^1\Gamma^\ast_{s_0,\epsilon})>0,\quad \eta^2(F^2\Gamma^\ast_{s_0,\epsilon})= F^2_{\#}\eta^\ast(F^2\Gamma^\ast_{s_0,\epsilon})>0.
\]
We claim now that $(\frac{1}{2}\eta^1+\frac{1}{2}\eta^2)$ gives positive mass to the set $\Gamma_{branch}= F^1\Gamma^\ast_{s_0,\epsilon}\cup F^2\Gamma^\ast_{s_0,\epsilon}\subset \Geo(X)$ and that it consists of branching geodesics.
The first part is clear, so let us just deal with the second part of the claim. 
Consider $\sigma^1\in F^1\Gamma^\ast_{s_0,\epsilon}$,  and notice that $\sigma^\ast = \pi \circ \sigma^1 \in \Gamma^\ast_{s_0,\epsilon}$  and so
$F^2\sigma^\ast \in F^2\Gamma^\ast_{s_0,\epsilon}$. 
Since $\sigma^\ast_0 \in \gamma^{1\ast}[0,t_0/2] $ and $\gamma^1_t=\gamma^2_t$ for all $t\in [0,t_0]$,
we have that

\begin{align*}
    (F^2\sigma^\ast)_t &= \sigma^1_t \text{ for } t\ll 1,\\
    (F^2\sigma^\ast)_1 &\neq \sigma^1_1.
\end{align*}
Therefore the measure $(\frac{1}{2}\eta^1+\frac{1}{2}\eta^2)$ gives positive mass to the set of branching geodesics $\Gamma_{branch}$. 
Let $\mathbb{H}$ be the normalized Haar measure of $G$ and consider the measure
\[
\eta_G = \int_G g_{\#}(\frac{1}{2}\eta^1+\frac{1}{2}\eta^2)d\mathbb{H}(g) \in \mathbb{P}(\Geo(X))
\]
which is clearly $G-$invariant. Observe that 
\[
\pi_{\#}(e_{0\#}\eta_G)= \mu_0, \quad \pi_{\#}(e_{1\#}\eta_G)= \mu_1. 
\]
So by \eqref{theorem-lift-to-G-invariant-measures (1)} in \th\ref{theorem-lift-to-G-invariant-measures} we have that 
\[
e_{0\#}\eta_G = \Lambda(\mu_0), \quad e_{1\#}\eta_G = \Lambda(\mu_1).
\]
Observe that both measures are absolutely continuous with respect to $\mathfrak{m}$ by \eqref{theorem-lift-to-G-invariant-measures (2)} in \th\ref{theorem-lift-to-G-invariant-measures}. Now 
\begin{align*}
    \mathbb{W}^{2}(\Lambda(\mu_0),\Lambda(\mu_1))\leq\int \dis^2(x,y) d(e_0,e_1)_{\#}\eta_G &= \int_G\int \dis^2(\sigma_0,\sigma_1)d g(\frac{1}{2}\eta^1+\frac{1}{2}\eta^2)(\sigma)d\mathbb{H}(g)\\
    &= \int_G\int \dis^2(\sigma_0,\sigma_1)d (\frac{1}{2}\eta^1+\frac{1}{2}\eta^2)(\sigma)d\mathbb{H}(g)\\
    &= \mathbb{W}_2^2(\mu_0,\mu_1)= \mathbb{W}_2^2 (\Lambda(\mu_0),\Lambda(\mu_1)).
\end{align*}
Which yields that $\eta_G$ is optimal. Finally, observe that $\eta_G(\cup_{g \in G}g\Gamma_{branch})>0$. Since it is clear that $\cup_{g \in G}g\Gamma_{branch}$ consists of branching geodesics this gives us a contradiction to the fact that $X$ is essentially non-branching.
\end{proof}

\begin{rmk}
Notice that in the proof of \th\ref{proposition.nonbranchinghorizontalgeodesic} we needed to take the measure $\mu_0$ with its support in $\gamma^{1\ast}[0,t_0/2]$ in order to ensure that the geodesics involved in the transport from $\mu_0$ to $\mu_1$ do branch as they move along $\gamma^{1\ast}[0,t_0].$     
\end{rmk}

As a quick consequence of \th\ref{proposition.nonbranchinghorizontalgeodesic}, we get the following two lemmas:

\begin{lemma}[Kleiner's Lemma]\th\label{L: Kleiners Lemma}
Let $(X,\dis,\mathfrak{m})$ be an essentially non-branching space and $G$ a compact Lie group acting by isometries. Then for any horizontal geodesic $\gamma \in \Geo(X)$ we have that for all $t\in (0,1)$ the isotropy group $G_{\gamma_t}$ is a subgroup of both $G_{\gamma_0}$  and $G_{\gamma_1}$.   
\end{lemma}

\begin{proof}
Let $\gamma \in \Geo(X)$ be a horizontal geodesic. For $t \in (0,1)$ take some $g\in G_{\gamma_t}$ and assume that $g\gamma_1 \neq \gamma_1$. Clearly $g\gamma$ is again a horizontal geodesic, but by \th\ref{proposition.nonbranchinghorizontalgeodesic} such geodesics cannot branch. Hence $g\in G_{\gamma_1}$ and an analogous argument gives us that $g \in G_{\gamma_0}$.     
\end{proof}

\begin{lemma}[Principal isotropy Lemma]\th\label{L: Principal Isotropy Lemma}
Let $(X,\dis,\mathfrak{m})$ be an essentially non-branching space and $G$ a compact Lie group acting by isometries such that $(X^\ast,\dis^\ast)$ is isometric to $[-1,1]$, $[0,\infty)$, $\mathbb{R}$, or $\Sp^1$. Then the isotropy groups of points in orbits corresponding to the interior of $X^\ast$ are conjugated. 
\end{lemma}

\begin{proof}
Let $x,y\in X$ be points such that $x^\ast,y^\ast \in int X^\ast$. If $x=gy$, it is clear that  $G_{gy}= gG_y g^{-1}$. So we will assume that they belong to different orbits.

Let $\gamma^\ast \in \Geo(X^\ast)$ be a geodesic between $x^\ast$ and $y^\ast$. Take now $z^\ast = \gamma^\ast_{\frac{1}{2}}$ and notice that by our assumptions on $X^\ast$ there must be geodesics $\sigma^{x\ast}, \sigma^{y\ast}\in \Geo(X^\ast)$
 such that
 \[
x^\ast, z^\ast \in \sigma^{x\ast}((0,1));\, y^\ast, z^\ast \in \sigma^{y\ast}((0,1)). 
 \]
 That is, they are in the interior of the corresponding geodesics. We can find horizontal geodesics $\sigma^x,\sigma^y\in \Geo (X)$ such that
 \begin{align*}
\pi\circ \sigma^x &= \sigma^{x\ast},\, \pi\circ \sigma^y= \sigma^{y\ast},\\
x,z \in &\sigma^x((0,1)), \mbox{ and } \, y,gz \in \sigma^y((0,1)),
 \end{align*}
where $z,gz$ are points that satisfy
\[
\dis(z,x)= \dis(G(z),x), \, \dis(gz,y)= \dis(G(z),y).
\]
$\sigma^x$ being a horizontal geodesic, Kleiner's \th\ref{L: Kleiners Lemma} tells us that the isotropy of points in $\sigma^x((0,1))$ is some constant group $H$. Similarly for $\sigma^y$ we can find some group $K$ which is the isotropy of all points in $\sigma^y((0,1))$. By the observation made in the beginning of the proof it follows that $H=G_z=G_x$, $K= G_{gz}=G_y$. So $H$ and $K$ must be conjugated. As $x,y$ were chosen arbitrarily we have the claim.
\end{proof}

\begin{prop}\th\label{proposition.uniquerealiser}
Let $x_0, y\in X$ such that $y^\ast$ is in the interior of $X^\ast$. Then, there exists a unique $gx_0\in \pi^{-1}(x_0^\ast)$ such that 
\[
\dis(G(x_0) ,y) = \dis(gx_0,y).
\]
\end{prop}

\begin{proof}
Let $y\in X$ with $y^\ast \in int X^\ast$. Take a point $gx_0 \in \pi^{-1}(x_0^\ast) $ such that
\[
\dis(gx_0,y) = \dis(G(x_0), y).
\]
Since $y^\ast \in int X^\ast$ we can find an horizontal geodesic $\gamma \in \Geo(X)$ with $\gamma_0= gx_0$ and $\gamma_t = y$ for some $t \in (0,1)$. Assume now that we can find a different $hx_0$ in the orbit such that $\dis(hx_0,y)=  \dis(G(x_0),y)$.

Consider $\sigma \in \Geo(X)$ a horizontal geodesic joining $hx_0$ with $y$. By adequately concatenating $\sigma$ and $\gamma$ we can obtain a branching horizontal geodesic. This contradicts \th\ref{proposition.nonbranchinghorizontalgeodesic}. Thus we must have that $gx_0= hx_0$.
\end{proof}

We are now ready to prove the main result of this section:

\begin{proof}[Proof of \th\ref{MT: Slice Theorem Cohomogenity one}]
Take $x_0 \in X$ and define 
\[
S_{x_0}:= \lbrace y \in X \,|\, \dis(x_0,y)=\dis(G(x_0), y)  \rbrace\cap B_{\delta}(x_0).
\]
Where $\delta >0$ is chosen in the following way:
\begin{align*}
    \delta < \begin{cases}
        \dis^\ast(x_0^\ast, \partial X^\ast) &\text{if } x_0^\ast\in int X^\ast,\\
        \dis^\ast(x_0^\ast, \partial X^\ast - \lbrace x_0^\ast\rbrace) &\text{if } x_0^\ast \in \partial X^\ast.
    \end{cases}
\end{align*}

We will now prove each of the items in \th\ref{D: Slice}:

\begin{itemize}
     \item[(\ref{Slice-1})] Let $y_n, n \in \mathbb{N}$ be a converging sequence of points in $S_{x_0}$ such that the limit $y \in G(S_{x_0})$. As the function 
     $\dis(x_0,\cdot) \colon X \rightarrow [0,\infty)$ is continuous we have that  $\dis(x_0, y_n)\rightarrow \dis(x_0, y)$ as $n \rightarrow \infty$. 
     Also, notice that  
     $\dis(x_0,y_n)= \dis(G(x_0),y_n)= \dis^\ast(x_0^\ast,y_n^\ast) \rightarrow \dis^\ast(x_0^\ast,y^\ast)$ as $n\rightarrow \infty$.

     Therefore $\dis(x_0,y)= \dis^\ast(x_0^\ast,y^\ast)= \dis(G(x_0),y)$.

    \item[(\ref{Slice-2})] Just notice that 
    \[
    G(S_{x_0}) = \bigcup_{gx_0\in x_0^\ast}S_{gx_0} = G(B_\delta (gx_0))
    \]
     with the set in the last equality clearly being an open neighborhood of the orbit $x_0^\ast$.

     \item[(\ref{Slice-3})] Take $g \in G_{x_0}$ and $y \in S_{x_0}$. Then $\dis(x_0,y)= \dis(x_0, gy)$. Take $hx_0 \in \pi^{-1}(x_0^\ast)$ such that $\dis(hx_0,gy)= \dis(G(x_0),gy)$ and observe that
     \[
       \dis(G(x_0), y)= \dis(G(x_0), gy)= \dis(g^{-1}hx_0, y). 
     \]
     Since $y^\ast \in int X^\ast$ by \th\ref{proposition.uniquerealiser} this implies that $g^{-1}hx_0 = x_0$ so then $h\in G_{x_0}$ and $gy \in S_{x_0}$.

     \item[(\ref{Slice-4})] First, observe that for any $g\in G$ $gS_{x_0}= S_{gx_0}$. Now, take a point $y\in S_{gx_0}\cap S_{x_0}$ and so
     \[
    \dis(gx_0, y)= \dis(x_0, y)= \dis(G(x_0), y).
     \]
     Again, by  \th\ref{proposition.uniquerealiser}  it follows that $gx_0 = x_0$ which means that $g\in G_{x_0}$.
\end{itemize}

As all of the conditions are satisfied, we conclude that $S_{x_0}$ is indeed a slice at $x_0$.
\end{proof}

\begin{rmk}\th\label{R: Slice is a topological cone}
Due to \th\ref{proposition.uniquerealiser}, we have that  $S_{x_0}$ is a union of horizontal geodesics starting which have $x_0$ as the single common point. Thus $S_{x_0}$ is homeomorphic to a cone.
\end{rmk}

If we allow the dimension of the quotient space $X^\ast$ to be larger than one, we can still get a slice but we require stronger assumptions on both $X$ and $X^\ast$.

\begin{theorem}
Let $G$ act by isometries on a non-branching space $(X,\dis,\mathfrak{m})$ and assume that $(X^\ast,\dis^\ast)$ is an Alexandrov space. Moreover, assume that there exists $\mathcal{M}\subseteq X^\ast$ an open set such that for $x_0^\ast \in \mathcal{M}$ the distance of the cut locus $\mathcal{C}(x_0^\ast)$ to $x_0^\ast$ is positive. Then, for every point $x\in x_0^\ast$ we have a slice. 
\end{theorem}

\begin{proof}
 Let $x_0^\ast$ and take $0< \delta < \dis^\ast (x_0^\ast, \mathcal{C}(x_0^\ast))$. For an orbit $y^\ast\in B_\delta(x_0^\ast)$  take some $y \in \pi^{-1}(y^\ast)$. Now, there exists some $gx_0\in \pi^{-1}(x_0^\ast)$ such that $\dis(gx_0,y)= \dis(G(x_0), y)$.   

 Noting that our space is non-branching and that we can find a geodesic $\gamma \in \Geo (X)$ such that $\gamma_0 = gx_0$ and $\gamma_t= y$ for some $t \in (0,1)$ we can conclude that 
 \[
\arg\min \lbrace \dis(gx_0,y)\,|\, gx_0 \in \pi^{-1}(x_0^\ast) \rbrace = \lbrace gx_0 \rbrace.
 \]
 Furthermore, by arguing as in \th\ref{lemma.projection.horizontalgeo} we can check that $\gamma$ is a horizontal geodesic.
 Once we have this we can define 
 \[
S_{x_0}= \lbrace y\in X\,|\, \dis(x_0,y)= \dis(G(x_0),y) \rbrace \cap B_{\delta}(x_0)
 \]
 and proceed as in the proof of \th\ref{MT: Slice Theorem Cohomogenity one} to conclude that it is indeed a slice.
\end{proof}

\begin{rmk}
    In the previous Theorem we needed to impose further restrictions on the geodesics of $X$. This is because if the dimension of the quotient space $X^\ast$ is larger than or equal to $2$, it could happen that \th\ref{proposition.nonbranchinghorizontalgeodesic} does not hold. This is because some horizontal branching geodesic $\gamma$ could very well be undetected as there is no guarantee that taking $\lbrace g\gamma \mid g \in G \rbrace$ would yield something of positive measure.  
\end{rmk}

\begin{rmk}
    The condition asking $x_0^\ast$ to be isolated from its cut locus $\mathcal{C}(x_0^\ast)$ is a not necessary for the existence of the slice. For example if $D\subset \R^2$ is a convex domain, then we consider $X$ to be its double which is an Alexandrov space with curvature $\geq 0$. For $x\in \partial D\subset X$, we have that the cut locus is $\mathcal{C}(x) = \partial D\setminus\{x\}$ (see \cite[Section 1, Example 2]{ShiohamaTanaka1996}). Thus $x$ is not isolated from its cut locus. Considering $G$ any compact Lie group with a left-invariant metric, the space $Y:=G\times X$ with the product metric is an Alexandrov space, whose orbit space is $X$, and for any point $(g,x)\in G\times X$, $S_{(g,x)}:= \{g\}\times X$ is a slice.
\end{rmk}

\subsection{Cohomogeneity one Group diagrams}\label{SS: Group Diagram}

We now show a relation between the isotropy subgroups for an action of a compact Lie group $G$ on a geodesic space $(X,\dis)$ by isometries. From this relation we obtain a group diagram for actions of cohomogeneity one by isometries on  essentially non-branching spaces. Given an action of the compact Lie group $G$ on a metric space $(X,\dis)$ such that $(X^\ast,\dis^\ast)$ is homoemorphic to $[-1,1]$, a \emph{group diagram} is a tuple $(G,K_-,K_+,H)$ where $H,K_-,K_+<G$ such that the following diagram of inclusions holds:
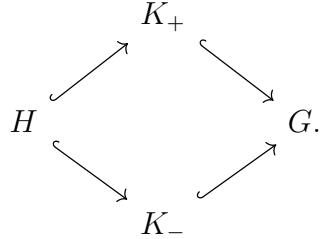
\begin{center}
    \begin{tikzcd}
        & K_+ \arrow[dr,hook] & \\
       H \arrow[ur,hook] \arrow[dr,hook] & & G.\\
        & K_- \arrow[ur,hook]&
    \end{tikzcd}
\end{center}

\begin{theorem}\th\label{T: Principal isotropy is subgroup of singular isotropy}
Let $(X,\dis, \m)$ be an essentially non-branching space, and $G$ a compact Lie group acting by isometries, such that $(X^\ast,\dis^\ast)$ is isometric to $[-1,1]$. Then we have three isotropy subgroups $H$, $K_-$, and $K_+$, whose isotropy types correspond to the orbits in the interior $(-1,1)$, and the end points $\{-1\}$ and $\{+1\}$ respectively. Moreover, $H$ is a subgroup of both $K_-$ and $K_+$.
\end{theorem}

\begin{proof}
We fix $\gamma\colon [0,1]\to X$ a geodesic realizing the distance between the orbits, guaranteed to exist by \th\ref{proposition.nonbranchinghorizontalgeodesic}. Then by Kleiner's \th\ref{L: Kleiners Lemma} the conclusions follow, by taking $H:= G_{\gamma_0}$, and $K_\pm= G_{\gamma_{\pm 1}}$.
\end{proof}

The tuple $(G,H,K_+,K_-)$ obtained in \th\ref{T: Principal isotropy is subgroup of singular isotropy} depends a priori on a choice of the geodesic. In the following lemma we show how two tuples associated to two different geodesics relate to each other.

\begin{lemma}\th\label{L: relation of diagrams and curves}
Let $(X,\dis,\m)$ be an essentially non-branching space, and $G$ a compact Lie group acting by isometries, such that $(X^\ast,\dis^\ast)$ is isometric to $[-1,1]$. Let $\gamma^1$ and $\gamma^2$ be two geodesics realizing the distance between the orbits, and $(G,H^1,K^1_+,K^1_-)$ and $(G,H^2,K^2_+,K^2_-)$ their associated tuples. Then for some $g_0\in G$ and $a_\pm\in N(H^1)_0$ we have 
\[
(G,H^2,K_+^2,K_-^2)= (G,g_0H^1g_0^{-1},g_0 a_+ K_+^2 a_+^{-1} g_0^{-1}, g_0 a_- K_-^2 a_-^{-1} g_0^{-1}).
\]
Here $N(H)_0$ denotes the connected component of the identity in the normalizer of $H^1$
\end{lemma}

\begin{proof}
We consider $\{x_0^1\}= \pi^{-1}(1/2)\cap \gamma^1([0,1])$ and $\{x_0^{2}\}=\pi^{-1}(1/2)\cap \gamma^2([ 0,1])$. Then there exists $g_0\in G$, such that $gx^1_0 =x_0^2$. Then $G_{x_0^2}= g_0G_{x^1_0}g_0^{-1}$. Moreover, $\tilde{\gamma}^2:=g_0^{-1}\gamma^2$ is also a geodesic with $\tilde{\gamma}^2(0) = x_0^1$. Moreover, there exists $a\colon [0,1]\to G$ such that $a_t\gamma^1_t = \tilde{\gamma}^2_t$. Observe that $a_{1/2}= e$, and thus we have that $G_{\tilde{\gamma}^2_t} = a_tG_{\gamma^1_t}a_t^{-1}$. But by \th\ref{L: Kleiners Lemma} we have that $G_{\gamma^1_t} = G_{\gamma^1_{1/2}}= G_{x_0^1} = G_{\tilde{\gamma}^2_{1/2}} = G_{\tilde{\gamma}^2_t}$ for all $t \in (0,1)$. Writing $H:= G_{x_0^1}$ we conclude that $a_t\in N(H)$ for all $t\in [0,1]$. Thus by setting $a_\pm:= a_{\pm 1}$ we have that 
\[
\tilde{K}_\pm^{2}= G_{\tilde{\gamma}^2_{\pm 1}} = a_\pm G_{\gamma^1_{\pm 1}} a_\pm^{-1} = a_\pm K_\pm^{1}a_\pm^{-1}.
\]
Since $\tilde{\gamma}^2_\pm = g_0^{-1} \gamma^2_\pm$, we have
\[
\tilde{K}_\pm^2= G_{\tilde{\gamma}^2_{\pm 1}} = g_0^{-1}G_{\gamma^2_{\pm 1}}g_0= g_0^{-1}K_\pm^2 g_0.
\]
Thus we conclude that
\[
(G,H^2,K_+^2,K_-^2)= (G,g_0H^1g_0^{-1},g_0 a_+ K_+^2 a_+^{-1} g_0^{-1}, g_0 a_- K_-^2 a_-^{-1} g_0^{-1}).
\]
\end{proof}

Recall that given two continuous actions of $G$ by homeomorphism on two topological spaces $X_1$ and $X_2$, the spaces $X_1$ and $X_2$ are \emph{equivariantly homeomorphic} if there exists $\phi\colon X_1\to X_2$ an homeomorphism such that for any $x\in X_1$, $y\in X_2$ and $g\in G$ we have
\[
\phi(gx) = g\phi(x),\quad\mbox{and}\quad \phi^{-1}(gy) = g\phi^{-1}(y).
\]

\begin{theorem}
Let $G$ be a compact Lie group acting on two essentially non-branching  spaces $(X_1,\dis_1,\m_1)$ and $(X_2,\dis_2,\m_2)$ by isometries such that $(X_1^\ast,\dis_1^\ast)$ and $(X_2^\ast,\dis_2^\ast)$ are isometric to $[-1,1]$. Moreover, assume that we have associated the tuples $(G,H^1,K^1_+,K_-^1)$ and $(G,H^2,K^2_+,K_-^2)$ respectively. If the spaces are equivariantly homeomorphic then there exists $g_0\in G$ and $a_\pm\in N(H^1)_0$ such that $H^2= g_0 H^1g_0^{-1}$ and either $K_\pm^2 = g_0a_\pm K_\pm^1 a_\pm^{-1} g_0^{-1}$ or $K_\pm^2 = g_0a_\pm K_{\mp}^1 a_\pm^{-1} g_0^{-1}$.
\end{theorem}

\begin{proof}
We consider  the metric $\phi^\ast(\dis_2)\colon X_1\times X_1$ given by $\phi^\ast(\dis_2)(x,y) = \dis_2(\phi(x),\phi(y))$. Then the geodesic $\gamma^2$ associated to the tuple $(G,H^2,K^2_+,K_-^2)$  induces a geodesic $\tilde{\gamma}^2$ on $X_1$ with same tuple $(G,H^2,K^2_+,K_-^2)$, if $\phi^{-1}$ preserves the orientation of $\gamma^2$, or tuple $(G,H^2,K^2_-,K_+^2)$, if $\phi^{-1}$ reversed  the orientation of $\gamma^2$. Following the proof of \th\ref{L: relation of diagrams and curves} we obtain our desired conclusions. 
\end{proof}

\begin{theorem}
Let $(X,\dis,\m)$ be an essentially non-branching space, and assume that a compact Lie group $G$ acts by isometries in such a way that $X^\ast$ is isometric to $[-1,1]$. Moreover, assume that we have a tuple $(G,H,G_+,G_-)$ associated to the action (see Section~\ref{SS: Group Diagram}). Then by considering a new tuple $(G,H_1,K_+,K_-)$ with $H_1= g_0Hg_0$, $K_\pm=g_0a_\pm G_\pm a_\pm^{-1} g_0$ or $K_\pm=g_0a_\pm G_{\pm} a_\pm^{-1} g_0$ for $g_0\in G$ and $a_\pm \in N(H)_0$, we obtain a space $X_2$ which is equivariantly homeomorphic to $X$.
\end{theorem}

\begin{proof}
We fix the end points $x_\pm\in X$ of the geodesic associated to the diagram. Recall the equivalent notion of slice given by \th\ref{T: equivalence slice}. Thus we have $G_\pm= G_{x_\pm}$. The fiber bundles $G \times_{G_\pm} S_{x_\pm}\to G(x_\pm) \cong G/G_\pm$ are associated to the principal bundles $G_\pm\to G\to G/G_\pm$. 

Without loss of generality we assume that $K_\pm = g_0 a_\pm G_{\pm} a_\pm^{-1} g_0$. Then the principal bundles $G_\pm\to G\to G/G_\pm$ are equivariantly equivalent to the principal bundles $K_\pm\to G\to G/K_\pm$. Thus the associated fiber bundles $G \times_{K_\pm} S_{(g_0a_\pm\cdot x_\pm)}\to  \cong G/K_\pm$ are equivarianlty equivalent to $G \times_{G_\pm} S_{x_\pm}\to G/K_\pm$. Also recall that $X$ is the union of $G \times_{G_-} S_{x_-}$ and $G \times_{G_+} S_{x_+}$ via their boundary. Thus  we obtain that $X$ is equivariantly homeomorphic to $Y$ the union of $G \times_{K_-} S_{x_-}$ and $G \times_{K_+} S_{x_+}$ via the boundary. By pushing the metric we can find a geodesic in $Y$ realizing the distance between the orbits and with associated tuple $(G,H_1,K_+,K_-)$.
\end{proof}

\section{Geometry of the slice}
\label{S: GEOMETRY-OF-THE-SLICE}

In this section we show that for a compact non-collapsed $\RCD$-space $(X,\dis,\m)$ with an isometric and measure preserving action of cohomogeneity $1$ by a compact Lie group $G$, the slice $S_x$ through $x\in X$ we presented in Section~\ref{S: Slice THeorem RCD} admits an $\RCD$-structure. Moreover this structure on $S_x$ is given by taking the cone over a homogeneous space. We end the section by collecting all the results and combining them with the conclusions in \cite{GalazGarciaKellMondinoSosa2018} to obtain a topological structural result. 

\subsection{Induced infinitesimal actions}
\vspace*{1em}
To obtain the desired conclusions about the slices, we will work at the level of tangent cones and obtain an ``induced infinitesimal action''. Roughly speaking, we will consider an induced action on the tangent cone in the following way: Recall that given a sequence of pointed metric measure spaces $(X_n,\dis_n, \m_n, x_n)$ converging to $(Y,\dis_Y, \m_Y,0_Y)$ with respect to the pointed measured Gromov-Hausdorff topology and a compact Lie group $G$ acting on $X_n$ by measure-preserving isometries,  then \th\ref{T: p-G-H implies eq-p-G-H} guarantees the existence of  a closed subgroup $\overline{G}$ of $\Iso Y$ of measure preserving isometries such that $(X_n,\dis_n, \m_n, G,x_n)$ converges to $(Y,\dis_Y, \m_Y, \overline{G},0_Y)$ with respect to  the equivariant pointed measured Gromov-Hausdorff topology. In this setting, \cite[Theorem 3.1]{Harvey2016} (see also \cite[Appendix]{MazurRongWang2008}) guarantees the existence of Lie group homomorphisms $\phi_n\colon G\to \overline{G}$, when $\overline{G}$ is compact. Whenever this situation arises, we say that $\overline{G}$ is \emph{the limit group induced by $G$}. For simplicity we omit the dependence on the sequence $X_n$ and the limit space $Y$ as these are clear from the context.  Let us point out that if $G$ acts effectively, then so does $\overline{G}$. 

In this section we consider the induced actions of $G$ on blow-up sequences around points of a non-collapsed $\RCD$ space $(X,\dis,\m)$ and take the induced limit group $\overline{G}$ that acts on the tangent cones at each point. We analyze the action of $\overline{G}$ on $Y$ and its relation to the original action. 

In the following lemma we extend some results from \cite{Harvey2016}, \cite{Santos2020} about the existence of a ``nice'' group action on the limit of a convergent sequence of non-collapsed $\RCD$-spaces equipped with isometric measure preserving group actions.

\begin{lemma}\th\label{L: pm-GH convergences implies eq-GH convergence}
Let $G$ be a compact Lie group acting on $(X_n,\dis_n,\m_n)$ effectively by measure preserving isometries, where $(X_n,\dis_n,\m_n)$ are  $\RCD(K,N)$-spaces. Let $(Y,\dis_Y,\m_Y)$ be a metric measure space such that 
\[
(X_n,\dis_n,\m_n,x_n)\to (Y,\dis_Y,\m_Y,0_Y)
\]
in the pointed measured Gromov-Hausdorff topology. Assume that the induced limit group $G_Y$, given by \th\ref{T: p-G-H implies eq-p-G-H}, acting effectively by measure preserving isometries on $Y$ is compact, and denote by $\phi_n\colon G\to G_Y$ the Lie group homomorphisms which demonstrate the equivariant pointed Gromov-Hausdorff convergence. If $X_n$ and $Y$ have constant essential dimension, then $\phi_n$ is injective.
\end{lemma}

\begin{proof}
Let $K_n = \kernel(\phi_n)\subset G$. Then we have that $(X_i,\dis_i,\m_i,x_i,K_i)$ converges in the pointed-equivariant-Gromov-Hausdorff sense to $(Y,\dis_Y,\m_Y,0_Y,\{e\})$. As pointed out in \cite[Proof of Theorem 3.1]{Harvey2016}, by possibly passing to a subsequence of $(X_i,\dis_i,\m_i,x_i)$ we can assume that the equivariant-Gromov-Hausdorff distance between $(X_i,\dis_i,\m_i,x_i,K_i)$ and $(Y,\dis_Y,\m_Y,0_Y, \{e\})$ is less than $1/i$, and $K_i(1/i)=K_i$ for sufficiently large $i$. Then for sufficiently large $i$ we have maps $f_i\colon X_i\to Y$, group morphisms $\phi_i\colon K_i\to \{e\}$, $\psi_n\colon \{e\}\to K_i$, such that $f_i(x_i) = 0_Y$, the $1/i$-neighborhood of $f_i(B_{i}^{X_i}(x_i))$ contains $B_i^Y(0_Y)$, for all $y_i,z_i\in B_{i}^{X_i}(x_i)$ we have 
\[
|\dis_i(y_i,z_i)-\dis_Y(f_i(y_i),f_i(z_i))|\leq \frac{1}{i},
\]
and for all $\gamma_i\in K_i$, and all $z_i,\gamma_iz_i\in B_{i}^{X_i}(x_i)$ we have 
\[
\dis_Y(f_i(\gamma_iz_i),\phi_i(\gamma_i)f_i(z_i))\leq \frac{1}{i}.
\] 
But since $\phi_i(\gamma_i) = e$, then that for all $z_i,\gamma_iz_i\in B_{i}^{X_i}(x_i)$ we have 
\[
\dis_Y(f_i(\gamma_iz_i),f_i(z_i))\leq \frac{1}{i}.
\]
This implies that for all $\gamma_i\in K_i$ we have that  $\dis_Y(f_i(\gamma_ix_i),f_i(x_i))\leq 1/i$. Combining this together with the fact that $|\dis_i(\gamma_ix_i,x_i)-\dis_Y(f_i(\gamma_ix_i),f_i(x_i))|\leq 1/i$ we conclude that
\[
\dis_i(\gamma_ix_i,x_i)\leq \frac{2}{i}.
\]
This implies that $D_i:=\diam(K_i(x_i))\leq 2/i<i$. Let $\delta_i$ be such that $\delta_i+D_i<i$. We point out that we can choose $\delta_i$ to be a strictly increasing sequence since $D_i$ is a decreasing sequence. Now observe that for $z_i\in B_{\delta_i}^{X_i}(x_i)$,  by the triangle inequality we have for any $\gamma_i\in K_i$:
\[
\dis_i(\gamma_iz_i,x_i)\leq \dis_i(\gamma_iz_i,\gamma_ix_i)+\dis_i(\gamma_ix_i,x_i) = \dis_i(z_i,x_i)+\dis_i(\gamma_ix_i,x_i) \leq \delta_i+D_i< i.
\] 
That is for all $z_i\in B_{\delta_i}^{X_i}(x_i)$ we have $\gamma_iz_i\in B_{i}^{X_i}(x_i)$ for all $\gamma_i\in K_i$. In particular, for all $\gamma_i\in K_i$ and $z_i \in B_{\delta_i}^{X_i}(x_i)$ we have
\begin{equation}\label{EQ: peqGH1}
\dis_Y(f_i(\gamma_iz_i),f_i(z_i))\leq \frac{1}{i}.
\end{equation}
Now fix $\varepsilon>0$, take $i$ sufficiently large so that $1/i<\varepsilon$ and $K_i=K_i(i)$. Consider $\gamma_i\in K_i$ and $z_i\in B_{\delta_i}^{X_i}(x_i)$. Then we have $|\dis_i(\gamma_i z_i,z_i)-\dis_Y(f_i(\gamma_i z_i),f_i(z_i))|\leq \frac{1}{i}$. Combining this with \eqref{EQ: peqGH1} we conclude that 
\[
\dis_i(\gamma_i z_i,z_i)<\frac{2}{i}.
\]
Setting $r_i=\delta_i$ for sufficiently large $i$, we conclude that for all $z_i\in B_{r_i}^{X_i}(x_i)$ we have
\[
\frac{1}{r_i}+ \dis_i(\gamma_i z_i,z_i)< \frac{1}{r_i}+\varepsilon.
\]
Since $\{r_i\}_{i\in \N}$ is a strictly increasing growing sequence, we conclude that for sufficiently large $i$ we have
\[
\frac{1}{r_i}+ \sup\{\dis_i(\gamma_i z_i,z_i)\mid z_i\in B_{r_i}^{X_i}(x_i)\}< 2\varepsilon.
\]
This implies that for any sequence $\{\gamma_i\mid i\in \N,\: \gamma_i\in K_i\}$ we have 
\[
\dis_0(\gamma_i,\mathrm{Id}_{X_i}):= \inf\left\{\frac{1}{r}+\sup\{\dis_{i}(\gamma_iz_i,z_i)\mid z_i\in B^{X_i}_{r}(x_i)\}\mid r>0\right\}\underset{i\to\infty}{\longrightarrow} 0.
\]
From this we conclude that $\{K_i\}_{i\in \N}$ is a sequence of small groups in the sense of \cite[Remark 75]{SantosZamora2023}. Then our conclusions follow from \cite[Theorem 93]{SantosZamora2023}.
\end{proof}

\begin{rmk}
\th\ref{L: pm-GH convergences implies eq-GH convergence} can be seen as an improvement  of \cite[Theorem 5.6]{Santos2020} when considering regular points as base points. Observe that the assumption on the limit group $G_Y$ being compact is also a necessary condition for \cite[Theorem 5.6]{Santos2020}.
\end{rmk}

\begin{rmk}
The condition in \th\ref{L: pm-GH convergences implies eq-GH convergence} requiring that the spaces in the sequence and the limit have constant essential dimension is a necessary one. The so-called ``horn'' space  \cite[Example 8.77]{CheegerColding1997} provides an example of a sequence of $\RCD(0,5)$-spaces $X_i$ which are warped products over the round $4$-sphere, for which the tangent space at the vertex has dimension $4$, converging to an $\RCD(0,5)$ space $Y$, but the tangent space at the limit of the vertices is a half-line.
\end{rmk}

In what follows we convene that for $(Z,\dis_Z,\m_Z,0_Z)$, a pointed metric measure space,  $(\R^{k},\dis_{\mathbb{E}},\mathcal{L}^k,0)\times (Z,\dis_Z,\m_Z,0_Z) = (Z,\dis_Z,\m_Z,0_Z)$ for $k=0$.

\begin{lemma}\th\label{L: Splitting of Euclidean space in tangent cone}
Let $(X,\dis,\m)$ be a  non-collapsed $\RCD(K,N)$-space 
and $G$ a compact Lie group acting effectively by measure-preserving isometries on $X$ with cohomogeneity $1$. For each $x_0\in X$ and $(Y,\dis_Y,\m_Y,0_Y)\in \mathrm{Tan}(X,\dis,\m_X,x_0)$, let $\overline{G}$ be the limit group induced by $G$ for some blowup-sequence of $X$ converging to $Y$. Then $Y$ is isomorphic to a product $(\R^{\dime \overline{G}(0_Y)},\dis_{\mathbb{E}},\mathcal{L}^{\dime \overline{G}(0_Y)},0)\times (Z,\dis_Z,\m_Z,0_Z)$, where $Z$ is an $\RCD(0,N-\dime (\overline{G}(0_Y)))$-space.
\end{lemma}

\begin{proof}
Let us begin by fixing a blow-up sequence $X_n:=(X,\dis_n,\m_n,x_0)$ converging pointed-measured-Gromov-Hausdorff to a tangent space $(Y,\dis_Y,\m_Y,0_Y)$. In particular, we have the equivariant Gromov-Hausdorff convergence  $(X_n,G) \to (Y,\overline{G})$, and thus $X_n/G$ converges in Gromov-Hausdorff to $Y/\overline{G}$. Hence the $\overline{G}$-action on $Y$ is of cohomogeneity $1$. 
Now, since $(Y,\dis_Y,\m_Y,0_Y)$ is the tangent space of a non-collapsed $\RCD(K,N)$-space then it is a non-collapsed $\RCD(0,N)$-space and a metric cone. By repeatedly applying Gigli's Splitting Theorem we conclude that $(Y,\dis_Y,\m_Y,0_Y)$ is pointed-isomorphic to a product $(\R^{l},\dis_{\mathbb{E}},\mathcal{L}^l,0)\times (W,\dis_W,\m_W,0_W)$ for some $0\leq l\leq N$ and where $W$ is an $\RCD(0, N-l)$-space which does not contain any lines.

We divide the subsequent analysis in two cases depending wether $0_Y$ is a fixed point of the $\overline{G}$ action on $Y$ or not:\\

\begin{itemize}
    \item[(1)] Assume first that $0_Y$ is not a fixed point. Then there exists $\overline{g}\in \overline{G}$ such that $y_1:=\overline{g}\cdot 0_Y \neq 0_Y$. Recall from \cite{dePhilippisGigli2018} that, since $X$ is non-collapsed then $Y$ is a metric cone with vertex $0_Y$. Now, since $\overline{g}$ is an isometry then $y_1$ is also a vertex of $Y$ (that is, $Y$ is invariant under rescalings centered at $y_1$). Then, since both $0_Y$ and $y_1$ are vertices of $Y$, there exists a (unique) line $L$ joining them. We now can use this fact in the following way: as $\overline{G}\leq \Iso(Y,\dis_Y)$, then $\overline{h}(L)$ is a line for all $\overline{h}\in \overline{G}$. As $\mathbb{R}^{l}\times W$ is equipped with the product metric and $W$ has no lines this implies that the orbit $\overline{G}(0_Y)\subset \R^{l}\times \{0_W\}$. Thus, we can also conclude that $\overline{G}(\R^{l}\times \{0_W\}) = \R^{l}\times \{0_W\}$. Indeed, any point of the form $(x,0_W)\in \R^{l}\times \{0_W\}$ is connected by a line with a vertex of $Y$, so that its orbit must remain in the Euclidean factor.\\

    \item[(2)] If $0_Y$ is a fixed point, then again, any point of the form $(x,0_W)\in \R^{l}\times \{0_W\}$ is connected by a line with $0_Y$, and again, its orbit must remain in the Euclidean factor.\\
\end{itemize}

Now, since $\overline{G}$ acts on $\R^{l}\times \{0_W\}$, the Slice Theorem yields that there exists $S^{\mathbb{R}^l}_{(0,0_W)}\subset \mathbb{R}^l\times\{0_W\}$ a slice (with respect to this action) at $(0,0_W)$.  Moreover, since the action is by isometries, the Myers-Steenrod Theorem \cite{MyersSteenrod} implies that the action is smooth. In particular, $\overline{G}((0,0_W))$ is a smooth manifold. Hence, by using the exponential map of $\mathbb{R}^l$ in the normal directions to the orbit, we have that $S^{\mathbb{R}^l}_{(0,0_W)}$ is isometric to $\mathbb{R}^m$ for some $0\leq m\leq l$. In other words, we have a pointed-isometry
\[
(\mathbb{R}^l,d_{\mathbb{E}},0) \cong (\mathbb{R}^{l-m},\dis_{\mathbb{E}},0)\times (\mathbb{R}^m ,\dis_{\mathbb{E}},0)
\]
induced by the isometric identifications $T_{(0,0_W)}\overline{G}(0_Y) \cong \mathbb{R}^{l-m}$ and $T_{(0,0_W)}S^{\mathbb{R}^l}_{(0,0_W)} \cong \mathbb{R}^{m}$. Moreover by Fubini's Theorem, this induces a pointed measure preseving isometry
\[
(Y, \dis_Y,m_Y,0_Y) \cong (\mathbb{R}^{l-m},\dis_{\mathbb{E}},\mathcal{L}^{l-m},0)\times (\mathbb{R}^m ,\dis_{\mathbb{E}},\mathcal{L}^{m},0)\times (W,\dis_W,\m_W,0_W).
\]

We now claim that $S^{Y}_{(0,0_W)}= \{0\}\times \mathbb{R}^m\times W$, where $S^{Y}_{(0,0_W)}$ is the slice through $(0,0_W)$ of the full $\overline{G}$-action on $Y$.
Indeed, using the characterization of slices in \th\ref{MT: Slice Theorem Cohomogenity one}, and the decomposition as a product of $Y$, it is clear that 
\[
\left\{ (t,r,w) \in \mathbb{R}^{l-m}\times \mathbb{R}^m\times W \mid \dis_Y( (t,r,w), \overline{G}(0_Y)) = \dis_Y((t,r,w), (0,0,0_W))  \right\}
\]
is precisely $\{0\}\times \mathbb{R}^m\times W$. This now means in turn that $\overline{G}_{0_Y}$ acts  on $\{0\}\times \mathbb{R}^m\times W$ and that $(\{0\}\times \mathbb{R}^m\times W)/\overline{G}_{0_Y}$ is $1$-dimensional. Thus, $\overline{G}_{0_Y}$ acts transitively on the unit spheres of $\{0\}\times \mathbb{R}^m\times W$ centered at $(0,0,0_W)$. This implies that the action of $\overline{G}_{0_Y}$ on $\{0\}\times \mathbb{R}^m\times W$ sends the line factors of $\{0\}\times\R^m\times\{0_W\}$ into $W$. However, since $W$ does not contain any lines, we have a dichotomy: either $W$ is a trivial space made up of one single point or $m=0$. We now define 

\[
    (Z,\dis_Z,m_Z,0_Z) =
    \begin{cases*}
     (W,\dis_W,m_W,0_W) & if $m=0$, \\
    (\R^{m},\dis_{\mathbb{E}},\mathcal{L}^m,0) & if $m>0$.
    \end{cases*}
\]
Observe finally that in both cases, using that $Y$ is non collapsed,  we have that $\dime (Z)= N - \dime(\overline{G}(0_Y))$, which is what we claimed. 
 \end{proof} 

\begin{lemma}
\th\label{L: infinitesimal action cohomogeneity 1}
Let $(X,\dis,\m)$ be a  non-collapsed $\RCD(K,N)$-space and $G$ a compact Lie group acting effectively by measure-preserving isometries on $X$ with cohomogeneity $1$. For $x_0\in X$ let $(Y,\dis_Y,m_Y,0_Y) \cong (\R^{m},\dis_\E,\mathcal{L}^{m},0)\times (Z,\dis_Z,m_Z,0_Z)\in \Tan(X,\dis,\m,x_0)$ where $m$ is the dimension of the orbit of $0_Y$ under any limit group induced by $G$. Then $Z$ is homeomorphic to the slice $S_{x_0}$ at $x_0$. Furthermore, $ \dime (G(x_0))= m$.
\end{lemma}

\begin{proof}
Since $(X^\ast,\dis^\ast,\m^\ast)$ is an $\RCD$-space (by \cite{GalazGarciaKellMondinoSosa2018}) of essential dimension equal to $1$, then by \cite{KitabeppuLakzian2016} we have that $(X^\ast,\dis^\ast)$ is isometric to one of the spaces: $(\R,\dis_{\E})$, $([0,\infty),\dis_{\E})$, $([0,1],\dis_{\E})$, or  $(\Sp^1(1),\rho)$ where $\dis_{\E}$ is the Euclidean metric and $\rho$ is the length metric induced by $\R^2$. 

Let $\overline{G}$ be a limit group induced by $G$ for some fixed blowup-sequence $(X_n,\dis_n,\m_n,x_0)$ of $X$ centered at $x_0$. Recall from the proof of \th\ref{L: Splitting of Euclidean space in tangent cone} that $\overline{G}(\R^{m}\times\{0_Z\}) = \R^{m}\times\{0_Z\}$. 

We now consider $(0,z),(0,z_1)\in (\{0\}\times Z)\cap \partial B_a(0,0_{Z})$ for some $a>0$. Observe then that using the submetry $\pi\colon \R^{m}\times Z\to (\R^{m}\times Z)/\overline{G}$, for small enough $a$ we have that 
\[
\pi(\partial B_a(0,0_{Z})) = 
\begin{cases}
    \text{a single point} & \text{if $X^\ast$ is homeomorphic to $[0,1]$}\\
    & \text{or $[0,\infty)$, and the orbit $G(x_{0})$}\\
    & \text{corresponds to a  boundary point}\\
    & \text{of $X^\ast$,}\\
    \text{two points} & \text{otherwise.}
\end{cases}
\]
To proceed then, we divide the proof into these two situations with the aim of proving that the limit group induced by $G_{x_0}$ acts by cohomogeneity one on $Z$.\\

\begin{itemize}
    \item[(1)] Assume first that $X^{\ast}$ is homeomorphic to $[0,1]$ or $[0,\infty)$ and $x_0^{\ast}$ corresponds to a boundary point. This implies that there exists $\bar{h}\in \overline{G}$ such that $\bar{h}\cdot (0,z) = (0,z_1)$. Since we have a product metric on $Y$ and $\overline{G}(\R^{m}\times \{0_Z\})= \R^{m}\times\{0_Z\}$, we conclude that $\overline{h}$ fixes $(0,0_{Z})$; More precisely, since $\overline{h}(0,0_{Z})\in \overline{G}(\R^{m}\times \{0_Z\})=\R^{m}\times \{0_Z\}$, then $\overline{h}(0,0_{Z})=(v,0_Z)$, for some $v$. Now we have that
\begin{linenomath}
\begin{align*}
\dis_Z(0_Z,z)^2&= \dis((0,0_Z),(0,z))^2 = \dis(\overline{h}\cdot (0,0_Z),\overline{h}\cdot (0,z))^2=\dis(v,0_Z),(0,z_1))^2\\
&=\dis_{E}(v,0)^2+\dis_Z(0_Z,z_1)^2=\dis_{E}(v,0)^2+\dis_Z(0_Z,z)^2.   
\end{align*}
\end{linenomath} 
Therefore $v=0$ and we obtain that $\overline{h}$ fixes $(0,0_{Z})$. 

Now observe that $a = \dis^\ast((0,0_{Z})^{\ast},(0,z)^{\ast})$ and consider a sequence $\{s_i^{\ast}\}_{i=1}^{\infty}\subset X^{\ast}$ such that $\dis_i^{\ast}(x_0^{\ast},s^{\ast}_i)=a$. By construction, the interval $(0,(0,z)^{\ast})$ is the limit of $\{s^{\ast}_i\}_{i=1}^{\infty}$ in $(\R^{m}\times Z)/\overline{G} =\mathrm{GH}\lim_{i\to \infty}(X^{\ast},\dis_i^\ast)$.

To continue, denote $K:=G_{x_0}$ and the ball of radius $a$ centered at $x_0$ in $X_i$ by $B^i_{a}(x_0)$. We claim that $K$ acts transitively on $\partial B^i_{a}(x_0)\cap S_{x_{0}}$. Indeed, the map $\phi\colon G\times_{K} S_{x_{0}}\to X$ given by $\phi[g,s]:=g\cdot s$ is a homeomorphism onto its image (see for example \cite[Theorem 2.10]{CorroKordass2021}). From this we deduce that the quotient space $(G\times_{K}S_{x_{0}})/G$ is homeomorphic to $S_{x_0}/K$ which proves our claim.

This implies that there exists a geodesic $\gamma\subset S_{x_{0}}$ starting at $x_0$ and a sequence of lifts $\{s_i\}_{i=1}^{\infty}\subset \gamma$ of the sequence $\{s_i^{\ast}\}_{i=1}^{\infty}$. Then there is a limit geodesic $\gamma_\infty\subset \R^{m}\times Z$ passing through $(0,0_{Z})$ and a limit point $s_{\infty}\in \gamma_{\infty}$ of $\{s_i\}_{i=1}^{\infty}$. Moreover, by construction $s_\infty$ is in the same $\overline{G}$-orbit as $(0,z)$ and $(0,z_1)$.

We now consider $\overline{h}\cdot \gamma_\infty$, and we point out that  $S_{x_0}/K$ is homeomorphic to a sufficiently small neighborhood of $x_0^\ast\in X^{\ast}$. Then there exists a sequence $\{g_i\}_{i=1}^{\infty}\subset G$ such that $g_i\cdot s_i$ converges to $\overline{h}\cdot s_{\infty}$ (see \cite[Remark 2.6.3.3]{PetruninRongTuschmann1999}). However, since $\overline{h}$ fixes $(0,0_Z)$, then the $g_i$ fix $x_0$, and thus $g_i\in K$. Now, since $K$ is compact we can find an element $k\in K$ (the limit of a convergent subsequence of $\{g_i\}_{i=1}^{\infty}$), such that the sequence $k\cdot s_i$ converges to $\overline{h}\cdot s_\infty$. This implies that $\overline{h}$ is an element of $\overline{K}$, the limit group induced by $K$. Since $\overline{K}$ fixes $(0,0_Z)$ by construction, then $\overline{K}$ leaves $B_a(0,0_{Z})\cap Z$ invariant. Thus  we conclude that $(\partial B_a(0,0_{Z})\cap Z)/\overline{K}$ consists of only one point. Whence, $\overline{K}$ acts by cohomogeneity one on $(Z,\dis_{Z})$.\\

\item[(2)] If $\pi(\partial B_a(0,0_{Z}))$ consists of only two points, the same reasoning we used in case $(1)$ to show that $z_1$ and $z$ are in the same $\overline{K}$-orbit holds, if $\pi(0,z_1)=\pi(0,z)$. Thus we now assume that $\pi(0,z_1)\neq\pi(0,z)$. In this case the geodesic $\gamma$ can be extended uniquely into a geodesic $\tilde{\gamma}\subset S_{x_0}$ so that $x_0$ is an interior point. Up to reparametrizing we can assume that  $[-1,1]\subseteq X^{\ast}$ (with the obvious abuse of notation if $X^{\ast}$ is homeomorphic to $\Sp^1(1)$) so that $x_{0}^{\ast}=0$ and for every index there exists a point $-s_i\in \tilde{\gamma}$ such that $(-s_i)^\ast = -(s_i^\ast)\in [0,1]$. From this we conclude that the sequence $(-s_i)^\ast$ converges to $(0,z_1)^\ast$, and we can apply the previous reasoning to conclude that $(\partial B_a(0,0_Z)\cap Z)/\overline{K}$ consists of only two points. This implies that $\overline{K}$ acts on $Z$ by cohomogeneity one.
\end{itemize}

We have so far shown that $\overline{K}$ acts by cohomogeneity one in all cases. Now we consider $\gamma\colon [0,1]\to S_{x_0}\subset X$ a minimizing geodesic of $X$, starting at $x_0$, and realizing the distance between $G(x_0)$ and $G(\gamma_{1})$. Then the (closed) ray $(\gamma([0,1)),\frac{1}{r_n}\dis)$ converges to an (open) minimizing ray $\widetilde{\gamma}\colon [0,\infty)\to (Y,\dis_Y)$. Observe that under the submetry $\pi_X\colon (X,\dis)\to (X/G,\dis^\ast)$, we have that $\pi_X(\gamma)$ is a minimizing geodesic between any of its points. Thus we conclude that for the submetry $\pi_Y\colon (Y,\dis_Y)\to (Y/\overline{G},\dis^\ast_Y)$, the curve $\pi_Y(\widetilde{\gamma})$ is a minimizing geodesic between its endpoints. Moreover, $\pi_Y(\widetilde{\gamma})$ is the limit of the curves $\pi_X(\gamma)$ under the blow up. Since $Z/\overline{K}$ is homeomorphic to $Y/\overline{K}$, there exists a geodesic $\widetilde{\gamma}_Z\subset Z$ that projects to $\pi_Y(\widetilde{\gamma})$. Then there exists $\overline{h}\in \overline{G}$ such that $\overline{h}\cdot \widetilde{\gamma}=\widetilde{\gamma}_Z$. In turn, there exists a geodesic $\alpha$ in $X$ starting at $x_0$ converging to $\widetilde{\gamma}_Z$, and elements $g_i$ mapping $\gamma$ to $\alpha$, and converging to $\overline{h}$. Thus we conclude that $\overline{h}\in \overline{K}$, and thus   $\widetilde{\gamma}$  is a geodesic in $Z$, and realizes the distance between the $\overline{K}$-orbit of the vertex $0_Z$ and the $\overline{K}$-orbits. Thus for each radial geodesic $\gamma$ in $S_{x_0}$ starting at $x_0$, we obtain a limit  radial geodesic $\widetilde{\gamma}$ in $Z$ starting at $0_Z$. From the analysis in Cases (1) and (2), given any $z$ in $Z$ we can find a geodesic ray $\gamma$ in $S_{x_0}$, such that in the limit geodesic ray $\gamma_\infty$ there exists $s_\infty$ in the same $\overline{G}$-orbit as $(0,z)$. But since we have also shown that $s_\infty$ is contained in $\{0\}\times Z$, then we conclude that actually $s_\infty$ is in the $\overline{K}$-orbit as $(0,z)$. But this implies that we can find a $k\in K$ and a sequence $s_i\in \gamma$ such that $k\cdot s_i$ converges to $(0,z)$. This implies that $(0,z)$ is in the limit ray of $k\cdot \gamma$. Thus we can uniquely identify each radial geodesics in $S_{x_0}$ and $Z$.

Now we point out that we can identify uniquely each radial geodesic starting at $x_0$ and $0_Z$ respectively, with its unique intersection point with the boundary of a closed ball $\overline{B}_{\varepsilon}(0_Z)$, and that the same holds for the ball $\overline{B}_{\varepsilon}(x_0)\cap S_{x_0}$ for sufficiently small $\varepsilon$, such that $B_\varepsilon(x_0)\cap S_{x_0}$ is the open $\varepsilon$-ball around $x_0$ in $S_{x_0}$ with the induced intrinsic metric. Then we have an homeomorphism between $\overline{B}_{\varepsilon}(x_0)$ and $\overline{B}_{\varepsilon}(0_Z)$.

Last we observe that $K$ acts transitively on $\partial\overline{B}_{\varepsilon}(x_0)$, and $\overline{K}$ acts transitively on $\partial \overline{B}_{\varepsilon}(0_Z)$. This is due to the fact that $\pi_X(\overline{B}_{\varepsilon}(x_0))$ is the boundary of a closed ball around $x_0^\ast$ in $(S_{x_0},\dis^\ast)$; but such a ball is homeomorphic to $[0,\varepsilon]$. Thus $\pi_X(\overline{B}_{\varepsilon}(x_0))$ consists of only one point, and this implies that the $K$-orbit is equal to $\partial \overline{B}_{\varepsilon}(x_0)$, i.e. the action of $K$ on $\partial \overline{B}_{\varepsilon}(x_0)$ is transitive. An analogous reasoning holds for $Z$, $\overline{K}$, and $\partial \overline{B}_{\varepsilon}(0_Z)$. Moreover, $S_{x_0}$ is homeomorphic to the topological cone over $\partial \overline{B}_{\varepsilon}(x_0)$, and analogously $Z$ is homeomorphic to the topological cone over $\partial \overline{B}_{\varepsilon}(0_Z)$. From this we conclude that the ball $\overline{B}_\varepsilon(x_0)$ is homeomorphic to $\overline{B}_\varepsilon(0_Z)$. Since $X$ is non-collapsed, then $Z$ is also a metric cone over $\partial \overline{B}_\varepsilon(0_Z)$. From this we conclude that $Z$ is homeomorphic to a cone over a homogeneous space. With this we get a homeomorphism between $S_{x_0}$ and $Z$.\\

We are only left with proving the claim about the value of $m$. Firstly, as obtained by Brena-Gigli-Honda-Zhu (see the discussion after Theorem 1.3 in \cite{BrenaGigliHondaZhu2023}) the topological dimension of $X$, $\dime_{top}(X)$ coincides with its essential dimension. Now apply the Slice Theorem to note that $X$ is locally a product of $G(x_0)$ and the slice $S_{x_0}$, and in turn to obtain that 
\[
N=\dime_{top}(X) = \dime_{top}(G(x_0)) + \dime_{top}(S_{x_0}) = \dime (G(x_0)) + m.
\]
where we used the identification of $S_{x_0}$ with $Z$. Therefore $\dime (G(x_0))=N-m$.
 \end{proof}

\begin{rmk}
The condition that all tangent spaces of $X$ are metric cones is a necessary one: Pan and Wei \cite{PanWei2022} have obtained a non-compact example of a collapsed $\RCD$-space  with an $\R$-action of cohomogeneity one by measure preserving isometries which is its own tangent space at a specific point, and such that no line can be split. Note that at the given point, the tangent space is not a metric cone. 
\end{rmk}

\begin{duplicate}[\ref{MT: Geometry of the slice}]
  Let $(X,\dis,\m)$ be a  non-collapsed $\RCD(K,N)$-space and $G$ a compact Lie group acting effectively by measure-preserving isometries on $X$ with cohomogeneity $1$. Then for every $x_0\in X$,
  
  \begin{itemize}
      \item[(a)] the slice $S_{x_0}$ admits an $\RCD(0,N-m)$ structure, where $m= \dim(G(x_0))$; 
  \item[(b)] the set of tangent spaces at $x_0$ is single-valued and 
  \[
  \Tan(X,\dis,\m,x_0) = \left\{ (\R^{N-m},\dis_\E,\mathcal{L}^{N-m},0)\times (\mathrm{Con}(M),\dis_{\mathrm{Con}},m_{\mathrm{Con}},o)  \right\},
  \]
  where $M$ is a compact homogeneous space with a Riemannian metric of Ricci curvature greater than or equal to $N-m-2$.
  \end{itemize} 
\end{duplicate}

\begin{proof}
Following the notation of the previous lemma, we showed that $Z$ is homeomorphic to $S_{x_0}$. Since $(Z,d_Z,m_Z,0_Z)$ is a non-collapsed $\RCD(0,N-m)$-space, we can equip $S_{x_0}$ with this same structure. This proves (a).

We showed in  \th\ref{L: infinitesimal action cohomogeneity 1} as well that $\overline{K}$ acts transitively on $\partial \overline{B}_{\varepsilon}(0_Z)$ for every $\varepsilon >0$ (in particular it does so on $M:=\partial \overline{B}_{1}(0_Z)$). Therefore $M$ is a homogeneous space and it follows from Ketterer's Theorem (\cite[Theorem 1.4]{Ketterer2015}) that $M$ admits a metric measure structure $(M,\dis_M,\m_M)$ which makes it an $\RCD(N-m-2,N-m-1)$-space. Let us observe that $(M,\dis_M,\m_M)$ is non-collapsed since $\dime M = \dime Z-1 = N-m-1$. Therefore $m_M = c\Hauss^{N-m-1}$ for some constant $c>0$, and by \cite[Proposition 5.14]{Santos2020} it follows that $(M,\dis_M, c\Hauss^{N-m-1})$ is isomorphic to a Riemannian manifold $(M,g_M,d\mathrm{vol})$. Finally, by \cite[Theorem~2~(i)]{Sturm2006} it follows that $\Ric(g_M) \geq N-m-2$ as claimed. This proves (b).
\end{proof}

\begin{rmk}\th\label{R: not infinitesimal representation}
Consider $(X,\dis,\Hauss^N)$ a non-collapsed $\RCD(K,N)$-space and $G$  a compact Lie group acting effectively by measure preserving isometries on $X$ with cohomogeneity one. Then for $x_0\in X$ fixed and $\delta>0$ small enough, if we take $K=G_{x_0}$ and $H= G_{y}$ for $y\in \Tub^\delta(G(x_0))$, then we have that $S_{x_0}$ is homeomorphic to the cone over $K/H$. This is due to the conclusions of \th\ref{L: infinitesimal action cohomogeneity 1} and the fact that $H$ is the largest subgroup in $K$ that fixes the boundary of $\Tub^\delta(G(x_0))$, and thus it fixes $\partial \overline{B}_\delta(x_0)\subset S_{x_0}$.
\end{rmk}

\begin{rmk}
\label{R: isotropy-acts-on-F}
For a non-collapsed $\RCD$-space $(X,\dis,\Hauss^N)$ satisfying the hypothesis of \th\ref{MT: Geometry of the slice}, by \th\ref{R: not infinitesimal representation} we have an action of $G_{x_0}$ on the base of the cone $M:=\overline{K}/\overline{H}$ the $\RCD(N-m-2,N-m-1)$-space by homeomorphisms, and for $x_0$ not in an orbit of principal type, the space $M$ is homeomorphic to $G_{x_0}/H$. However, it is not known a priory if the transitive action of $G_{x_0}$  is by measure preserving isometries with respect to the $\RCD(N-m-2,N-m-1)$-structure.
\end{rmk}

To proceed we have the following lemma roughly asserting that orbits of the base-points in a epGH-converging sequence locally converge to the orbit of the base-point in the limit. We denote the Hausdorff distance in a metric space $(Y,d)$ by $d^Y_H$. 

\begin{lemma}
\th\label{L:Orbits-locally-converge-to-the-limit-orbit}
Assume that $\{(X_n,d_n,x_n,G_n)\}_{n\in \N}$ converges to $(Y,d,y_0,\overline{G})$ 
in the equivariant pointed Gromov--Hausdorff sense. Then for every $R>0$,
\[
d^Y_H\!\left(f_n(G_n(x_n))\cap B_R^Y(y_0),\ 
             \overline{G}(y_0)\cap B_{R+2\varepsilon_n}^Y(y_0)\right)\leq 2\varepsilon_n,
\]
for some $\varepsilon_n\to 0$ and corresponding equivariant pointed 
$\varepsilon_n$-approximations $(f_n,\theta_n,\psi_n)$. 
In particular, the sets $f_n(G_n(x_n))$ converge 
locally in the Hausdorff sense to the orbit $\overline{G}(y_0)$.   
\end{lemma}

\begin{proof}
Let us fix $R>0$ and take a decreasing sequence $\{\varepsilon_n\}_{n=1}^{\infty}$ converging to $0$ and corresponding equivariant $\varepsilon_n$-approximations $(f_n,\theta_n,\psi_n)$ in such a way that $R+2\varepsilon_n \leq 1/\varepsilon_n$ \footnote{By the usual formula for roots of quadratic polynomials, it suffices to take $\varepsilon_n\leq \frac{-R+\sqrt{R^2+8}}{4}$.}. 

Let $z\in f_n(G_n(x_n))\cap B_R(y_0)$. Then $z=f_n(gx_n)$ for some $g\in G_n$  and 
$d(y_0,f_n(gx_n))<R$. By applying the triangle inequality and using that the distortion of $f_n$ is at most $\varepsilon_n$, we have that 
\begin{align*}
d_n(x_n,gx_n)& \leq d(f_n(x_n),f_n(gx_n)) +\varepsilon_n\\
&\leq  d(f_n(x_n), y_0)+d(y_0, f_n(gx_n)) + \varepsilon_n\\
&\leq R+ \varepsilon_n. 
\end{align*}

Therefore, $g\in G_n(1/\varepsilon_n)$. From this we can estimate, using the triangle inequality, the almost isometries, and the fact that $\theta_n(g)$ is an isometry on $Y$,
\begin{align*}
d\left(z,\theta_n(g) y_0\right)
&\leq d\left(f_n(gx_n),\theta_n(g) f_n(x_n)\right)
   + d\left(\theta_n(g)f_n(x_n),\theta_n(g) y_0\right) \\
&\leq \varepsilon_n + d\left(f_n(x_n),y_0\right)\\
&\leq \varepsilon_n.
\end{align*}
Now using the triangle inequality again, and the previous estimate, we have that 
\[
d(y_0,\theta_n(g) y_0)\leq d(y_0,z)+   d\left(z,\theta_n(g) y_0\right) \leq R
+\varepsilon_n.
\]
It now follows that $\theta_n(g) y_0\in \overline{G}(y_0)\cap B_{R+\varepsilon_n}(y_0)$. Putting these together, we have that
\begin{equation}
\label{EQ1:Orbits-locally-converge-to-the-limit-orbit}
\sup_{z \in f_n(G_n(x_n))\cap B_R(y_0)}
   d\left(z,\overline{G}(y_0)\right)
   \leq \varepsilon_n.
\end{equation}
Now to proceed, we take $w=\overline{g}(y_0)\in \overline{G}(y_0)\cap B_R(y_0)$.
Then $d(y_0, \overline{g}(y_0))<R$. Now, using the distortion of $f_n$ and  that $\overline{g}\in \overline{G}(R)$, we get the following estimate:
\begin{align*}
d_n(x_n,\psi_n(\overline{g})x_n)
&\leq d\left(f_n(x_n),f_n(\psi_n(\overline{g}) x_n)\right)+\varepsilon_n \\
&\leq d\left(f_n(x_n),\overline{g} f_n(x_n)\right) + d\left(\overline{g} f_n(x_n), f_n(\psi_n(\overline{g}) x_n)\right) + \varepsilon_n\\
&\leq d\left(f_n(x_n),\overline{g} f_n(x_n)\right)+2\varepsilon_n \\
&\leq d\left(f_n(x_n), y_0\right)+ d\left(y_0, \overline{g}(y_0)\right) + d\left(\overline{g}(y_0), \overline{g}(f_n(x_n))\right) + 2\varepsilon_n\\
&\leq d(y_0,\overline{g}(y_0))+2\varepsilon_n\\
&\leq R+2\varepsilon_n.
\end{align*}

Therefore, $\psi_n(\overline{g})\in G_n(1/\varepsilon_n)$. We now have by definition of equivariant approximations that,
\[
d\!\left(\overline{g}\left( f_n(x_n)\right),\, f_n\left(\psi_n(\overline{g}) x_n\right)\right)\leq \varepsilon_n,
\]
and so in turn, by the triangle inequality and the fact that $\overline{g}$ is an isometry,
\[
d\left(w,\, f_n\left(\psi_n(\overline{g}) x_n\right)\right)
\leq d\left(\overline{g}(y_0),\, \overline{g}(f_n(x_n))\right)+\varepsilon_n
= d\!\left(y_0,\, f_n(x_n)\right)+\varepsilon_n\leq \varepsilon_n.
\]
Thus every such $w$ is at distance at most $\varepsilon_n$ of $f_n(G_n(x_n))$, and then
\begin{equation}  
\label{EQ2:Orbits-locally-converge-to-the-limit-orbit}
\sup_{w\in \overline{G}(y_0)\cap B_R^Y\left(y_0\right)}
\dis\left(w,f_n\left(G_n(x_n)\right)\right)\ \leq\ \varepsilon_n.
\end{equation}
Now we just need to put together inequalities \eqref{EQ1:Orbits-locally-converge-to-the-limit-orbit} and \eqref{EQ2:Orbits-locally-converge-to-the-limit-orbit} to get the desired conclusions.
\end{proof}

In the following result we forgo the cohomogeneity one assumption and obtain the same conclusions than in \th\ref{L: Splitting of Euclidean space in tangent cone}, albeit assuming that $G$ acts Lipschitz and co-Lipschitz continuously by measure-preserving isometries, that is, we assume that the metric on $G$ is such that the map $\star y : G \to G(y)$, given by $g \mapsto gy$, is locally Lipschitz and co-Lipschitz continuous, for some (and hence for all) $y$ with principal orbit type. In other words, we assume that for every $y\in X$ in a principal orbit there exist constants $R,C>0$ such that for all $r\in (0,R)$, 
\[
B_{C^{-1}r}(y)\cap G(y) \subset \{ g \cdot y \mid g \in B^{G}_{r}(e) \} \subset B_{Cr}(y)\cap G(y).
\]

This assumption is used substantially in \cite{GalazGarciaKellMondinoSosa2018}. In the proof we very briefly use the Assouad dimension in the passing, and therefore we do not define it here. We refer the interested reader to \cite{Heinonen2001} for the definition.
    
\begin{mtheorem}\th\label{MT: infinitesimal action non-collapsed}
Let $(X,\dis,\Hauss^N )$ be a non-collapsed $\RCD$-space and $G$ an $m$-dimensional compact, connected Lie group acting effectively and Lipschitz and co-Lipschitz continuously by measure-preserving isometries on $X$.  For each $x_0\in X$ fixed, let $K = G_{x_0}$ and $k= \dim(K)$. Then every $(Y,\dis_Y,\m_Y,0_Y)\in \mathrm{Tan}(X,\dis,\m_X,x_0)$ is isomorphic to $(\R^{m-k},\dis_{\mathbb{E}},\mathcal{L}^{m-k},0)\times (Z,\dis_Z,\m_Z,0_Z)$, for some space $Z$, a non-collapsed $\RCD(0,N-m+k)$-space. Moreover $K$ leaves each factor invariant, acts by measure preserving isometries on $(Z,\dis_Z,\m_Z)$, and fixes $0_Z$.
\end{mtheorem}

\begin{proof}
The strategy of proof is similar to that of \th\ref{L: Splitting of Euclidean space in tangent cone}, with some modifications. Let us fix a blowup sequence $(X_n,\dis_n,\m_n,x_0)$ converging pointed measured Gromov-Hausdorff to $(Y,\dis_Y,\m_Y,0_Y)$. In fact, as the cohomogeneity one assumption is not used in the first part of \th\ref{L: Splitting of Euclidean space in tangent cone}, the exact same proof shows that in this case 
\[
(Y, \dis_Y,m_Y,0_Y) \cong (\mathbb{R}^{l},\dis_{\mathbb{E}},\mathcal{L}^{l},0)\times (W,\dis_W,\m_W,0_W).
\]
where $(W,\dis_W,\m_W)$ is a non-collapsed $\RCD(0,N-l)$-space for some $0\leq l\leq N$ which does not contain any lines. Moreover, $\overline{G}(\R^{l}\times \{0_W\}) = \R^{l}\times \{0_W\}$, where $\overline{G}$ is the limit group induced by $G$ for some blowup sequence converging to $Y$. Now again, exactly as in \th\ref{L: Splitting of Euclidean space in tangent cone}, this implies that $\overline{G(0_Y)}$ is a smooth manifold, and in turn that $T_{0_Y}G(0_Y) \cong \mathbb{R}^s\times \{0\}\times\{0_W\} \subseteq \mathbb{R}^s\times \mathbb{R}^{l-s}\times W \cong Y$ for some $0\leq s\leq l$. 

Let us now point out that \th\ref{L:Orbits-locally-converge-to-the-limit-orbit} implies that up to passing to a (non-relabeled) subsequence, the orbits $G(x_0)\subset X_n$, equipped with the restricted metrics $d_n$, and with the base point $x_0$, pointed-Gromov-Hausdorff converge to $G(0_Y)$, equipped with the restricted metric of $Y$, with base point $0_Y$. This further implies by a diagonal sequence argument that by blowing up $(G(x_0), d_n|_{G(x_0)},x_0)$, and $\overline{G}(0_Y),d_Y|_{\overline{G}(0_Y)},0_Y)$, 
\[
T_{0_Y}G(0_Y) \cong \mathbb{R}^s\times \{0\}\times\{0_W\}\in \mathrm{Tan}(G(x_0),d_n|_{G(x_0)},x_0).
\]
We now apply \cite[Corollary 2.17]{LeDonneRajala2015} to get that 
\[
s\leq \dim_{Assouad}(G(x_0),d_n|_{G(x_0)},x_0).
\]

We now take a page from the proof of \cite[Theorem 6.3]{GalazGarciaKellMondinoSosa2018}:  Since $G$ is a compact, connected Lie group acting on $(X_n,d_n)$ by maps that are both Lipschitz and co-Lipschitz, the induced intrinsic metrics $d_{n,G(x_0)}$ on $G(x_0)$ are bi-Lipschitz equivalent to the restrictions of $d_n$ to $G(x_0)$.  
The isotropy $K$ is compact, and $G(x_0)$ is homeomorphic to the homogeneous space $G/K$. Therefore, by the the work of Berestovskii (\cite[Theorem 3(i)]{Berestovskii1989}), there exist a connected Lie group $G'$ and a compact subgroup $K'\leq G'$ such that
\[
\left(G(x_0),\, d_{n,G(x_0)}\right)\ \cong\ G'/K'
\]
with a $G'$-invariant Carnot-Carathéodory-Finsler metric determined by a completely nonholonomic distribution on $G'/K'$. In particular the Gromov-Hausdorff tangent cone $T_{x_0}G(x_0)$ is unique and isometric to a Carnot group equipped with a (possibly Finsler) left-invariant metric. As pointed out in \cite[Theorem 6.3]{GalazGarciaKellMondinoSosa2018} it follows that $\left(G(x_0),\, d_{n,G(x_0)}\right)$ is isometric to a (homogeneous) Riemannian manifold. Moreover,
\[
T_{0_Y}G(0_Y) \cong \mathbb{R}^s\times \{0\}\times\{0_W\}\in \mathrm{Tan}(G(x_0),d_{n,G(x_0)},x_0)
\]
and that the topological dimension of $G(x_0)$ is $s$. Recall that the topological dimension is always below the Assouad dimension and that these coincide for Riemannian manifolds with the inner metric. From this we conclude that $s=m-k$, and the result is proved. 
\end{proof}

Observe that when the base point $x_0$ is regular in \th\ref{L: pm-GH convergences implies eq-GH convergence}, we obtain actions of $G_{x_0}$ on the tangent space $(Y,\dis_Y,0_y)\cong (\R^n,\dis_{\mathbb{E}},0)$ of $X$ at $x_0$. Then by the same proof  as the one of \th\ref{MT: infinitesimal action non-collapsed} we obtain the following theorem. Recall that the essential dimension of a non-collapsed $\RCD(K,N)$ space $(X,\dis,\Hauss^N)$ is precisely $N$ (see Theorem $2.20$ in \cite{BrenaGigliHondaZhu2023}).

\begin{mtheorem}[Regular orbit representation]\th\label{MT: Principal Isotropy Representation}
Let $(X,\dis,\Hauss^N)$ be a non-collapsed $\RCD$-space, $G$ a compact, connected Lie group of dimension $m$ acting effectively and Lipschitz and co-Lipschitz continuously by measure-preserving isometries on $X$. For $x_0\in X$ a fixed regular point, we consider $K = G_{x_0}$. Then the action of $G$ induces a linear representation of $K$ into $\mathrm{O}(N^\ast)$ where $N^\ast$ is the essential dimension of $X^\ast$. Moreover, this representation splits as a product of  representations of $K$ into $\mathrm{O}(m-k)$ and $\mathrm{O}(N-m+k)$, where $k=\dim(K)$. We call the representation of $K$ into $\mathrm{O}(N-m+k)$ the \emph{slice representation}.
\end{mtheorem}

\begin{rmk}
When $X$ is a Riemannian manifold, then a Lie group action by isometries is Lipschitz and co-Liptschitz. Thus this condition is a natural condition to ask.
\end{rmk}

\begin{rmk}
In the case when $X$ is either a Riemannian manifold, or an Alexandrov space, recall that $x_0$ is contained in a principal orbit, if and only if, for the isotropy group at $x_0$ we have that the restriction of the  isotropy representation to the slice is trivial \cite[Exercise 3.77]{Alexandrino}. In the case of non-collapsed $\RCD$-spaces we do not know if for a principal orbit (see \cite[Theorem 4.7]{GalazGarciaKellMondinoSosa2018}) the action of the isotropy group on the slice is trivial. That is, if the isotropy representation is trivial for principal orbits.
\end{rmk}

\subsection{Topological Rigidity}\hfill\\

The following results gives an explicit topological description of essentially non-bran\-ching metric measure spaces with a Lie group action of cohomogeneity one.

\begin{theorem}\th\label{T: Topological classification geodesic coho 1 spaces.}
Let $(X,\dis, \m)$ be  an essentially non-branching  space, and $G$ compact acting by isometries on $X$. Assume that the orbit space $(X^\ast,\dis^\ast)$ is isometric to $[-1,1]$, $[0,\infty)$, $\R$, or $\Sp^1$. Then:
\begin{enumerate}[(a)]
\item\label{T: Topological classification geodesic coho 1 spaces (a).} When $X^\ast$ is homeomorphic to $[0,1]$, then $X$ is equivariantly homeomorphic to the union of two cone bundles over the singular orbits, glued along their boundary.

\item\label{T: Topological classification geodesic coho 1 spaces (b).} When $X^\ast$ is homeomorphic to $[0,\infty)$, then $X$ is homeomorphic to a cone bundle.

\item\label{T: Topological classification geodesic coho 1 spaces (c).} When $X^\ast$ is homeomorphic to $\Sp^1$, then $X$ is the total space of a fiber bundle with homogeneous fiber $G/H$, and structure group $N_G(H)/H$. 

\item\label{T: Topological classification geodesic coho 1 spaces (d).} When $X^\ast$ is homeomorphic to $\R$, then $X$ is homeomorphic to $\R\times G/H$.
\end{enumerate}
\end{theorem}

\begin{proof}
We begin by proving \eqref{T: Topological classification geodesic coho 1 spaces (c).}, and \eqref{T: Topological classification geodesic coho 1 spaces (d).}. By \th\ref{L: Principal Isotropy Lemma} it follows that for any point $x\in X$, the isotropy $G_x$ equals a fix subgroup $H<G$. Moreover, by \th\ref{T: equivariant homeomorphism from the orbit} we conclude that for any $x\in X$, the orbit through $x$ is equivariantly homeomorphic to $G/H$, i.e. $G(x)\cong G/H$. Then, as pointed out in the proof of \cite[Corollary 4.9]{GalazGarciaKellMondinoSosa2018}, any lift $\tilde{\gamma^\ast}\colon I\to X$ of any short enough geodesic $\gamma^\ast\colon I\to \Sp^1=X^\ast$ gives a trivialization. In the case \eqref{T: Topological classification geodesic coho 1 spaces (c).}, since any two points in $X$ have conjugated isotropy subgroups, we get the conclusion on the structure group of the bundle. The case \eqref{T: Topological classification geodesic coho 1 spaces (d).} follows from the fact that $\R$ is contractible.

The cases \eqref{T: Topological classification geodesic coho 1 spaces (a).} and \eqref{T: Topological classification geodesic coho 1 spaces (b).} follow from \th\ref{R: Slice is a topological cone},  \th\ref{T:Slice-Theorem}, and the observation that the spaces $\pi^{-1}([-1,0])$ and $\pi^{-1}([0,1])$ are equivariantly homeomorphic to the cone bundles $G\times_{G_-} \cl(S_-)$  and $G\times_{G_+} \cl(S_+)$ respectively. Here $\pi\colon X\to X^\ast$ is the orbit quotient map, $G_\pm$ are the isotropy groups of the orbits $\pi^{-1}(\{\pm1\})$, and $S_\pm$ are slices through points in $\pi^{-1}(\{\pm1\})$.
\end{proof}

\noindent Combining \th\ref{T: Topological classification geodesic coho 1 spaces.} with \th\ref{MT: Geometry of the slice} we obtain the following corollary: 

\vspace*{8pt}

\begin{duplicate}[\ref{MC: Homeomorphism rigidity}]
Let $(X,\dis,\m)$ be an $\RCD(K,N)$-space. Let $G$ be a compact Lie group acting on $X$ by measure preserving isometries and cohomogeneity one. Then the following hold:
\begin{enumerate}[(a)]
\item When $X^\ast$ is homeomorphic to $[0,1]$, then $X$ is equivariantly homeomorphic to the union of two cone bundles over the singular orbits, glued along their boundary.

\item When $X^\ast$ is homeomorphic to $[0,\infty)$, then $X$ is homeomorphic to a cone bundle.

\item When $X^\ast$ is homeomorphic to $\Sp^1$, then $X$ is the total space of a fiber bundle with homogeneous fiber $G/H$ and structure group $N_G(H)/H$. 

\item When $X^\ast$ is homeomorphic to $\R$, then $X$ is homeomorphic to $\R\times G/H$.
\end{enumerate}

\noindent Moreover the cone fibers in items $(a)$ and $(b)$ are cones over homogeneous spaces. 

In the case when $X$ is non-collapsed then  the cone fibers admit  metric cone structures over $\RCD(N-k_{\pm}-2,N-k_\pm-1)$-spaces, where $k_\pm = \dim(G(x_\pm))$.
\end{duplicate}

\vspace*{1em}

This completes the structural topological results in \cite[Corollary 1.4]{GalazGarciaKellMondinoSosa2018}

\begin{rmk} Assume we are in the case that $X^{\ast}$ is closed interval. Let us denote by $G_\pm$ the isotropy groups of the singular orbits corresponding to the end point of the interval, and by $H$  the principal isotropy corresponding to the orbits in the interior of the interval. We point out that in the case of \th\ref{MC: Homeomorphism rigidity}~\eqref{MC: Homeomorphism rigidity close interval} together with Kleiner's \th\ref{L: Kleiners Lemma}  and Section~\ref{SS: Group Diagram}, the space $X$ and the group action  are determined by a group diagram: 
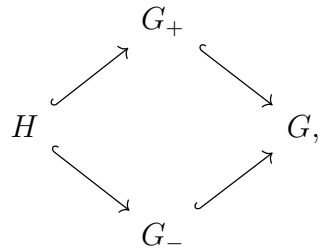
\[
\begin{tikzcd}
    & G_+ \arrow[dr,hook] & \\
   H \arrow[ur,hook] \arrow[dr,hook]& & G,\\
    & G_- \arrow[ur,hook]& 
\end{tikzcd}
\]
together with two $\RCD$-spaces $S_\pm$, that have a cohomogeneity one action by $G_\pm$ by measure preserving isometries. Moreover, in the case when $X$ is non-collapsed, $S_\pm$ is the cone over a homogeneous space $Z_\pm = \bar{G}_\pm / \bar{H}_\pm$ with effective an transitive action of $G_\pm$, respectively. This is in contrast to context of Riemannian manifolds and Alexandrov spaces, where the homogeneous spaces $Z_\pm$ are determined by the subgroups $G_\pm$ and $H$. Namely for Riemannian manifolds and Alexandrov spaces we have $S_\pm = G_\pm / H$.
\end{rmk}

With \th\ref{T: Topological classification geodesic coho 1 spaces.} we can compute the fundamental group of the geodesic spaces of cohomogeneity one, following \cite{Hoelscher2007}.

\begin{theorem}\th\label{T: fundamental group}
Let $(X,\dis, \m)$ be an essentially non-branching space, and $G$ a compact Lie group acting by isometries, such that $X^\ast$ is isometric to $[-1,1]$. We also consider a tuple $(G,H,K_+,K_-)$ given by \th\ref{T: Principal isotropy is subgroup of singular isotropy} and assume that the orbits $K_\pm/H$ are connected. Then the fundamental group of $X$ is isomorphic to
\[
\big((\pi_1(G/H)/N_+) \ast (\pi_1(G/H)/N_-)\big)/N(\mathrm{span}\{(\rho_+)_\ast(\omega)(\rho_+)_\ast(\omega)^{-1}\mid \omega \in \pi_1(G/H)\}),
\]
where $N_\pm = \ker\{(\rho_\pm)_\ast\colon \pi_1(G/H)\to \pi_1(G/H)\}$ for the orbit projection maps $\rho_\pm\colon G/H\to G/K_\pm$ give by $\rho_\pm(gH):=gK_\pm$.
\end{theorem}

\begin{proof}
We recall that the group tuple $(G,H,G_+,G_-)$ is associated to a horizontal geodesic \linebreak$\gamma\colon [0,1]\to X$. Set $x_- = \gamma_0, x_+= \gamma_1$, and $x_0=\gamma_{1/2}$. Then $x_\pm \in \pi^{-1}(\pm 1)$ and $x_0\in\pi^{-1}(0)$.

We consider $U_+^\ast,U_-^\ast\subset [-1,1]$ with $U_+^\ast\cap U_-^\ast = (-\varepsilon,\varepsilon)$, and apply van Kampen's Theorem to the decomposition of $X$ by the open subsets $U_-:= \pi^{-1}(U_-^\ast)$, $U_+:= \pi^{-1}(U_+^\ast)$, and base point $x_0$. 

We point out that the slices $S_{x_\pm}$ are contractible, and thus the tubular neighborhoods $\mathrm{Tub}(\pi^{-1}(\pm 1))\cong G\times_{G_\pm} S_{x_{\pm}}$ onto the orbits $G(x_\pm)$. Moreover, the point $x_0$ goes to the points $x_\pm$ under this deformation retract. Thus we have  that $U_\pm \simeq \mathrm{Tub}(\pi^{-1}(\pm 1))\simeq \pi^{-1}(\pm 1)\cong G/K_\pm$, and $U_-\cap U_+ \simeq \pi^{-1}(0)\cong G/H$. Under this homotopy equivalences we have that $x_0$ is mapped to the classes $K_\pm\in G/K_\pm$ and $H\in G/H$. Thus we have the following commutative diagram:
\[
\begin{tikzcd}
    (U_+\cap U_-,x_0) \arrow[r, hookrightarrow]\arrow[d,swap,"\simeq"] & (U_\pm,x_0)\arrow[d,"\simeq"]\\
    (G/H,H)\arrow[r,swap,"\rho_\pm"] & (G/K_\pm,K_\pm),
\end{tikzcd}
\]
where $j_\pm\colon (U_+\cap U_-,x_0)\hookrightarrow (U_\pm,x_0)$ is the inclusion map, and $(G/H,H)\to (G/K_\pm,K_\pm)$ is the map given by $gH\mapsto gK_\pm$. Thus, instead of the induced  maps $(j_\pm)_\pm  \colon \pi_1(U_+\cap U_-,x_0)\to (U_\pm,x_0)$ we can consider the maps $(\rho_\pm)_\ast\colon \pi_1(G/H,H)\to \pi_1(G/K_\pm,K_\pm)$

Using the fibration long exact sequence of homotopy groups:
\[
\cdots \pi_1(G/H,H)\overset{(\rho_\pm)_\ast}{\longrightarrow} \pi_1(G/K_\pm,K_\pm)\to \pi_0(K_\pm/H,H)\to \cdots
\]
we conclude that  $\pi_1(G/K_\pm,K_\pm)\cong \pi_1(G/H,H)/\ker \pi_1(G/H,H)$.

Putting all together, we obtain by van Kampen's Theorem that 
\begin{linenomath}
\begin{align*}
\pi_1(X,x_0)&\cong (\pi_1(U_+,x_0)\ast \pi_1(U_-,x_0))/N(\mathrm{span}\{(j_+)_\ast(\omega)(j_-)_\ast(\omega)^{-1}\mid \omega_\in \pi_1(U_+\cap U_-,x_0)\})\\
&\cong \big((\pi_1(G/H)/N_+) \ast (\pi_1(G/H)/N_-)\big)/N_0, 
\end{align*}
where $N_0=N(\mathrm{span}\{(\rho_+)_\ast(\omega)(\rho_+)_\ast(\omega)^{-1}\mid \omega \in \pi_1(G/H)\})$, i.e. the normalizer of the subgroup generated by $(\rho_+)_\ast(\omega)(\rho_+)_\ast(\omega)^{-1}$.
\end{linenomath}
\end{proof}

From \th\ref{T: fundamental group} we obtain the following sufficient conditions for a compact  essentially non-branching cohomogeneity one space to be simply connected.

\begin{cor}
Let $(X,\dis, \m)$ be an essentially non-branching  space, and $G$ a compact Lie group acting by isometries, such that $X^\ast$ is isometric to $[-1,1]$. We also consider a tuple $(G,H,G_+,G_-)$ given by \th\ref{T: Principal isotropy is subgroup of singular isotropy} and assume that the orbits and $K_\pm/H$ are connected. If the inclusions 
\[
(i_\pm)_\ast\colon  \pi_1(K_\pm/H)\to \pi_1(G/H)
\]
are surjective then $X$ is simply-connected.
\end{cor}

\begin{proof}
By the fibration long exact sequence of homotopy groups:
\[
\cdots\to \pi_1(K_\pm/H,H)\overset{(i_\pm)_\ast}{\longleftrightarrow} \pi_1(G/H,H)\overset{(\rho_\pm)_\ast}{\longrightarrow} \pi_1(G/K_\pm,K_\pm)\to \pi_0(K_\pm/H,H)\to \cdots
\]
we have that $\ker (\rho_\pm)_\ast = \mathrm{img}(i_\pm)_\ast$. Then the conclusion follows from \th\ref{T: fundamental group}.
\end{proof}

\section{Construction of cohomogeneity one \texorpdfstring{$\RCD$}{RCD}-spaces}
\label{S: Gluing of RCD-spaces}
In this Section we  show how to construct a cohomogeneity one RCD space from a given cohomogeneity one group diagram, which we define next: 

\begin{definition}\th\label{D: Coho one diagram}[Cohomogeneity one diagram]
Given a compact Lie group $G$, and Lie subgroups $H$, $G_\pm$ of $G$ satisfying the following diagram
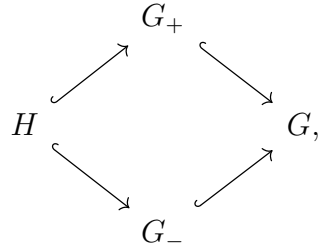
\[
\begin{tikzcd}
    & G_+ \arrow[dr,hook] & \\
   H \arrow[ur,hook] \arrow[dr,hook]& & G,\\
    & G_- \arrow[ur,hook]& 
\end{tikzcd}
\]
we say that it is a \emph{cohomogeneity one group diagram} if the homogeneous spaces $G_{\pm}/H$ admit $G_\pm$-invariant Riemannian metrics with positive Ricci curvature when $\dim(G_\pm/H)\geq 2$, or have finite diameter when $\dim(G_\pm/H)=1$. 
\end{definition}

\begin{rmk}
We recall that by the characterization in \cite{Berestovskii1995}, a homogeneous space $G/H$ admits a $G$-invariant metric of positive Ricci curvature if and only if it has finite fundamental group.
\end{rmk}

We now state our construction theorem.
\vspace*{8pt}

\begin{duplicate}[\ref{thm.rcd space from group diagram}]
Let $G$ be a compact Lie group and $G_+$, $G_-$, $H$ Lie subgroups of $G$ such that $(G,H,G_+,G_-)$ is a cohomogeneity one group diagram. Then, there exists an $\RCD(K,N)$-space $(X,\dis,\m)$ admitting a cohomogeneity one action of $G$ by measure preserving isometries, such that the associated group diagram is $(G,H,G_+,G_-)$. Moreover, we have $K\leq 0$, and $N=\max\{n_-+1+D,n_+ +1+ D, n_0+1\}$, where $n_\pm = \dim(G_\pm/H)$, $n_0 =\dim (G/H)$ and $D= \dim(G)$.
\end{duplicate}

\subsection{Cone-like warped products.}

In this section we will work with certain warped products over compact Riemannian manifolds. The distances of the warped products that we will be considering will always be defined on $[0,\infty)\times X /\sim$, where $(t,x)\sim (s,y)$ if and only if $t=s=0$. The warping function $f$ that we consider only vanish at $t=0$. 

The next result tells us that under appropriate assumptions on the warping function we can obtain convexity of the closed balls around a particular point, namely, the ``vertex'' of the warped product. 

\begin{prop}\th\label{prop.convexconelike}
Let $(X,\dis)$ be a compact geodesic space and take $f\colon [0,\infty)\rightarrow [0,\infty)$ a continuous non-decreasing function such that $f(0)=0$. Then for every $R>0$ we have that $\bar{B}_R(0)= [0,R]\times_{f}X \subset [0,\infty]\times_{f}X$ is convex. 
\end{prop}

\begin{proof}
Let $R>0$, and consider $\gamma\colon [0,1]\rightarrow [0,\infty)\times_{f}X$ an admissible curve such that $\gamma_0,\gamma_1 \in [0,R]\times_{f}X$ but that $\gamma([0,1])\not\subset [0,R]\times_{f}X$. The goal is to find a curve $\tilde{\gamma}$ contained in $[0,R]\times_{f}X$ with length less than or equal to that of $\gamma$. 
We have $\gamma = (\alpha,\beta)$, where $\alpha \in AC([0,1], [0,\infty])$. Recall that the absolute value of an absolutely continuous function is still an absolutely continuous function. We define the admissible curve $\tilde{\gamma} =(\tilde{\alpha},\beta)$ where $\tilde{\alpha}(t):=\max\lbrace 0, R-|\alpha(t)-R|\rbrace$. Observe that $f^2\circ\tilde{\alpha}(t)\leq f^2\circ \alpha (t)$ for all $t \in [0,1]$ and that $|\dot{\tilde{\alpha}}_t|\leq|\dot{\alpha}_t|$ on a set of full measure. Therefore 
\[
\ell(\tilde{\gamma})= \int^1_0 \sqrt{|\dot{\tilde{\alpha}}_t|^2+f^2\circ\tilde{\alpha}(t)|\dot{\beta}_t|^2}dt \leq \int^1_0 \sqrt{|\dot{\alpha}_t|^2+f^2\circ\alpha(t)|\dot{\beta}_t|^2}dt= \ell(\gamma)
\]
and so we have the admissible curve we were looking for in $[0,R]\times X$. Hence we have the result.
\end{proof}

For our purposes we will consider the warping function $f_A\colon [0, \infty) \rightarrow [0,\infty) $ given by $f_A(x):= \min \{1,Ax\}$ for $A>0$. Observe that $f_A$ is continuous but not smooth at $x=1/A$. Then, given $\epsilon>0$ we define smooth functions $f_{A,\epsilon}$ to approximate $f_A$, by setting:
\[
f_{A,\epsilon}(x) := \frac{1}{2}\left((Ax+1)-\sqrt{(Ax-1)^2+\epsilon^2}-(1-\sqrt{1+\epsilon^2})\right).
\]

The following lemma is straightforward to prove. 

\begin{lemma}
We have for all $x\in [0,\infty)$ that $f_{A,\epsilon}(x)\geq 0$, $f_{A,\epsilon}(x)\leq (1+\sqrt{1+\epsilon^2})/2$. It also holds that for $x\in [0,2/A]$,  $f_{A,\epsilon}(x)\leq f_A(x)$. Moreover
\[
|f_{A,\epsilon}(x)-f_A(x)|\leq \epsilon-(1-\sqrt{1-\epsilon^2})/2.
\]
Thus $f_{A,\epsilon}$ converges uniformly to $f_A$ as $\epsilon\to 0$.
\end{lemma}

We consider now warped products using the function $f_{A,\epsilon}$ as warping function, with the objective of showing pointed measured Gromov-Hausdorff convergence of these warped spaces to the warped space with respect to $f_A$. In view of the convexity granted by \th\ref{prop.convexconelike}, and because of the specific properties of the $f_{A,\epsilon}$, henceforward we  restrict ourselves to $[0,2/A]\times_{f_{ A,\epsilon} }X$ and $[0,2/A]\times_{f_A }X$. Let $[t,x],[s,y]\in [0,2/A]\times X/\sim$ and consider the set
\begin{linenomath}
\begin{align*}
    \mathcal{A}:= \big\{ \gamma\colon [0,1]\to [0,2/A]\times X/\sim \mid &\, \gamma = (\alpha,\beta),\: \alpha\in AC([0,1],[0,2/A]),\\
    & \beta\in AC([0,1],X),\: \gamma_0 = [t,x],\gamma_1=[s,y] \big\}.
\end{align*}
\end{linenomath}

We set $\bar{\mathcal{A}}$ to be the closure of $\mathcal{A}$ with respect to the compact-open topology. Clearly $\bar{\mathcal{A}}= \bar{\mathcal{A}}([t,x],[s,y])$ (that is, $\bar{\mathcal{A}}$ depends on the points in question), but unless necessary we omit this in the notation. For simplicity, in the following we drop the dependence on the parameter $A$ from the notation, and simply write $f$ and $f_{\epsilon}$ instead of $f_A$ and $f_{A,\epsilon}$.

Similarly to what happens with $f_\epsilon$, $f$ we have the following result. 

\begin{lemma}\th\label{lemma.increasing.lengths}
    The length functionals $\ell_{\epsilon}$ and $\ell$ of the warped products $[0,2/A]\times_{f_{\epsilon} }X$ and $[0,2/A]\times_{f }X$, respectively,  satisfy the following for all curves $\gamma \in \bar{\mathcal{A}}$:
    \begin{enumerate}
    \item $\ell_\epsilon(\gamma) \leq \ell(\gamma)$, for all $\epsilon >0$;
    \item $\ell_{\epsilon_1}(\gamma)\leq \ell_{\epsilon_2}(\gamma)$, for all $0<\epsilon_2<\epsilon_1$.   
    \end{enumerate}
\end{lemma}

\begin{proof}
We first prove $(1)$. Fix $\epsilon >0$ and $\gamma = (\alpha, \beta)\in \mathcal{A}$. We know from the fact that $f_{\epsilon}\leq f$ for all $x\in [0,2/A]$ that 
\[
\sqrt{|\dot{\alpha}_t|^2+f^2_{\epsilon}\circ\alpha(t)|\dot{\beta}_t|^2}\leq\sqrt{|\dot{\alpha}_t|^2+f^2\circ\alpha(t)|\dot{\beta}_t|^2}
\]
on the set of full measure where both $|\dot{\alpha}_t|$ and $|\dot{\beta}_t|$ are defined. On the remaining null set the functions are arbitrarily defined so they pose no problem. Now, integrating we get
\[
\int^1_0\sqrt{|\dot{\alpha}_t|^2+f^2_{\epsilon}\circ\alpha(t)|\dot{\beta}_t|^2}dt\leq\int^1_0\sqrt{|\dot{\alpha}_t|^2+f^2\circ\alpha(t)|\dot{\beta}_t|^2}  dt.
\]
Which implies that $\ell_\epsilon \leq \ell$. The proof of $(2)$ is carried out by a similar argument.
\end{proof}

We now establish pointwise convergence of the length functionals $\ell_\epsilon$ to $\ell$ as $\epsilon\to 0$. 

\begin{prop}
The length functionals $\ell_\epsilon$ converge pointwise to $\ell$ when $\epsilon\to 0$.    
\end{prop}

\begin{proof}
Take $\gamma \in \mathcal{A}$ and notice that for all $\epsilon >0$ the functions 
\[
t \mapsto \sqrt{|\dot{\alpha}_t|^2+f^2_{\epsilon}\circ\alpha(t)|\dot{\beta}_t|^2}
\]
are non-negative and measurable. Furthermore we have that for all $\epsilon_1<\epsilon_2$

$0\leq \sqrt{|\dot{\alpha}_t|^2+f^2_{\epsilon_2}\circ\alpha(t)|\dot{\beta}_t|^2}\leq \sqrt{|\dot{\alpha}_t|^2+f^2_{\epsilon_1}\circ\alpha(t)|\dot{\beta}_t|^2} \leq \infty $     
And that due to the uniform convergence of the $f_\epsilon$, the pointwise limit is 
\[
\sqrt{|\dot{\alpha}_t|^2+f^2\circ\alpha(t)|\dot{\beta}_t|^2}=\lim_{\epsilon \rightarrow 0}\sqrt{|\dot{\alpha}_t|^2+f^2_{\epsilon}\circ\alpha(t)|\dot{\beta}_t|^2}.
\]
Then, by Beppo Levi's Lemma (exchanging integration with supremum) we obtain:
\[
\lim_{\epsilon\rightarrow 0}\int^1_0 \sqrt{|\dot{\alpha}_t|^2+f^2_{\epsilon}\circ\alpha(t)|\dot{\beta}_t|^2}dt = \int^1_0\sqrt{|\dot{\alpha}_t|^2+f^2\circ\alpha(t)|\dot{\beta}_t|^2}dt.
\]
Which is just what we wanted.
\end{proof}

In order to prove the pointed measured Gromov-Hausdorff convergence, we show that the induced distances converge pointwise; thus we  need to check the convergence of the infima of the lengths $\ell_\epsilon$ on $\bar{\mathcal{A}}$ to the infimum of the length $\ell$ on $\bar{\mathcal{A}}$. To do so we make use  of the notion of $\Gamma$-convergence. The interested reader may consult \cite{Braides2006} for further details.

\begin{prop}\th\label{prop.Gammaconvergence.lengths}
 The length functionals $\ell_\epsilon$ $\Gamma$-converge to $\ell$ as $\epsilon\to 0$.   
\end{prop}

\begin{proof}
Recall that the length functionals are lower semi-continuous. Now, from \th\ref{lemma.increasing.lengths} we know that $\ell_\epsilon$ are non-decreasing with respect to $\epsilon$. Then using \cite[Remark 2.12 ii)]{Braides2006} we get that
\[
\Gamma\text{-}\lim_{\epsilon \rightarrow 0} \ell_\epsilon = \lim_{\epsilon \rightarrow 0}\ell_\epsilon = \ell.
\]    
\end{proof}
The next ingredient we need is the equicoercivity of the family $\{\ell_{\epsilon}\}$ (see \cite[Definition 2.9]{Braides2006}), which we now prove. 
\begin{prop}\th\label{prop.equicoercivity.lengths}
The sequence $\lbrace \ell_\epsilon \rbrace_{0<\epsilon \leq 1}$ is equicoercive.
\end{prop}

\begin{proof}
Let $\lambda \geq 0$ and  consider the set $K_{\lambda}:= \lbrace \gamma \mid \ell_1(\gamma) \leq 2\lambda \rbrace$. As we are dealing only with curves in compact spaces, \cite[Theorem 2.5.14]{BuragoBuragoIvanov} tells us that $K_\lambda$ is compact.
We will only work with $0<\epsilon\leq 1$. Take $\gamma =(\alpha, \beta) \in \lbrace \gamma \mid \ell_{\epsilon}(\gamma)\leq \lambda\rbrace$. It is immediate to see that 
\[
\int^1_0 |\dot{\alpha}_t|dt \leq \lambda, \quad \int^1_0 f_{\epsilon}\circ \alpha (t)|\dot{\beta}_t|dt \leq \lambda.
\]
Since $f_1\leq f_\epsilon$ for all $x\in [0,2/A]$, it follows that 
\[
\int^1_0 f_{1}\circ \alpha (t)|\dot{\beta}_t|dt \leq\int^1_0 f_{\epsilon}\circ \alpha (t)|\dot{\beta}_t|dt \leq \lambda.\]
So then 
\[
\int^1_0 \sqrt{|\dot{\alpha}_t|^2+f^2_1\circ\alpha(t)|\dot{\beta}_t|^2}dt\leq \int^1_0|\dot{\alpha}_t|dt +\int^1_0 f_1\circ\alpha(t)|\dot{\beta}_t|dt \leq 2\lambda.
\]
Hence we have equicoercivity for $0<\epsilon \leq 1$.
\end{proof}

We now have all the ingredients needed to prove the convergence of the distances. Before moving on to the next result, let us point out that so far we have only considered a fixed topology on $\bar{\mathcal{A}}$. However in order to apply the Fundamental Theorem of $\Gamma$-convergence we need to fix a metric on $\bar{\mathcal{A}}$ that induces the open-compact topology. 

Note that for all $\epsilon >0$ we have that $f\leq 2f_{\epsilon}$. This is useful because it implies that the $f$-length $\ell$ is bounded from above by the $2f_{\epsilon}$-length, that is, $\ell \leq \ell_{2f_\epsilon}$. This will in turn yield that $\dis\leq \dis_{2f_\epsilon}$. Now, if $\gamma$ is an admissible curve then it has small $f_\epsilon$-length if and only if its $2f_{\epsilon}$-length is also small. Then, the induced distances $\dis_\epsilon$, $\dis$ are equivalent. All this allows us to put on $\bar{\mathcal{A}}$ the sup norm induced by $\dis$ without any problem.

\begin{prop}
The distances $\dis_\epsilon$ induced by the lengths $\ell_\epsilon$ converge uniformly to the distance $\dis$ induced by $\ell$.
\end{prop}

\begin{proof}
Let $[t,x],[s,y] \in [0,1]\times X/\sim$. Then from \th\ref{prop.Gammaconvergence.lengths} we have that the functionals $\ell_\epsilon\colon \bar{\mathcal{A}}\rightarrow [0,\infty]$ $\Gamma$-converge to the functional $\ell\colon\bar{\mathcal{A}}\rightarrow [0,\infty]$. Now \th\ref{prop.equicoercivity.lengths} tells us that for $0<\epsilon \leq 1$ we have equicoercivity, so then
we can apply the Fundamental Theorem of $\Gamma$-convergence (see \cite[Theorem 2.10]{Braides2006}) and obtain that:
\[
\inf_{\bar{\mathcal{A}}}\ell = \lim_{\epsilon \rightarrow 0}\inf_{\bar{\mathcal{A}}}\ell_\epsilon,
\]
or equivalently that $\dis_{\epsilon}([t,x],[s,y])\rightarrow \dis([t,x],[s,y])$ as $\epsilon \rightarrow 0$.  
Since the points in $[0,1]\times X/\sim$ were chosen arbitrarily we have pointwise convergence of the distance functions.
Now, it is easy to convince oneself that the following properties hold:
\begin{itemize}
\item $\dis_\epsilon \leq \dis$ for all $\epsilon >0$.
\item $\dis_{\epsilon_2} \leq \dis_{\epsilon_1}$ for all $0<\epsilon_1<\epsilon_2$.
\end{itemize}
To conclude just observe that $\dis_\epsilon$, $\dis$ are continuous functions on a compact space so then by Dini's Theorem we have that the convergence must be uniform.
\end{proof}

We are now able to show the measured Gromov-Hausdorff convergence of the warped product spaces. 

\begin{prop}\th\label{prop.mGHconvergencewarped}
The warped products $([0,2/A]\times_{f_{\epsilon}}X,\dis_{\epsilon})$ equipped with the reference measures $\m_\epsilon := f^N_\epsilon dt\otimes \m_X$ for $N>1$ converge in the measured Gromov-Hausdorff topology to the warped product  $([0,2/A]\times_{f}X,\dis)$ equipped with the measure $\m:= f^N dt\otimes d\m$.  
\end{prop}

\begin{proof}
We show that the identity map $\mathrm{Id} \colon [0,2/A]\times_{f_{\epsilon}} X\rightarrow [0,2/A]\times_{f}X$ is a Gromov-Hausdorff-approximation. Define
\[
C_\epsilon = \dfrac{\int_X\!\int_0^{2/A} f^{N}\, \mathrm{d}t\,\mathrm{d}\m}{\int_X\!\int_0^{2/A} f^{N}_\epsilon \mathrm{d}t\,\mathrm{d}\m}\geq 1.
\]
By the Dominated Convergence Theorem it is clear that as $\epsilon \rightarrow 0$, then $C_\epsilon \rightarrow 1$. With these constants we have that the measures $C_\epsilon f^{N}_\epsilon dt\otimes d\mathfrak{m}$ and $f^{N}dt\otimes d\mathfrak{m}$ have the same mass. 

Let $\delta >0$, then for all $\epsilon$ sufficiently small we have that 
\[
\left|C_\epsilon-1\right|<\delta,\quad \left|f^N(t)-f^{N}_\epsilon (t)\right|<\delta,\quad \mbox{and}\: \left|\dis([t,x],[s,y])-\dis_\epsilon([t,x],[s,y])\right|<\delta
\]
hold for all $[t,x],[s,y]\in [0,1]\times X/\sim$. Observe that this last condition tells us that $\mathrm{Id}$ is a $\delta$-Gromov-Hausdorff approximation for sufficiently small $\epsilon$.

We now show the weak convergence of measures. Take $\varphi \in C_{b}([0,2/A]\times_{f}X)$ and note that, since over $[0,2/A]$ we have $0\leq f_\epsilon(t)\leq 1$ then
\begin{linenomath}
\begin{align*}
\left|\int (f^N(t)-C_{\epsilon}f^N_{\epsilon}(t))\varphi(t,x)\, \mathrm{d}t\,\mathrm{d}\mathfrak{m}\right|\leq&\int|f^N(t)-C_{\epsilon}f^N_\epsilon(t)||\varphi|(t,x)\, \mathrm{d}t\,\mathrm{d}\mathfrak{m}\\
\leq& \int |f^N(t)-f^{N}_{\epsilon}(t)||\varphi|(t,x) \, \mathrm{d}t\,\mathrm{d}\mathfrak{m}\\
&+\int |f_\epsilon^N (t)-C_{\epsilon}f_\epsilon^N (t)||\varphi|\, \mathrm{d}t\,\mathrm{d}\mathfrak{m}\\
\leq& \int |f^N(t)-f^{N}_{\epsilon}(t)||\varphi|(t,x) \, \mathrm{d}t\,\mathrm{d}\mathfrak{m}\\
&+\int |f_\epsilon^N (t)||1-C_{\epsilon}||\varphi|\, \mathrm{d}t\,\mathrm{d}\mathfrak{m}\\
\leq& \int |f^N(t)-f^{N}_{\epsilon}(t)||\varphi|(t,x) \, \mathrm{d}t\,\mathrm{d}\mathfrak{m}\\
&+\int |C_{\epsilon}-1||\varphi|\, \mathrm{d}t\,\mathrm{d}\mathfrak{m}\\
\leq& 2\delta\int|\varphi|(t,x)\, \mathrm{d}t\,\mathrm{d}\mathfrak{m}.
\end{align*}
\end{linenomath}
So we also have the convergence of the measures. Hence we have the sought measured Gromov-Hausdorff convergence.
\end{proof}

We now prove that when we require $X$ to have small diameter, then near the boundary the metric on $[0,2/A]\times_{f}X$ is the product metric.

\begin{prop}\th\label{prop.near bdry metric is product in warped product}
Let $(X,\dis_X)$ be a compact geodesic metric space with $\diam(X) < 1/A$. Then there exists $\delta >0$ such that the metric of the warped product $[0,2/A]\times_{f}X$ restricted to $[2/A-\delta,2/A]\times_{f}X$  is the product metric.    
\end{prop}

\begin{proof}
First we prove that for $t$ sufficiently close to $2/A$ the geodesics joining points in the set $\lbrace t \rbrace\times X$ are always contained in $[1/A,2/A]\times X$. Since the warping function $f$ equals $1$ for all $t\in [1/A,2/A]$ this will yield that $\lbrace t \rbrace\times X$ is actually totally geodesic. 
 
Proceeding by contradiction we assume that for all $\delta>0$ there exists $t\in [2/A -\delta,2/A]$ such that there are points $[t,x],[t,y]\in \lbrace t \rbrace\times X$ such that a minimal geodesic $\gamma$ joining them is not contained in $[1/A,2/A]\times X$.  
First observe that this means that $d([t,x],[t,y])< \diam(X)<1/A$ with the diameter in terms of the distance in $X$. Define the points
\[
t_0:=\min\lbrace t\,|\, \gamma_{t}\in \lbrace 1/A\rbrace \times X \rbrace < t_1:=\max\lbrace t\,|\, \gamma_{t}\in \lbrace 1/A\rbrace \times X \rbrace.
\]
 It is clear that
\[
d([t,x],[t,y]) = d([t,x],\gamma_{t_0})+d(\gamma_{t_0},\gamma_{t_1})+d(\gamma_{t_1},[t,y]).
\]
Now, since $\gamma_{t_0}\in \lbrace1/A\rbrace\times X$ we have that 
\[
d([t,x],\gamma_{t_0})\geq d([t,x],\lbrace1/A\rbrace\times X)= d([t,x],[1/A,x])=|1/A-t|\geq 1/A-\delta,
\]
where the last inequality follows from recalling that $t\in [2/A-\delta,2/A]$. Analogously, we obtain the same lower bound for $d(\gamma_{t_1},[t,y])$. Thus for all $\delta>0$ we have that 
\[
1/A>d([t,x],[t,y]) \geq 2/A-2\delta+d(\gamma_{t_0},\gamma_{t_1}).
\]
This implies that $1/A> 2/A+d(\gamma_{t_0},\gamma_{t_1})$, which is a contradiction. Hence we must have that there exists some $\delta >0$ such that for all $t\in [2/A-\delta,2/A]$ the set $\lbrace t\rbrace\times X$ is totally geodesic. 
 
We will now prove that geodesics joining points in $[2/A-\delta,2/A]\times X$ must lie there. Again, we proceed by contradiction: take two points $[r,x],[s,y]$ with $r<s$, and a minimal geodesic $\gamma$ with $\gamma_0 =[r,x], \gamma_1 =[s,y]$ such that it is not contained in $[2/A-\delta,2/A]\times X$. Now consider $t_0 :=\max \lbrace t\,|\, \gamma_{t}\in \lbrace r\rbrace\times X  \rbrace$. 
Notice that then $\gamma([t_0,1])\subset [2/A-\delta,2/A]\times X$. The restriction of $\gamma$ to the interval $[0,t_0]$ can be reparametrized to obtain a geodesic between $[r,x]$ and $\gamma_{t_0}$. However, the previous argument implies then that $\gamma([0, t_0])\subset [2/A-\delta,2/A]\times X$, giving the contradiction.
 
Finally, as we have proved that geodesics between points in $[2/A-\delta,2/A]\times X$ must lie in there, it follows that the metric must be the product one.
\end{proof}

We are now tasked with looking at the structure of the Sobolev space $W^{1,2}$ of the warped product we have considered. We use the tools developed in \cite{GigliHan2018} where  Gigli and Han give conditions to ensure that the minimal weak upper gradient of a Sobolev function $g$ in the warped product $[0,\infty)\times_{f}X$ can be described in terms on the minimal weak upper gradients of the functions $g^{(t)}:=g(t,\cdot)$  and $g^{(x)}:= g(\cdot, x)$.

\begin{prop}\th\label{prop.infHilbertwarped}
Let $(X,\dis,\mathfrak{m})$ be an infinitessimally Hilbertian space, and consider $f\colon [0,\infty)\rightarrow [0,\infty)$ to be a warping function such that $\lbrace f= 0 \rbrace$ is discrete and $f$ decays at least linearly  near its zeroes (see equations 3.14, 3.15 in \cite{GigliHan2018}). Then the warped product $[0,\infty)\times_{f}X $ is infinitessimally Hilbertian.
\end{prop}

\begin{proof}
Take $h,g \in W^{1,2}([0,1]\times_{f}X)$. Then by  \cite[Propositions 3.10 and 3.12]{GigliHan2018},  we have that the minimal weak upper gradient of $h$ is given by:
\[
|\nabla h|^2(t,x) = f^{-2}(t)|\nabla h^{(t)}|^2(x)+|\nabla h^{(x)}|^2(t),\quad f^Ndt\otimes \mathfrak{m}-\text{a.e.}
\]
And similarly for $g$.  \cite[Theorem 4.3.3]{GigliPasqualetto2020} states that $W^{1,2}([0,1]\times_{f}X)$ being a Hilbert space is equivalent to a pointwise parallelogram law. So we just need to check that this parallelogram law is satisfied from the hypothesis  that both $[0,1]$ and $X$ satisfy such a parallelogram law:
\begin{linenomath}
 \begin{align*}
     2\left(|\nabla h|^2(t,x)+|\nabla g|^2(t,x)\right) =& 2\big(f^{-2}(t)|\nabla h^{(t)}|^2(x)+|\nabla h^{(x)}|^2(t)\\
     &+f^{-2}(t)|\nabla g^{(t)}|^2(x)+|\nabla g^{(x)}|^2(t)\big)\\
     =& f^{-2}(t)2\left(|\nabla h^{(t)}|^2(x)+|\nabla g^{(t)}|^2(x)\right)\\
     &+2\left(|\nabla h^{(x)}|^2(t)+|\nabla g^{(x)}| ^2(t)\right)\\
     =& f^{-2}(t)\left(|\nabla (h+g)^{(t)}|^2(x)+|\nabla (h-g)^{(t)}|^2(x)\right)\\
     &+\left(|\nabla (h+g)^{(x)}|^2(t)+|\nabla (h-g)^{(x)}|^2(t)\right)\\
     =& |\nabla (h+g)|^2(t,x)+|\nabla (h-g)|^2(t,x).
 \end{align*}   
\end{linenomath}
\end{proof}


We now turn our attention to the problem of finding a curvature-dimension bound for the warped products $([0,2/A]\times_{f}X,\dis)$, when we assume that $X$ satisfies a curvature dimension bound. 

Recall that given $K\in \R$ and $B$ a geodesic space, we say that a function $f\colon B\to [0,\infty)$ is \emph{$\mathcal{F}K$-concave} if for every geodesic $\gamma$ we have:
\begin{linenomath}
\begin{align}
(f\circ \gamma)''(t)+K(f\circ \gamma)(t)\leq 0.
\end{align}
\end{linenomath}

With this definition in hand we recall the following result of Ketterer (building upon the work of Alexander and Bishop \cite{AlexanderBishop2016}), regarding curvature bounds and $\mathcal{F}K$-concavity. 

\begin{theorem}[Theorem A in \cite{KettererPhD}]\th\label{T: Ketterer warped products}
Let $B$ be a complete $d$-dimensional (local) Alexandrov space of $\mathrm{curv}\geq K$, such that $B\setminus \partial B$ is a smooth Riemannian manifold. Let $f\colon B\to [0,\infty)$ be $\mathcal{F}K$-concave and smooth over $B\setminus \partial B$. Assume that $\partial B\subseteq f^{-1}(0)$. Let $(F,\m_F)$ be a weighted complete Finsler manifold. Let $N\geq 1$, and $K_F\in \R$. If $N=1$ and $K_F>0$, we assume that $\mathrm{diam}(F)\leq \pi/\sqrt{K_F}$. In  any case, $F$ satisfies the $\mathsf{CD}^\ast((N-1)K_F,N)$-condition. Moreover, we assume that 
\begin{enumerate}
\item if $\partial B=\varnothing$, suppose $K_F\geq Kf^2$;
\item if $\partial B\neq \varnothing$, suppose $K_F\geq 0$ and $|\nabla f|_p\leq \sqrt{K_F}$ for all $p\in \partial B$.
\end{enumerate}
Then the $N$-warped product $B\times_f^N F$ satisfies the $\CD^\ast((N+d-1)K,N+d)$-condition.
\end{theorem}

As we apply this result in our context, we first show the $\mathcal{F}K$-concavity of the approximating functions $f_{A,\epsilon}$ for $K=0$. 

\begin{lemma}\th\label{L: F_A varepsilon are FK concave}
The functions $f_{A,\epsilon}$ are $\mathcal{F}K$-concave for $K=0$.
\end{lemma}

\begin{proof}
We compute the second derivatives of $f_{A,\epsilon}$:
\begin{linenomath}
\begin{align*}
f_{A,\epsilon}''(x) =& \frac{d}{dx}\left(\frac{1}{2}\left(A-A(Ax-1)((Ax-1)^2+\epsilon^2)^{-\frac{1}{2}}\right)\right)\\
=& -\frac{A}{2}\Big(A((Ax-1)^2+\epsilon^2)^{-\frac{1}{2}}-(Ax-1)\frac{1}{2}((Ax-1)^2+\epsilon^2)^{-\frac{3}{2}}2(Ax-1)A\Big)\\
=& -\frac{A^2}{2}\Big(((Ax-1)^2+\epsilon^2)^{-\frac{1}{2}}-(Ax-1)^2((Ax-1)^2+\epsilon^2)^{-\frac{3}{2}}\Big)\\
=& -\frac{A^2}{2}\Big(((Ax-1)^2+\epsilon^2)^{-\frac{3}{2}}((Ax-1)^2+\epsilon^2)-(Ax-1)^2((Ax-1)^2+\epsilon^2)^{-\frac{3}{2}}\Big)\\
=& -\frac{A((Ax-1)^2+\epsilon^2)^{-\frac{3}{2}}\epsilon^2}{2}\\
=& -\frac{A\epsilon^2}{\sqrt{((Ax-1)^2+\epsilon^2)^3}}\leq 0.
\end{align*}
\end{linenomath}
\end{proof}

Now we consider $(F,\dis_F,\m_F)$ to be a weighted Finsler (Riemannian) manifold satisfying the $\mathsf{CD}^\ast((N+1)K_F,N)$-condition (respectively, the $\RCD((N-1)K_F,N)$-condition). We also consider $B=[0,\infty)$.

\begin{theorem}\th\label{T: cone over f A sigma are CD}
Let $(F,\m_F)$ be a weighted Finsler manifold which satisfies the $\CD((N-1)K_F,N)$-condition for $K_F> 0$ such that, if $N=1$ then $\diam(F)\leq \pi/\sqrt{K_F}$, and $\sqrt{K_F}\geq (A/2)(1+(1/\sqrt{1+\epsilon^2}))$. Then the $N$-warped product $[0,\infty)\times^N_{f_{A,\epsilon}} F$ is an $\CD^\ast(NK,N+1)$-space for $K=0$. Moreover, if $(F,\m_F)$ is infinitesimally Hilbertian, then $[0,\infty)\times^N_{f_{A,\epsilon}} F$ is an $\RCD(NK,N+1)$-space for $K=0$.
\end{theorem}

\begin{proof}
The interval $[0,\infty)$ with the standard Euclidean metric is an Alexandrov space of $\curv\geq 0$. By  \th\ref{L: F_A varepsilon are FK concave} a we have that $f_{A,\epsilon}$ is $\mathcal{F}K$-concave for $K=0$. Moreover we have $|f_{A,\epsilon}'|\leq \sqrt{K_F}$. Thus the hypothesis of \th\ref{T: Ketterer warped products} are satisfied yielding the first claim. Moreover, when $(F,\m_F)$ is infinitesimally Hilbertian, we obtain the second conclusion from \th\ref{prop.infHilbertwarped}.
\end{proof}

With this result, we are able to show that spaces in \th\ref{MC: Homeomorphism rigidity}~\eqref{MC: homeomorphism rigidity ray} admit an $\RCD(0,N)$-structure on which $G$ acts by measure preserving isometrics of cohomogeneity one.

\begin{mtheorem}\th\label{MT: coho one ray with conditions is RCD 0,N}
Let $(X,\dis,\m)$ be an $\RCD(K,N)$-space, and let $G$ be a compact Lie group acting on $X$ by measure preserving isometries, such that the orbit space $X/G$ is homeomorphic to $[0,\infty)$. Let $G_0$ be the unique singular isotropy group, and $H$ the principal isotropy group. Under the assumption that $G_0/H$ is simply-connected when $\dim(G_0/H)\geq 2$,  there exists a new metric $\tilde{\dis}$ and measure $\tilde{\m}$ on $X$ making into an $\RCD(0,N')$-space, where $N'=\dim(G_0/H)+1+\dim(G)$. Moreover, $G$ acts on $(X,\tilde{\dis},\tilde{\m})$ by measure preserving isometries with cohomogeneity one.
\end{mtheorem}

\begin{proof}
We consider on $G$ a bi-invariant Riemannian metric. By \th\ref{L: infinitesimal action cohomogeneity 1} and \th\ref{MT: Geometry of the slice}, we have that cone fibers which are homeomorphic to the space $F = G_{0}/H$. As proven in \cite{Berestovskii1995}, the restriction of the $G$-bi-invariant Riemannian metric to $G_0$ induces a $G_0$ Riemannian metric of positive Ricci curvature on $F$ when $\dim(F)\geq 2$. In the case when $N_0=1$, since $G_0$ and $H$ are compact ($G$ is compact and acts by isometries, thus the isotropy groups are compact), then $F$ is compact. Thus, there exists $K_F>0$ such that $\diam(F)\leq \pi/\sqrt{K_F}$, and $F$ with the metric induced by  the restriction of the $G$-biinvariant metric to $G_0$ is an $\RCD(0,1)=\RCD((N_0-1)K_F,N_0)$-space, where $N_0:=\dim(F)$.

In  case when $N_0\geq 2$, the space $(F,\dis_0,\m_0)$ with the induced metric and measure is an $\RCD(K_0,N_0)$-space with $K_0>0$, and $N_0=\dim(F)$. Thus there exists $K_F>0$ such that $K_0=(N_0-1)K_F$. 

We consider now $A>0$ and $\varepsilon>0$ appropriately such that $\sqrt{K_F}\geq (A/2)(1+(1+\sqrt{1+\epsilon^2}))$. Then by \th\ref{T: cone over f A sigma are CD} we have that the $N_0$-warped product $[0,\infty)\times_{f_{A,\epsilon}} F$ is an $\RCD(0,N_0+1)$-space. Thus the product metric measure space $G\times ([0,\infty)\times_{f_{A,\epsilon}} F)$ is an $\RCD(0,N_0+1+D)$-space, where $D=\dim(G)$, and $G_0$ acts by measure preserving isometries via the diagonal action $\Delta\colon G\times ([0,\infty)\times_{f_{A,\epsilon}} F)$ given by 
\[
\Delta(g_0,(g,[t,p]):=(gg_0^{-1},[t,g_0p]).
\]
Thus by \cite[Theorem 1.1]{GalazGarciaKellMondinoSosa2018} the quotient space $(G\times ([0,\infty)\times_{f_{A,\epsilon}} F))/\Delta$ is an $\RCD(0,N_0+1+D)$ space. But by \th\ref{MC: Homeomorphism rigidity} this space is homeomorphic to $X$.
\end{proof}

We now consider $N$-warped products of the form $[0,2/A]\times_{f_A}^N F$:

\begin{theorem}\th\label{T: Special type of cone}
Let $(F,\dis_F,\m_F)$ be a weighted Finsler manifold which satisfies the \linebreak$\CD((N-1)K_F,N)$-condition for $K_F>0 $ such that, if $N=1$ then $\diam(F)\leq \pi/\sqrt{K_F}$. Assume that $K_F\geq A^2$. Then the $N$-warped product $[0,2/A]\times_{f_A}^N F$ is a $\CD^\ast(0,N+1)$-space, and an $\RCD(0,N+1)$-space if $(F,\dis_F,\m_F)$ satisfies the $\RCD((N-1)K_F,N)$-condition.
\end{theorem}

\begin{proof}
By the previous theorem, the $N$-warped products $[0,\infty)\times_{f_{A,\epsilon}}^N F$ are $\CD^\ast(0,N+1)$-spaces (respectively $\RCD(0,N+1)$-spaces if $F$ is infinitesimally Hilbertian). 

We now consider the open subset $[0,2/A)\times F/(0,p)\sim(0,q)$. This is the open ball of radius $2/A$ centered at the vertex $\ast$ of the cones, with respect to the metrics $\dis_A$ or $\dis_\epsilon$. Moreover as we showed, it is totally geodesic, and by construction we have that $\m_\epsilon(\partial [0,2/A)\times F/(0,p)\sim(0,q))=0= \m_A(\partial [0,2/A)\times F/(0,p)\sim(0,q))$. Thus we conclude that $[0,2/A]\times_{f_{A,\epsilon}}F$ are $\CD^\ast(0,N+1)$-spaces (respectively, $\RCD(0,N+1)$-spaces when $F$ is infinitesimally Hilbertian).

Since the $N$-warped product $[0,2/A]\times_{f_{A}}^N F$ is the pointed measured Gromov-Hausdorff limit of the spaces $[0,2/A]\times_{f_{A,\epsilon}}^N F$ as $\epsilon\to 0$, and the $\CD^\ast$, $\RCD$-conditions are stable under this type of convergence, we obtain the desired result.
\end{proof}

We observe that for spaces with positive Ricci curvature and dimension at least $2$ we can also consider the the function $f_2$ given by $A=2$, due to the following result:

\vspace*{8pt}
\begin{duplicatePROP}[\ref{prop.productsrescallings} (3)]
Let $(X,\dis,\m)$ be an $\RCD(K,N)$. Then for $\alpha,\beta>0$, the space $(X,\alpha\dis,\beta\m)$ is an $\RCD(\alpha^{-2}K,N)$-space.
\end{duplicatePROP}
\vspace*{8pt}

We consider  an $n$-dimensional closed Riemannian manifold $F$  such that for  $0<\kappa_F$ we have $\kappa_F\leq \mathrm{Ric}(F)$, where $\mathrm{Ric}(F)= \min\{\mathrm{Ric}(x)\mid x\in TM\: \|x\|=1\}$. Thus we have that $(F,\dis_F,\m_F)$ is a $\CD^\ast(\kappa_F,n)$-space. By taking $\lambda = \sqrt{\kappa_F/4(n-1)}$ we have that $(F,\lambda \dis_F,\m_F)$ is an $\CD^\ast((n-1)4,n)$-space. Given that $F$ is a Riemannian manifold, we obtain that $(F,\lambda \dis_F,\m_F)$ is an $\RCD(0,n)$-space, and also $(F,\m_F)$ is a weighted Finsler manifold. Thus we have that $|f_{2,\epsilon}'(0)|\leq  2$, and by \th\ref{T: cone over f A sigma are CD} that $[0,\infty)\times_{f_{2,\epsilon}}(F,\lambda \dis_F,\m_F)$ is an $\RCD(0,n+1)$-space. Since the balls of radius $1$ are geodesic, then by taking the limit as $\epsilon\to 0$ we obtain that that $[0,1]\times_{f_{2}} (F,\lambda \dis_F,\m_F)$ is an $\RCD(0,n+1)$-space. \\

\subsection{Gluing of \texorpdfstring{$\RCD$}{RCD}-cone-bundles.}

We  now consider gluings along the boundary of (certain quotients of) the $N$-warped products we have considered so far using the functions $f_A\colon [0,2/A]\to [0,\infty)$. Consider the cone-like warped product $C=[0,2/A]\times F/(0,p)\sim (0,q)$ with the distance $\dis_C = \dis_A$ and measure $\m_C=f^{n-1}_Ad\m_F\otimes d\mathcal{L}$, where $(F,\dis_F,\m_F)$ is an (infinitesimally Hilbertian) weighted Finsler manifold satisfying the $\CD^\ast((N-1)K_F,N)$-condition for some $K_F>0$. We consider $G$ equipped with a bi-invariant metric $\dis_G$ and the Haar measure $\m_G$. We take  $G\times C$ equipped with the product distance $\dis_G\oplus \dis_C$ and measure $\m_G\otimes \m_C$. Thus, taking $D:=\dim(G)$,  by \th\ref{T: Special type of cone} the  space $G\times C$ is an $\RCD(0,N+1+D)$-space.

We  further assume that a Lie subgroup $K\subset G$ acts by measure preserving isometries on $F$, the fiber of the cone $C$. Thus $K$ also acts by measure preserving isometries on the metric product $G\times C$ by a diagonal action. That is, each $h\in K$ acts by 
\[
h (g,[t,x]) := (gh^{-1}, [t,h x]),
\]
where, $hx$ represents the $K$-action on $F$. Thus the quotient space $(G\times C)/\Delta$ equipped with the quotient distance $\dis^\ast$ and quotient measure $\m^\ast$ is an $\RCD(0,N+1+D)$-space by \cite{GalazGarciaKellMondinoSosa2018}. 

We collect this conclusion in the following lemma:

\begin{lemma}
Consider $G$ an $D$-dimensional compact Lie group, $K\subset G$ a Lie subgroup, and $(F,\dis_F,\m_F)$  a Riemannian manifold, which satisfies the $\RCD((N-1)K_F,N)$-condition on which $K$ acts transitively by measure preserving isometries for some $K_F>0$. Assume that $\diam(F)\leq \pi/\sqrt{K_F}$, if $N=1$. Consider $A>0$ such that $K_F>A^2$, and the function $f_A(x):=\min\{1,Ax\}$. Consider the warped product $C:=[0,2/A]\times_{f_A}^N F$, and denote by $\Delta$ the diagonal action of $K$ on $G\times C$. Then the space $G\times C/\Delta$ is an $\RCD(0,N+1+D)$-space.
\end{lemma}

Moreover when we consider a homogeneous space $X=G/H$, equipped with $G$-invariant distance $\dis$ and measure $\m$ making it an $\RCD$-space, the measure $\m$ has to be  (a multiple of) the Hausdorff measure as we now show (this has also been independently proved by \cite[Theorem 2]{HondaNepchey2024}).

\begin{prop}
Let $G$ be a compact Lie group acting transitively by measure preserving isometries on an $\RCD(K,N)-$space $(X,\dis,\m)$. Then $\m$ equals a multiple of the $n$-dimensional Hausdorff measure of $X$, where $n$ is the topological dimension of $X$.
\end{prop}

\begin{proof}
By  \cite[Proposition 5.14]{Santos2020} we have that $X$ must be isometric to a Riemannian manifold. So, in particular we have that all the tangent spaces must be isometric to $\mathbb{R}^n$ for some natural number $n$.
Then, by  \cite[Theorem 1.12]{BrueSemola2020} we have then that $\mathfrak{m}= \rho\mathcal{H}^n$, where $\mathcal{H}^n$ denotes the $n$-dimensional Hausdorff measure. We claim now that for all $g\in G$. $\rho\circ g=\rho$.
Let $A\in \mathcal{B}(X)$, $g \in G$, and recall that $g_{\#}\mathcal{H}^n=\mathcal{H}^n$. As the action is by measure preserving isometries $\mathfrak{m}(A)=g_{\#}\mathfrak{m}(A)$. This implies
\[
\int_A \rho d\mathcal{H}^n=\int_{g^{-1}A}\rho d\mathcal{H}^n=\int_A\rho\circ g d\mathcal{H}^n.
\]
So the density $\rho$ is $G$-invariant. By the Lebesgue differentiation theorem we have that for $\mathcal{H}^n$-a.e. $x\in X$ 
\[
\rho(x) = \lim_{r\rightarrow 0}\frac{\mathfrak{m}(B_r(x))}{\mathcal{H}^n(B_r(x))}.
\]
But notice that the limit on the RHS is the same if we put $gx$ instead of $x$. So then $\rho(x)=\rho(gx)$ $\mathcal{H}^n$-a.e.
This implies then that $\rho$ must be constant.
\end{proof}

\begin{theorem}
Consider $G_\pm$ compact Lie groups, $(P_\pm,\dis_{\pm}^P,\m_\pm^P)$ be $N_\pm^P$-dimensional Riemannian manifolds with $\Ric(P_\pm)\geq K^P_\pm$, and $(F_\pm,\dis_{\pm}^{F},\m_\pm^F)$ be weighted Riemannian manifolds which are $\RCD(K_\pm^F,N_\pm^F)$ for $N_\pm\geq 1$, with $K_\pm^F>0$ if $N_\pm^F>1$, and have bounded diameter in the case of $N_\pm=1$. Assume that $G_\pm$ acts effectively  on $P_\pm$ and $F_\pm$ by measure preserving isometries. Furthermore assume that the balanced product $P_- \times_{G_-} F_-$ is homeomorphic to $P_+ \times_{G_+} F_+$. Then the glued space 
\[
X = (P_-\times_{G_-} \mathrm{Cone}(F_-))\cup_{\partial} (P_-\times_{G_+} \mathrm{Cone}(F_+))
\]
admits an $\RCD(K,N)$-structure. Here $\mathrm{Cone}(F_\pm)$ denotes the topological cone over $F_\pm$.
\end{theorem}

\begin{proof}
In the case that $N_\pm^F>1$, we choose $K_\pm> 0$ so that $(F_\pm,\dis^F_\pm,\m_\pm^F)$ are $\RCD((N_\pm^F-1)K_\pm,N^F_\pm)$, for example by taking $K_\pm\leq K_\pm^F/(N^F_\pm-1)$. In the case of $N_\pm^F=1$, we always have that $(F_\pm,\dis^F_\pm,\m_\pm^F)$ are $\RCD((N_\pm^F-1)K_\pm,N_\pm^F)=\RCD(0,N_\pm^F)$ for $K_\pm>0$.  Observe that in the case when either $N^F_+=1$ or $N^F_-=1$, then by taking $0<\lambda_\pm\leq (\pi/\sqrt{K_\pm})/\diam(F_\pm)$, we have that by rescaling $F_\pm$ by $\lambda_\pm$, the rescaled space $(F_\pm,\lambda_\pm\dis^F_\pm,\m^F_\pm)$ is an $\RCD((N^F_\pm-1)K_\pm/\lambda_\pm,N^F_\pm)$-space which satisfies $\diam(F_\pm, \lambda_\pm \dis^F_\pm)\leq \pi/\sqrt{K_\pm}$. Alternatively, we can choose $K_\pm>0$ such that $\diam(F_\pm)\leq \pi/\sqrt{K_\pm}$. Thus without loss of generality, in case either $N^F_+=1$ or $N^F_-=1$, we may assume that $\diam(F_\pm)\leq \pi/\sqrt{K_\pm}$.

We now consider $A_\pm>0$ such that $K_\pm>A_\pm^2$. We now consider the following two $\RCD(\bar{K}_\pm,N^F_{\pm}+1+N^P_\pm)$-spaces:
\[
\Big((P_\pm\times [0,2/A_{\pm}]\times_{f_\pm} F_{\pm})/G_{\pm},(\dis^P_\pm\otimes \dis_{A_\pm})^\ast, (\m^P_\pm\otimes d\m_{A_{\pm}})_\ast\Big),
\]
where $\bar{K}_\pm =\min\{K_\pm^P,0\}$, and $(\dis_{A_\pm},\m_{A_\pm})$ is the distance and measure of the warped product $[0,2/A_\pm]\times_{f_{A_\pm}}^{N^F_\pm} F_\pm$. The warping function $f_\pm$ is given by $f_\pm(t)=\min\lbrace 1,A_{\pm} t\rbrace$. We denote by $\dis_\pm$ the metric $(\dis^P_\pm\otimes \dis_{A_\pm})^\ast$, and by $\m_\pm$ the measure $(\m^P_\pm\otimes d\m_{A_{\pm}})_\ast$.

We observe that the boundaries of the cone bundles $P_\pm\times_{G_\pm}([0,2/A_\pm]\times_{f_\pm} F_\pm\to P_\pm/G_\pm$ are homeomorphic  to $ P_+\times_{G_+}(F_+) = (P_+\times F_+) /G_+ =\{2/A_+\}\times ((P_+\times F_+)/G_+)$ and $P_-\times_{G_-}(F_-) = (P_-\times F_-) /G_- =\{2/A_-\}\times ((P_-\times F_-)/G_-)$ respectively, and by hypothesis both are homeomorphic to a space $B$.

We also recall that by construction $B$ with the induced metric is geodesic in $(P_\pm\times [0,2/A_\pm]\times_{f_\pm} F_\pm)/G_\pm$ with respect to the the two $\RCD(\bar{K}_\pm,N^F_{\pm}+1+N^P_\pm)$-structures $(\dis_\pm,\m_\pm)$. Moreover, over the open collars $(P_\pm \times ((1/A_\pm,2/A_\pm]\times F_\pm))/G_\pm$ the metric structures $(\dis_\pm,\m_\pm)$ are the product structures $(\dis_\R\oplus \dis^P_\pm\oplus \dis^F_\pm,f^{N^F_\pm}_\pm d\mathcal{L}\otimes d\m^P_\pm\otimes d\m^F_\pm)$. From this we  observe that induced metric structures on $B$ are of the form $((\dis^P_\pm\otimes \dis^F_\pm)^\ast,(\m^P\pm\otimes \m^F_\pm)_\ast)$ and are given by a Riemannian metric $g_\pm$ on $B$: since $P_\pm\times F_\pm$ with $g_{P_\pm}\oplus g_{F_\pm}$ is a Riemannian manifold inducing $(\dis^P_\pm\oplus \dis^F_\pm,\m^P_\pm\oplus \m^F_\pm)$, and the action of $G_\pm$ over $P_\pm\times F_\pm$ is free,  we obtain quotient Riemannian metrics $g_\pm$ on $B$ inducing $((\dis^P_\pm\oplus \dis^F_\pm)^\ast,(\m^P_\pm\otimes\m^F_\pm)_\ast)$.

We consider $[-1,1]\times B$, and consider $g(t)$ a path in the set of Riemannian metrics on $B$ with $g(t)=g_-$ over $[-1,-1+\varepsilon]$ and $g(t)=g_+$ over $[1-\varepsilon,1]$. We set $C_0 := [-1,1]\times B$ and equip it with the metric $\dis_0$ and measure $\m_0$ induced by the Riemannian metric $h=dt^2\oplus g(t)$. We observe that since $([-1,1]\times B, h)$ is a compact Riemannian manifold, then there exists $K_0\in \R$, $|K_0|\leq\infty$ such that $\Ric(h)\geq K_0$. This implies that $(C_0,\dis_0,\m_0)$ is an $\RCD(K_0,\dim(B)+1)$-space.

We identify $(P_-\times [0,2/A_-]\times_{f_-} F_-)/G_-$ with $C_-:=(P_-\times ([-1-2/A_-,-1]\times F/(-1-2/A_-,p)\sim (-1-2/A_-,q))/G_-$ and push the $\RCD(\bar{K}_-,N^F_-+1+N^P_-)$-structure $(\dis_-,\m_-)$ of $(P_-\times[0,2/A_-]\times_{f_{-}} F_-)/G_-$ to get  an $\RCD(\bar{K}_-,N^F_-+1+N^P_-)$-structure on $C_-$, which we also denote by $(\dis_-,\m_-)$ abusing notation. We also identify $(P_+\times [0,2/A_+]\times_{f_{+}} F_+)/G_+$ with $C_+:=(P_+\times ([1,1+2/A_+]\times F/(1+2/A_+,p)\sim(1+2/A_+,q))/G_+$ and push the $\RCD(\bar{K}_+,N^F_+ +1+N^P_+)$-structure $(\dis_+,\m_+)$ of $(P_+ \times[0,2/A_+]\times_{f_{+}} F_+)/G_+$ to get an $\RCD(\bar{K}_+,n_+1 +D)$-structure on $C_+$, which we also denote by $(\dis_+,\m_+)$.

We consider the glued space
\begin{linenomath}
\begin{align*}
X_1:= C_-\bigcup_{\mathrm{Id}_B}C_0.
\end{align*}
\end{linenomath}
On $X_1$ we define a metric $\dis_1$ given by
\[
\dis_1(x,y)=\begin{cases}
\dis_{-}(x,y), & x,y\in C_-\setminus \partial C_- = C_-\setminus B,\\
\dis_{0}(x,y), &  x,y\in C_0\setminus (\{-1\}\times B),\\
\min\{\dis_-(x,a)+\dis_{0}((-1,a),y)\mid a\in B\}, & x\in C_-\mbox{ and }y\in C_0.\\
\end{cases}
\]
We also set the measure $\mathfrak{n}_1:=(i_-)_\ast(\m_-)+(i_0)_\ast(\m_0)$, where $i_-\colon C_-\to X$ and  $i_0\colon C_0\to X$ are the respective inclusions. Also observe that for $K_1=\min\{K_0,\bar{K}_+,\bar{K}_-\}$ and $N_1=\max\{N^F_-+1 +N^P_\pm,\dim(B)+1\}$, we have that $C_0$, $C_-$ are $\RCD(K_1,N_1)$-spaces. We consider  $\delta_1<\min\{\varepsilon,1/A_-\}$ as well as the open subsets $U_- =C_-\cup([-1,-1+\delta)\times B\subset X_1$, and $U_0 = (-1-\delta,1]\times B$, and observe that they are an open cover of $X_1$ due to the fact that the metrics $\dis_-$ and $\dis_0$ are the product metrics over the preimage of $[-1-(2/A_-)-\delta, -1+\delta]\subset X_1^\ast$ under the quotient map $\pi\colon X_1\to X_1^\ast$. Moreover, since the metric $\dis_1$ and measure $\m_1$ are the product over $[-1-\delta,-1+\delta]$, we have that the spaces  $(\bar{U}_-,\dis_1,\m_1\llcorner \bar{U_-})$, $(\bar{U}_0,\dis_1,\m_1\llcorner \bar{U_0})$  are $\RCD(K_1,N_1)$-spaces by construction. Thus by \th\ref{T:local-to-global}, we conclude that $X_1$ is an $\RCD(K_1,N_1)$-space.

In analogous fashion we can prove that the space
\[
X = X_1\bigcup_{\mathrm{Id}_B} C_+
\]
is an $\RCD(K,N)$-space where $K=K_1\leq 0$ and $N=\max\{N_1,N^P_+ +1+ N^P_+\}$.
\end{proof}

We now prove \th\ref{thm.rcd space from group diagram}. 

\begin{proof}[Proof of \th\ref{thm.rcd space from group diagram}]

We set $F_\pm$ to be the homogeneous Riemannian manifolds $G_\pm/H$ with the $G$-invariant Riemannian metrics of positive Ricci curvature  when \linebreak$n_\pm:=\dime(F_\pm)>1$. Then they are $\RCD((n_\pm-1)K_\pm,n_\pm)$-spaces, for $n_\pm $ and some $K_\pm> 0$ fixed when $n_\pm>1$. In the case when either $n_+=1$ or $n_-=1$, since $G_\pm/H$ have finite diameter $\leq D$, we can choose $K_\pm>0$  such that $\diam(F_\pm)\leq D\leq \pi/\sqrt{K_\pm}$.

We now consider $A_\pm>0$ such that $K_\pm>A_\pm^2$. We now consider the following two $\RCD(0,n_{\pm}+1+D)$-spaces:
\[
\Big((G\times [0,2/A_{\pm}]\times_{f_\pm}^{n_\pm} F_{\pm})/G_{\pm},(\dis_G\otimes \dis_{A_\pm})^\ast, (\m_G\otimes d\m_{A_{\pm}})_\ast\Big),
\]
where $D=\dim G$, and $(\dis_{A_\pm},\m_{A_\pm})$ is the distance and measure of the warped product $[0,2/A_\pm]\times_{f_{A_\pm}}^{n_\pm} F_\pm$. The warping function $f_\pm$ is given by $f_\pm(t)=\min\lbrace 1,A_{\pm} t\rbrace$. We denote by $\dis_\pm$ the metric $(\dis_G\otimes \dis_{A_\pm})^\ast$, and by $\m_\pm$ the measure $(\m_G\otimes d\m_{A_{\pm}})_\ast$.

We showed in \th\ref{prop.near bdry metric is product in warped product} that over $(1/A_\pm,2/A_\pm]\times F_\pm\subset [0,2/A_\pm]\times_{f_\pm} F_\pm$ we have $\dis_{A_\pm} = \dis_{\R}\oplus \dis_{F_\pm}$. Since the action of $G_\pm$ on the $[0,2/A\pm]$ factor is free, we see that the image  of $G\times [1/A_\pm,2/A_\pm]\times F_\pm$ under the quotient map $G\times ([0,2/A_\pm]\times_{f_{\pm}} F_\pm)\to (G\times ([0,2/A_\pm]\times_{f_\pm} F_\pm))/\Delta$ is isometric to 
\[
\Big([1/A_\pm,2/A_\pm]\times \big((G\times F_\pm)/\Delta\big),\dis_\R\oplus (\dis_G\otimes \dis_{F_\pm})^\ast, f_\pm^{n_\pm}d\mathcal{L}\otimes (\m_G\otimes d\m_{F_\pm})_\ast\Big).
\]

We observe that the boundaries of the cone bundles $G\times_{G_\pm}([0,2/A_\pm]\times_{f_\pm} F_\pm\to G/G_\pm$ are $G/H \cong G\times_{G_+}(G_+/H) = (G\times F_+) /G_+ =\{2/A_+\}\times ((G\times F_+)/G_+)$ and $G/H \cong G\times_{G_-}(G_-/H) = (G\times F_-) /G_- =\{2/A_-\}\times ((G\times F_-)/G_-)$. Thus they are homeomorphic to $G/H$. From now on we set $B=G/H$. 

We also recall that by construction the space $B$ with the induced metric is geodesic in $(G\times [0,2/A_\pm]\times_{f_\pm}^{n_\pm} F_\pm)/G_\pm$ with respect to the the two $\RCD(0,n_{\pm}+1+D)$-structures $(\dis_\pm,\m_\pm)$. Moreover, over the open collars $(G ((1/A_\pm,2/A_\pm]\times F_\pm))/G_\pm$ the metric structures $(\dis_\pm,\m_\pm)$ are the product structures $(\dis_\R\oplus \dis_G\oplus \dis_{F_\pm},f^{n_\pm}_\pm d\mathcal{L}\otimes d\m_G\otimes d\m_{F_\pm})$. From this we  observe that induced metric structures on $B$ are of the form $((\dis_G\otimes \dis_{F_\pm})^\ast,(\m_G\otimes \m_{F_\pm})_\ast)$ and are given by a Riemannian metric $g_\pm$ on $B$: since $G\times F_\pm$ with $g_G\oplus g_{F_\pm}$ is a Riemannian manifold inducing $(\dis_G\oplus \dis_{F_\pm},\m_G\oplus \m_{F_\pm})$, and the action of $G_\pm$ over $G\times F_\pm$ is free,  we obtain quotient Riemannian metrics $g_\pm$ on $B$ inducing $((\dis_G\oplus \dis_{F_\pm})^\ast,(\m_G\otimes\m_F)_\ast)$.

We consider $[-1,1]\times B$, and consider $g(t)$ a path in the set of Riemannian metrics on $B$ with $g(t)=g_-$ over $[-1,-1+\varepsilon]$ and $g(t)=g_+$ over $[1-\varepsilon,1]$. We set $C_0 := [-1,1]\times B$ and equip it with the metric $\dis_0$ and measure $\m_0$ induced by the Riemannian metric $h=dt^2\oplus g(t)$. We observe that since $([-1,1]\times B, h)$ is a compact Riemannian manifold, then there exists $K_0\in \R$, $|K_0|\leq\infty$ such that $\Ric(h)\geq K_0$. This implies that $(C_0,\dis_0,\m_0)$ is an $\RCD(K_0,\dim(B)+1)$-space.

We identify $[0,2/A_-]\times_{f_-}^{n_-} F_-$ with $C_-:=(G\times([-1-2/A_-,-1]\times F/(-1-2/A_-,p)\sim (-1-2/A_-,q)))/G_-$ and push the $\RCD(0,n_-+1+D)$-structure $(\dis_-,\m_-)$ of $(G\times [0,2/A_-]\times_{f_{-}}^{n_-} F_-)/G_-$ to get  an $\RCD(0,n_-+1+D)$-structure on $C_-$, which we also denote by $(\dis_-,\m_-)$ abusing notation. We also identify $G\times([0,2/A_+]\times_{f_{+}}^{n_+} F_+)/G_+$ with $C_+:=(G\times ([1,1+2/A_+]\times F/(1+2/A_+,p)\sim(1+2/A_+,q)))/G_+$ and push the $\RCD(0,n_+1 +D)$-structure $(\dis_+,\m_+)$ of $(G\times [0,2/A_+]\times_{f_{+}}^{n_+} F_+)/G_+$ to get an $\RCD(0,n_+1 +D)$-structure on $C_+$, which we also denote by $(\dis_+,\m_+)$.

We consider the glued space
\begin{linenomath}
\begin{align*}
X_1:= C_-\bigcup_{\mathrm{Id}_B}C_0.
\end{align*}
\end{linenomath}
We observe that it admits an action by $G$, and moreover by construction we have $X_1^\ast=X/G \cong [-1-2/A_-,1]$. On $X_1$ we define a metric $\dis_1$ given by
\[
\dis_1(x,y)=\begin{cases}
\dis_{-}(x,y), & x,y\in C_-\setminus \partial C_- = C_-\setminus B,\\
\dis_{0}(x,y), &  x,y\in C_0\setminus (\{-1\}\times B),\\
\min\{\dis_-(x,a)+\dis_{0}((-1,a),y)\mid a\in B\}, & x\in C_-\mbox{ and }y\in C_0.\\
\end{cases}
\]
We also set the measure $\mathfrak{n}_1:=(i_-)_\ast(\m_-)+(i_0)_\ast(\m_0)$, where $i_-\colon C_-\to X$ and  $i_0\colon C_0\to X$ are the respective inclusions. Also observe that for $K_1=\min\{K_0,0\}$ and $N_1=\max\{n_-+1 +D,\dim(B)+1\}$, we have that $C_0$, $C_-$ are $\RCD(K_1,N_1)$-spaces. We consider  $\delta_1<\min\{\varepsilon,1/A_-\}$ as well as the open subsets $U_- =C_-\cup([-1,-1+\delta)\times B\subset X_1$, and $U_0 = (-1-\delta,1]\times B$, and observe that they are an open cover of $X_1$ due to the fact that the metrics $\dis_-$ and $\dis_0$ are the product metrics over the preimage of $[-1-(2/A_-)-\delta, -1+\delta]\subset X_1^\ast$ under the quotient map $\pi\colon X_1\to X_1^\ast$. Moreover, since the metric $\dis_1$ and measure $\m_1$ are the product over $[-1-\delta,-1+\delta]$, we have that the spaces  $(\bar{U}_-,\dis_1,\m_1\llcorner \bar{U_-})$, $(\bar{U}_0,\dis_1,\m_1\llcorner \bar{U_0})$  are $\RCD(K_1,N_1)$-spaces by construction. Thus by \th\ref{T:local-to-global}, we conclude that $X_1$ is an $\RCD(K_1,N_1)$-space. Observe that since the action of $G$ is transitive on $B$, we have that 
\begin{linenomath}
\begin{align*}
\min\{\dis_-(x,a)&+\dis_0((-1,a),y)\mid a\in B\}\\ 
&= \min\{\dis_-(x,g^{-1}a)+\dis_0((-1,g^{-1}a),y)\mid g^{-1}a\in B\}\\ 
&= \min\{\dis_-(x,\tilde{a}])+\dis_0((-1,\tilde{a}),y)\mid \tilde{w}\in B\}.  
\end{align*}
\end{linenomath}

Thus $G$-acts by isometries on $X_1$. Consider $Z\subset X_1$ measurable. Take $Z_-=\mathrm{cl}(\mathrm{cl}(Z)\cap C_-)$ and set $Z_0 = \mathrm{cl}(Z)\setminus Z_-$. Then $Z_-\cup Z_0 = \cl(Z)$, and they are disjoint. Since $G$ preserves the measures $(i_-)_\ast(\m_-)$ and $(i_0)_\ast(\m_0)$ we get that $G$ acts by measure preserving isometries.

In analogous fashion we can prove that the spaces
\[
X = X_1\bigcup_{\mathrm{Id}_B} C_+
\]
is an $\RCD(K,N)$-space where $K=K_1\leq 0$ and $N=\max\{N_1,n_+ +1+ d\}$, and $G$ acts by measure preserving isometries on $X$. Moreover, we have that $X/G$ is isometric to $([-1-2/A_-,1+2/A_+], dt^2, f\mathcal{L}^1$), where 
\[
f(t)=\begin{cases}
    (f_{A_-}(t))^{n_-}, & t\in [-1-2/A_-,-1],\\
    1 & t\in [-1,1],\\
   ( f_{A_+}(t))^{n_+}, & t\in [1,1+2/A_+].
\end{cases}
\]
Thus $X$ is an $\RCD(K,N)$-space of cohomogeneity one with group diagram $(G,H,G_-,G_+)$.
\end{proof}

\section{Classification of non-collapsed cohomogeneity one \texorpdfstring{$\RCD$}{RCD}-spaces with \texorpdfstring{$\dimess\leq 4$}{<4}}
\label{S:Low-DIM-Classification}

In this section we obtain the topological classification of non-collapsed $\RCD$-spaces admitting effective actions by a compact Lie group $G$ via measure-preserving isometries of with cohomogeneity one, and which have essential dimension less than or equal to $4$. The classification for cohomogeneity one Alexandrov spaces up to dimension $4$ was obtained by Galaz-García and Searle \cite{Galaz-GarciaSearle2011}, which in turn extends the extensive work done in the smooth manifold case by Mostert \cite{Mostert1957a}, \cite{Mostert1957b}, Neumann \cite{Neumann1967}, Parker \cite{Parker1986}, and Hoelscher, \cite{Hoelscher2010}, and in the topological category by Galaz-García and Zarei \cite{GalazGarciaZarei2018} (see also the thesis of Muzzy \cite{Muzzy2017} for a very nice account of the classification of low dimensional cohomogeneity one manifolds in the simply-connected case). 

In this section we obtain  that  non-collapsed $\RCD$-spaces admitting effective actions by a compact Lie group $G$ via measure-preserving isometries of  with cohomogeneity one with $\dimess \leq 4$ are equivariantly homeomorphic to Alexandrov spaces. Thus the topological classification follows from the one in \cite{Galaz-GarciaSearle2011}.

\subsection{Essential dimensions \texorpdfstring{$1$}{1} and \texorpdfstring{$2$}{2}.} 

Without any symmetry assumptions, the work of Kitabeppu-Lakzian \cite{KitabeppuLakzian2016} and of Lytchak-Stadler \cite{LytchakStadler2023} shows that non-collapsed $\RCD(K,1)$- and $\RCD(K,2)$-spaces are homeomorphic to Alexandrov spaces. 

\subsection{Essential dimensions \texorpdfstring{$3$}{3} and \texorpdfstring{$4$}{4}.} We now analyze the case of non-collapsed spaces with essential dimensions $3$ and $4$. 
 
Let $x_0\in X$ be contained in a principal orbit, and take $k = \dime(G(x_0))$. Recall that the orbit $G(x_0)$ is homeomorphic to the homogeneous space $G/G_{x_0}$. By \cite[Chapter III, Section IV~54]{Cartan1952}, there exists a $G$-invariant Riemannian metric $g_0$ on $G/G_{x_0}$ such that $G\subset \Iso(G/G_{x_0},g_0)$. Hence, due to \cite{MyersSteenrod} and \cite[Chapter II, Section 3, Theorem 3.1]{Kobayashi}, we have the following dimension estimate 
\begin{equation} 
\label{EQ:dimension-bound-G-by-dimension-of-orbit}
k\leq \dime(G) \leq \dime(\Iso(G/G_{x_0},g_0))\leq \frac{k(k+1)}{2}.
\end{equation}

We use this dimension estimate to limit the possible groups that can act by cohomogeneity one on a non-collapsed $\RCD$-space of low essential-dimension. 

\begin{mtheorem}
Let $G$ be a compact Lie group acting effectively by measure preserving isometries and cohomogeneity one on a closed non-collapsed $\RCD(K,3)$-space $(X,\dis,\m)$.  Then $G$ is one of $\SO(3)$ or $T^2$ and $X$ is equivariantly homeomorphic to an Alexandrov $3$-space. 
\end{mtheorem}

\begin{proof} 
We begin by pointing out that, if $X^* = \Sp^1$, then $X$ is homeomorphic to a $3$-manifold by item \eqref{MC: homeomorphism rigidity Circle} of \th\ref{MC: Homeomorphism rigidity}, and therefore the statement holds true. Thus, in the following we assume that $X^*=[-1,1]$.

Consider $x_0\in X$ fixed, contained in a principal orbit. As the action is of cohomogeneity one, we have that $\dime(G(x_0)) =2$. Then the dimension estimate \eqref{EQ:dimension-bound-G-by-dimension-of-orbit}, implies that $2\leq \dime(G)\leq 3$. It now follows from the classification of low dimensional compact Lie groups that $G$ is either $T^2$, $T^3$ or $\SO(3)$. 

Assume first that $G= T^3$, and  observe that all principal isotropy groups are the same for this action, since they are conjugated to each other (see \cite[Theorem 4.7 and proof of Corollary 4.9]{GalazGarciaKellMondinoSosa2018}) and $G$ is abelian. Moreover, as $G$ acts by cohomogeneity one on $X$ we have a group diagram of the form:
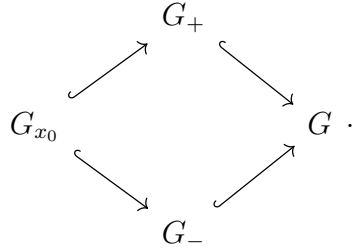
\[
\begin{tikzcd}
         & G_+ \arrow[dr,hookrightarrow] & \\
G_{x_0} \arrow[ur,hookrightarrow]\arrow[dr,hookrightarrow] & & G\\
 & G_{-}\arrow[ur,hookrightarrow] & 
\end{tikzcd}.
\]
Thus we have that $G_{x_0}\subset G_+$ and $G_{x_0}\subset G_{-}$. This implies that the ineffective kernel of the action $\cap_{y\in X} G_y$ contains $G_{x_0}$. But since the action is effective, we conclude that $G_{x_0} = \{e\}$. This implies that $G(x_0)$ is homeomorphic to $T^3$ which is a contradiction. Thus $G$ is either $T^2$ or $\SO(3)$, proving the first statement.\\

To proceed we observe on one hand that the possible isotropy groups when $G = T^2$ are $\{e\}$, $\Z_k$, $\Sp^1$, $\Z_k\times \Sp^1$, and $T^2$. When $G= \SO(3)$, since the orbits have dimension $2$, we conclude that $\SO(2)$ is a subgroup of any isotropy group, and thus from any isotropy group. Thus the possible isotropy groups are $\SO(2)$ $\mathrm{O}(2)$, $\SO(3)$.

We recall from \cite[Corollary 10]{GalazGarciaKellMondinoSosa2018} that $\pi^{-1} (-1,1)$ is homeomorphic to $(-1,1)\times G_x$, where $x\in \pi^{-1}(-1,1)$. 

Now from \th\ref{T: Characterization non-collapsed} we have that $X$ is non-collapsed, and  we consider $x_{\pm}\in X$ such that $x_{\pm}^\ast = \pm 1$. We study the tubular neighborhoods of the orbit $G(x_\pm)$. We denote by $m_\pm = \dime(G(x_\pm))$, and observe that by our hypothesis $N=3$. Due to \th\ref{T:Slice-Theorem}, we have that a tubular neighborhood of the orbit $G(x_\pm)$ is homeomorphic to $G\times_{G_\pm} S_{\pm}$, where $S_{\pm}$ is the slice through $x_\pm$.\\

By \th\ref{MT: Geometry of the slice} and \th\ref{L: infinitesimal action cohomogeneity 1}, the slice $S_{\pm}$ is homeomorphic to $\mathrm{Cone}_0^N(F)$ where $(F,\dis_F,\m_F)\in \RCD(3-m_\pm-2,3-m_\pm-1) = \RCD(1-m_\pm,2-m_\pm)$, and $(F,\dis_F,\m_F)$ is a smooth homogeneous space $K_\pm/H_\pm$. We analyze different cases based on the dimension $m_\pm$.

When $m_\pm = 2$, then $F$ is $0$-dimensional, and $Z=\mathrm{Cone}_0^N(F)$ is an $\RCD(0,1)$-space. By \cite{KitabeppuLakzian2016}, it follows that $Z$ is homeomorphic to one of $\R$, $[0,\infty)$, $[0,1]$, or $\Sp^1$. More over by \th\ref{L: infinitesimal action cohomogeneity 1}, we have that the tangent space at $x_\pm$ is isometric to $Y\times Z = \R^2\times Z$. Since $X$ is non-collapsed we have that the tangent  space at $x_\pm$ is isometric to $\R^3$. Thus we conclude that $Z$ is isometric to $\R$, and thus the slice $S_\pm$ is homeomorphic to $\R$. Moreover, in this case the dimension of $G_\pm$ is $0$, and thus $G_\pm$ is a discrete group acting effectively on $\R$ by homeomorphisms, with $S_\pm/G_\pm = [0,\varepsilon)$. Thus $G_\pm = \Z_2$. This implies that $G= T^2$, and thus we have that a tubular neighborhood of $G(x_\pm)$ is homeomorphic to $T^2\times_{\Z_2} \R$. From the classification in \cite{Galaz-GarciaSearle2011} we see that this space is homeomorphic to an Alexandrov space.

When $m_\pm  = 1$, then $F$ is an $\RCD(0,1)$-space, homogeneous, and of dimension $1$. From the characterization of low dimensional $\RCD$-spaces \cite{KitabeppuLakzian2016} we have that $(F,\dis_F,\m_F)$ is one of $(\R,\dis_{\mathrm{E}},\mathcal{L}^1)$, $([0,\infty),\dis_{\mathrm{E}},h\mathcal{L}^1)$, $([0,1],\dis_{\mathrm{E}},h\mathcal{L}^1)$, $(\Sp^1,\dis_{\Sp^1},h\mathcal{L}^1)$, where $h$ is a $(0,1)$-convex function. We recall that the due to our non-collapse assumption, the tangent space at $x_\pm$ is isometric to $\R^3$. Moreover, due to \th\ref{L: Splitting of Euclidean space in tangent cone} and \th\ref{L: infinitesimal action cohomogeneity 1} we conclude that the slice is homeomorphic to $\R^2$, and a cone over $F$.  Thus the cases cases when $F$ is homeomorphic to $[0,\infty)$ and $[0,1]$ are excluded. Moreover, the isotropy $G_\pm$ acts smoothly by cohomogeneity $1$ (see \cite[Corollary E]{GalazGarciaZarei2018}). Thus we have that $G_\pm$ is either $\Sp^1$ when $G = T^2$, or $G_\pm = \SO(2)$ or $\mathrm{O}(2)$ when $G= \SO(3)$. Comparing to \cite{Galaz-GarciaSearle2011}, we conclude that the tubular neighborhoods are Alexandrov spaces.

In the case that $m_\pm = 0$, we have that $F$ is an $\RCD(1,2)$-space. Moreover, since it is a homogeneous metric space, the measure $m_F$ is up to a constant the Haar measure. This implies that $F$ admits a new $\RCD(1,2)$-structure with respect to the $2$-dimensional Hausdorff measure. By \cite{LytchakStadler2023} it follows that $F$ is homeomorphic to an Alexandrov surface. Thus we have that $S_\pm$ is homeomorphic to the tubular neighborhood of $G(x_\pm) = \{x_\pm\}$, and to the cone over $F$, which admits an Alexandrov space structure. 

Thus by \cite{Galaz-GarciaSearle2011}, our conclusion follows, since $X$ is equivariantly indistinguishable from a cohomogeneity one Alexandrov space. 
\end{proof}

As pointed out in \cite[Section 1.6]{Hoelscher2010}, every compact connected Lie group $G$ has a finite cover of the form $G_{ss}\times T^k$, where $G_{ss}$ is semisimple and simply connected. If $G$ acts on $(X,\dis,m)$ by measure-preserving isometries then every finite cover $\tilde{G}$ of $G$ also acts on $(X,\dis,m)$ by measure-preserving isometries in the obvious way, but the action of $\tilde{G}$ has non-trivial (but finite) ineffective kernel. Therefore, if we allow for a finite ineffective kernel we can assume $G$ to be of the form $G_{ss}\times T^k$. 

The classification of semisimple and simply connected compact Lie groups in low dimensions is well known. For the readers convenience we reproduce the part of \cite[Table 1.29]{Hoelscher2010} that we need below, namely, the list of these groups up to dimension $3$ (there are no semisimple and simply connected compact groups on dimensions $>3$ up to dimension $8$). 

\begin{table}[h]
\begin{tabular}{|l|l|}
\hline
Group                                                                   & Dimension \\ \hline
$\Sp^1 \cong \mathrm{U}(1) \cong \mathrm{SO}(2)$                        & $1$       \\ \hline
$\Sp^3 \cong \mathrm{SU}(2) \cong \mathrm{Sp}(1)\cong \mathrm{spin}(3)$ & $3$       \\ \hline
\end{tabular}
\end{table}

Let us further observe that although our main topological recognition tool,\th\ref{MT: Geometry of the slice} is formulated for effective actions, one obtains the same result in the case of almost effective actions. This follows immediately by observing that from an almost effective action of a Lie group $G$ one gets an effective action of $G/K$ where $K$ is the ineffective kernel of the $G$-action by letting $(gK)x:=gx$ for each $x\in X$. The $G/K$ action then has the same orbits as the $G$-action; consequently the slices of both actions coincide and, since $K$ is finite, the dimensions of $G$ and $G/K$ agree. 

In the following theorem we assume as is usual that the action is non-reducible (cf. for example with \cite[Definition 2.1.19]{Hoelscher2007}) and that the group is acting almost effectively so that we can use the characterization of groups up to cover in the table above. It is worth pointing out that in the list of groups appearing in the corresponding theorem for Alexandrov spaces \cite[Theorem C]{Galaz-GarciaSearle2011}, the identification $\SO(4) \cong \Sp^3\times\Sp^3$ is used. As we are primarily interested in the topological recognition of non-collapsed $\RCD(K,4)$-spaces $X$ we only assume that the action is almost effective, to deal with less possible groups acting on $X$. Note however that, once the classification has been established in this case, one can recover the full classification. 

\begin{mtheorem}
Let $G$ be a compact Lie group acting almost effectively by measure preserving isometries and cohomogeneity one on a closed non-collapsed $\RCD(K,4)$-space $(X,\dis,\m)$.  Then $G$ is, up to finite cover, one of $\Sp^3$, $\Sp^3\times \Sp^3$, $\Sp^1\times\Sp^3$, $T^3$, and $X$ is homeomorphic to an Alexandrov space. 
\end{mtheorem}

\begin{proof}
As in the three-dimensional case, the topological and equivariant type of the space is that of a cohomogeneity one manifold if $X^*=\Sp^1$, and the theorem holds by \cite{Parker1986}. Therefore we only need to analyze the case in which $X^*=[-1,1]$. 

We fix a point $x_0$ on a principal orbit and note that $\dime(G(x_0))=3$. Now the bound of the dimension \eqref{EQ:dimension-bound-G-by-dimension-of-orbit} grants that  $3\leq \dime(G)\leq 6$. As we are allowing for a finite ineffective kernel, the classification of low dimensional compact Lie groups grants that $G$ is, up to finite cover, one of $\Sp^3$, $\Sp^3\times \Sp^3$, $\Sp^1\times \Sp^3$, $T^2\times \Sp^3$, $T^3\times \Sp^3$, $T^3$, $T^4$, $T^5$ and $T^6$. We can immediately rule out the tori $T^k$ for $k\geq 4$, as in these cases the principal orbits are homeomorphic to $T^k/H$ where $H$ is finite, which is impossible as the space $X$ is $4$-dimensional. Moreover, as we have assumed the action to be non-reducible we can also rule out $T^2\times \Sp^3$ by noting that in such a case, we have a $2$-dimensional subgroup $H$ of $\Sp^3\times\{1\}$, which is a contradiction since $\Sp^3$ has no $2$-dimensional subgroups. Finally $T^3\times \Sp^3$ is ruled out because the principal isotropy being $3$-dimensional, contains a torus as a proper normal subgroup acting trivially on any point of the space. This contradicts the almost effectiveness of the action. Whence, the first claim is settled. 

We consider now $x_{\pm}\in X$ with $x_{\pm}^\ast = \pm 1$ and denote by $m_\pm = \dime(G(x_\pm))$ as in the previous theorem. Observe now that by our hypothesis $N=4$, and that by the Slice \th\ref{T:Slice-Theorem}   we have that small tubular neighborhoods of the orbits $G(x_\pm)$ are homeomorphic to $G\times_{G_\pm} S_{\pm}$, where $S_{\pm}$ is the slice through $x_\pm$. Therefore, \th\ref{MT: Geometry of the slice} implies that the slice $S_{\pm}$ is homeomorphic to $\mathrm{Cone}_0^N(F)$ where $(F,\dis_F,\m_F)\in \RCD(4-m_\pm-2,4-m_\pm-1) = \RCD(2-m_\pm,3-m_\pm)$, and $(F,\dis_F,\m_F)$ is a smooth homogeneous space $K_\pm/H_\pm$. We now analyze different cases depending on the value of $m_\pm$.

In the case that $m_\pm = 3$, then $F$ is $0$-dimensional. Hence, $Z=\mathrm{Cone}_0^N(F)$ is an $\RCD(0,1)$-space and, as in the previous theorem, due to \cite{KitabeppuLakzian2016}, is one of $\mathbb{R}$, $[0,\infty)$, $[0,1]$ or $\Sp^1$. Similarly to the proof of the previous theorem, using \th\ref{L: infinitesimal action cohomogeneity 1}, we get that the tangent  space at $x_\pm$ is homeomorphic to $Y\times Z = \R^3\times Z$. By the non-collapsed assumption on $X$ we have that the tangent  space at $x_\pm$ is homeomorphic to $\R^4$. Whence, the slice $S_\pm$ is homeomorphic to $\R$. Now again, in this case the dimension of $G_\pm$ is zero and therefore $G_\pm$ is a discrete group acting effectively on $\R$ by homeomorphisms, with $S_\pm/G_\pm = [0,\varepsilon)$. This is only possible if $G_{\pm} = \mathbb{Z}_2$, in turn forcing $G$ to be either $\Sp^3$ or $T^3$. Thus, the small tubular neighborhoods of $G(x_{\pm})$ are equivariantly homeomorphic to either $\Sp^3 \times_{\mathbb{Z}_2}\mathbb{R}$ or $T^3 \times_{\mathbb{Z}_2}\mathbb{R}$, and it follows from the classification \cite{Galaz-GarciaSearle2011} that $X$ is then homeomorphic to an Alexandrov space. 

Now we turn to the case in which $m_{\pm}\leq 2$. Under this assumption $2-m_{\pm}\geq 0$ and by \cite[Theorem 1]{SturmVonRenesse2005}, $F$ has a Riemannian metric of positive Ricci curvature if $m_{\pm}\leq 1$ and of non-negative Ricci curvature if $m_{\pm}=2$. Hence, by using \cite[Main Theorem]{Hamilton1982}, $F$  admits a metric of positive sectional curvature if $m_{\pm}\leq 1$ and of nonnegative sectional curvature if $m_{\pm}=2$. To conclude we only need to observe then that the possible choices of $F$ already appear in the classification \cite{Galaz-GarciaSearle2011}, and therefore $X$ is homeomorphic to an Alexandrov space. 
\end{proof}

\section{New Examples of \texorpdfstring{$\RCD$}{RCD}-spaces of cohomogeneity one}

In this section we use the characterization of homogeneous spaces with positive Ricci curvature in \cite{Berestovskii1995} to present examples of non-collapsed $\RCD$-spaces with an action of cohomogeneity one by measure preserving isometries. We also give an explicit example of a non-collapsed $\RCD$ spaces of cohomogeneity one with essential dimension equal to $5$ that do not admit a metric making them an Alexandrov space, showing that the results in Section~\ref{S:Low-DIM-Classification} are optimal.

\begin{theorem}\th\label{T: Suspensions are RCD spaces of coho 1}
Let $F =G/H$ be a  homogeneous space of dimension $n\geq 2$ with finite fundamental group. Then there exists a $G$-invariant Riemannian metric $g$ with $\Ric(g)\geq (n-1)>0$ and $\diam(g)\leq \pi$. Moreover, for any $K\geq  0$, the $(K,n)$-cone $\mathrm{Con}^n_K(F)$ is a non-collapsed $\RCD(Kn,n+1)$ with an action by measure preserving isometries of $G$ making it a cohomogeneity one $\RCD$-space. In the case of $K>0$, the space is homeomorphic to the suspension of $F$ and the orbit space is isometric to $[0,\pi/\sqrt{K}]$ with weighted measure $\sin_K(t)^n d\mathcal{L}^1$ and we have the group diagram $(G,H,G,G)$. In the case when $K=0$, then the space is homeomorphic to the cone over $F$ with the vertex the only fixed point,  and the orbit space is isometric to $[0,\infty)$ with weighted measure $t^n d\mathcal{L}^1$.
\end{theorem}

\begin{proof}

By \cite[Theorem 1]{Berestovskii1995} $F$ admits a $G$-invariant Riemannian metric $\tilde{g}$ with $\Ric(\tilde{g})\geq a>0$. Consider $\lambda>0$ such that $a/(n-1)\geq \lambda^2$ and $\pi/\diam(\tilde{g})\geq \lambda$. Then with respect to the Riemannian metric $g=\lambda g$, we have that $(F,\dis_g,\mathrm{vol}(g)) = (F,\lambda\dis_{\tilde{g}},\lambda^n\mathrm{vol}(\tilde{g}))$ is an $\RCD(a\lambda^{-2},n)$-space. Since by construction $a\lambda^{-2}\geq (n-1)$, then $(F,\dis_g,\mathrm{vol}(g))$ is an $\RCD(n-1,n)$-space. By \cite[Theorem~2~(i)]{Sturm2006} it follows that $\Ric(g)\geq (n-1)$. Moreover, by construction we have that $\diam(g) = \lambda\diam(\tilde{g})\leq \pi$. Then by \cite[Theorem 1.1]{Ketterer2013} we obtain that for $K\geq 0$ the $(K,n)$-cone  $\mathrm{Con}^n_K(F)$ is $\RCD(Kn,n+1)$. But by construction for $K>0$ the underlying topological space of $\mathrm{Con}^n_K(F)$ is $F\times[0,\pi/\sqrt{K}]/(F\times \{0,\pi/\sqrt{K}\})$, and thus homeomorphic to  the suspension $\mathrm{Susp}(F):= F\times [0,1]/(F\times\{0,\pi/\sqrt{K}\})$. And for $K=0$ the underling space of $\mathrm{Con}_K^n(F)$ is the topological cone $F\times [0,\infty)/(F\times\{0\})$.

Recall that in both cases we can define an action of $G$ on $\mathrm{Con}_K^n(F)$ by setting
\[
g[x,t]:=[gx,t].
\]
In the case when $K=0$ we refer to this action by the cone of the action of $G$ on $F$, and for $K>0$ we refer to this action by suspension  of the action of $G$ on $F$. 

Recall from Section~\ref{SS: Warped Products} that for $K\geq 0$ the distance function $\dis_{\mathrm{Con_K^n}}$ of $\mathrm{Con}^n_K(F)$ is given by
\begin{linenomath}
\begin{align*}
    \dis_{\mathrm{Con_K^n}}&([x,s],[y,t])\\
     &:=\begin{cases}
         \cos_K^{-1}\Big(\cos_K(s)\cos_K(t) +K\sin_K(s)\sin_K(t)\cos\big(\dis_F(x,y)\wedge\pi\big)\Big) & \mbox{if }K> 0,\\
         \sqrt{s^2+t^2-st\cos\big(\dis_F(x,y)\wedge\pi\big)} & \mbox{if }K=0.
     \end{cases}
\end{align*}
where $\dis_F(x,y)\wedge\pi:=\min\{\dis_F(x,y),\pi\}$,  $\cos_K\colon [0,\pi/\sqrt{K}]\to [0,1]$ given by $\cos_K(t):=\cos(\sqrt{K}t)$, and $\sin_K\colon [0,\pi/\sqrt{K}]\to [0,\infty)$ given by
\[
\sin_K(t):=\begin{cases}
    \frac{1}{\sqrt{K}}\sin(\sqrt{K}t) &\mbox{if } K>0,\\
    t&\mbox{if } K=0.
\end{cases}
\]
From the description of $\dis_{\mathrm{Con}_K^n}$ it easy to see that $G$ acts by isometries on $\mathrm{Con}_K^n(F)$, with respect to the cone and  suspension of the action of $G$ on $F$ depending on $K$.

The measure $\m_{\mathrm{Con}_K^n}$ of $\mathrm{Con}_K^n(F)$ is defined as $\m_{\mathrm{Con}_K^n}:= (\sin_K(t))^n dt\otimes\m_F$. Thus for $g\in G$ fixed, denoting by $\alpha_g\colon F\times[0,\pi/\sqrt{K}]/(F\times \{0,\pi/\sqrt{K}\})\to F\times[0,\pi/\sqrt{K}]/(F\times \{0,\pi/\sqrt{K}\})$ the homeomorphism given by $\alpha_g[x,t]:=[gx,t]$, we have for $A=A_1\times A_2\subset F\times[0,\pi/\sqrt{K}]/(F\times \{0,\pi/\sqrt{K}\})$ that 
\begin{linenomath}
\begin{align*}
(\alpha_g)_{\#}(\m_{\mathrm{Con}_K^n})(A)&=\m_{\mathrm{Con}_K^n}(\alpha_g^{-1}(A))\\
& = \m_{\mathrm{Con}_K^n}(A_1\times \alpha_g^{-1}(A_2))\\
&= \int_{A_1\times \alpha_g^{-1}(A_2)} \sin_K(t))^N\, dt\otimes d\m_F\\
&= \int_{A_1} \sin_K(t))^N \int_{\alpha_g^{-1}(A_2)} d\,\m_F\,dt\\
&= \int_{A_1} \sin_K(t))^N \int_{A_2)} d\,(\alpha_g)_{\#}\m_F\,dt\\
&= \int_{A_1} \sin_K(t))^N \int_{A_2)} d\,\m_F\,dt\\
&= \int_{A_1\times A_2} \sin_K(t))^N \,dt\otimes d\m_F\\
&= \m_{\mathrm{Con}_K^n}(A).
\end{align*}
\end{linenomath}
Thus $G$ acts by measure preserving isometries on $\mathrm{Con}_K^n(F)$. Moreover, the vertexes of the suspension and the vertex of the cone are fixed points, and the group $H$ is the isotropy group of any point $[x,t]$ with $0<t<\pi/\sqrt{K}$ when $K>0$, or with $t>0$ when $K=0$.

Last we point out that the essential dimension of $\mathrm{Con}^n_K(F)$ is $n+1$. Thus by \cite[Theorem 1.12]{dePhilippisGigli2018} and \cite[Theorem 1.3]{BrenaGigliHondaZhu2023} we have that $\mathrm{Con}^n_K(F)$ is non-collapsed.
\end{linenomath}
\end{proof}

\vspace*{8pt}
\begin{duplicate}[\ref{MT: Grassmannians as examples}]
The suspension $\mathrm{Susp}(\Sp^2\times\Sp^2)$ of $\Sp^2\times\Sp^2$ admits a non-collapsed $\RCD(K,5)$-structure for any $K\geq 0$ such that the suspension of the $\SO(3)\times\SO(3)$-action on $\Sp^2\times\Sp^2$ is by measure preserving isometries and by cohomogeneity one, with group diagram  $(\SO(3)\times\SO(3),\SO(2)\times\SO(2),\SO(3)\times\SO(3),\SO(3)\times\SO(3))$. Moreover, $\mathrm{Susp}(\Sp^2\times\Sp^2)$ with this action of $\SO(3)\times\SO(3)$ cannot be an Alexandrov space of cohomogeneity one. I.e., it does not admit a metric making it an Alexandrov space and such that the action of $\SO(3)\times\SO(3)$ is by isometries.
\end{duplicate}
\vspace*{8pt}

\begin{proof}
We observe that $\Sp^2\times\Sp^2$ is the homogeneous space $\SO(3)\times\SO(3)/(\SO(2)\times\SO(2))\cong (\SO(3)/\SO(2))\times (\SO(3)/\SO(2))$. Thus by \th\ref{T: Suspensions are RCD spaces of coho 1}, there exists a $\SO(3)\times\SO(3)$-invariant Riemannian metric $g$ on $\Sp^2\times\Sp^2$ such that for any $K\geq 0$ the $(K,5)$-cone $\mathrm{Con}^5_K (\Sp^2\times\Sp^2)$ is a non-collapsed $\RCD(K,5)$-space with a cohomogeneity one action by measure preserving isometries of $\SO(3)\times\SO(3)$. But for $K>0$ the $(K,5)$-cone is homeomorphic to the suspension of $\Sp^2\times\Sp^2$.  This proves the first claim.

Assume that the suspension of $\Sp^2\times\Sp^2$ admits an Alexandrov metric with an action of cohomogeneity one by isometries of $\SO(3)\times \SO(3)$ fixing the vertices of the suspension and with principal isotropy $\SO(2)\times\SO(2)$. Consider $X$ the space of directions  at one of the vertexes of the suspension. This space is homeomorphic to $\Sp^2\times\Sp^2$, $\SO(3)\times \SO(3)$ acts transitively on $X$, and it is a $4$-dimensional Alexandrov space of curvature at least one. Thus $X$ is a homogeneous space, and by \cite[Theorem 7~I]{Berestovskii1989} $X$ is isometric to a homogeneous  Finsler manifold. Since $X$ is an Alexandrov space of curvature at least $1$, then  it is also an $\RCD(3,4)$-space, then $X$ is isometric to a homogeneous Riemannian manifold $(M,g)$. Moreover we have that $\sec (g)\geq 1$. Since we have non-trivial Killing vector fields by the action of $\SO(3)\times \SO(3)$, by \cite[Theorem 1]{HsiangKleiner1989} $M$ is homeomorphic to $\Sp^4$ or $\C P^2$, which is a contradiction. Thus $\Sp^2\times\Sp^2$ is not an $\SO(3)\times\SO(3)$-cohomogeneity one Alexandrov space.
\end{proof}

\begin{rmk}
When considering $G/H$ a homogeneous smooth manifold with finite fundamental group, using \th\ref{thm.rcd space from group diagram} for the group diagram $(G,H,G,G)$ we also get an $\RCD(K,N)$-space. From the proof if follows that  this case $K= 0$ and $N=\dim(G/H)+\dim(G)+1$, and the orbit space is isometric to $[-1-2/A,1+2/A]$ for an appropriate choice of $A$ and measure $f^{\dim(G/H)}\mathcal{L}$, where $f$ is a non-smooth function. Nonetheless, this space is equivariantly homeomorphic to the suspension of $G/H$ and thus to the space given by \th\ref{T: Suspensions are RCD spaces of coho 1}. That is, for a group diagram $(G,H,G,G)$ with $G/H$ having finite fundamental group, by \th\ref{thm.rcd space from group diagram} and \th\ref{T: Suspensions are RCD spaces of coho 1} we obtain two equivariantly homeomorphic spaces, but each one with a different $\RCD$-structure.
\end{rmk}

\bibliographystyle{siam}
\bibliography{References}
\end{document}